\RequirePackage{fix-cm}
\documentclass[smallextended,envcountsame,numbook]{svjour3}       
\smartqed  
\usepackage{graphicx}
\usepackage{soul}
\usepackage{xcolor}
\definecolor{lavender}{RGB}{230,230,250}
\usepackage{todonotes}

\usepackage{hyperref}
 \hypersetup{
     colorlinks=true,
     linkcolor=blue,
     filecolor=magenta,      
     urlcolor=cyan,
     citecolor=red
     }
\newenvironment{fig}
{\begin{figure}[hbt]}
{\end{figure}}
\newcommand{\bfig}{\begin{fig}}
\newcommand{\efig}{\end{fig}}
\renewcommand{\subset}{\subseteq}

\newcommand{\N}{{\mathbb{N}}}
\newcommand{\R}{{\mathbb{R}}}
\newcommand{\T}{{\mathbb{T}}}
\newcommand{\Z}{{\mathbb{Z}}}

\newcommand{\cA}{{\mathcal A}}
\newcommand{\cB}{{\mathcal B}}

\newcommand{\sA}{{\mathsf A}}

\newcommand{\sJ}{{\mathsf J}}

\newcommand{\sL}{{\mathsf L}}
\newcommand{\sM}{{\mathsf M}}
\newcommand{\sN}{{\mathsf N}}
\newcommand{\sO}{{\mathsf O}}

\newcommand{\sR}{{\mathsf R}}

\newcommand{\sfS}{{\mathsf \Sigma}}

\newcommand{\smin}{\smallsetminus}

\newcommand{\scrS}{\mathscr{S}}
\newcommand{\scrR}{{\mathscr R}}

\newcommand{\scrC}{{\mathscr C}}
\newcommand{\scrU}{{\mathscr U}}
\newcommand{\scrT}{{\mathscr{T}}}

\newcommand{\Bb}{{\mathscr{B}}}

\newcommand{\sARpair}{{\mathsf{ ARpair}}}
\newcommand{\sAR}{{\mathsf{ AR}}}
\newcommand{\sAtt}{{\mathsf{ Att}}}
\newcommand{\sRep}{{\mathsf{ Rep}}}
\newcommand{\sMorse}{{\mathsf{ Morse}}}

\newcommand{\sInvset}{{\mathsf{ Invset}}}

\newcommand{\sANbhd}{{\mathsf{ ANbhd}}}

\newcommand{\sRNbhd}{{\mathsf{ RNbhd}}}

\newcommand{\sub}{{\rm sub}}

\newcommand{\sSet}{{\mathsf{Set}}}
\newcommand{\SC}{{\mathsf{SC}}}
\newcommand{\RC}{{\mathsf{RC}}}

\newcommand{\Chi}{\raise .75ex\hbox{$\chi$}}

\newcommand{\pred}{{{\triangleleft}}}

\newcommand{\down}{\downarrow\!}
\newcommand{\up}{\uparrow\!}

\newcommand{\3}{{\mathbf{3}}}

\newcommand{\vgln}{\lb\begin{array}{rcl}}
\newcommand{\eindvgln}{\end{array}\right.}
\newcommand{\alphaOg}{\alpha_{\rm o}(\gamma_x^-)}

\newcommand{\cl}{{\rm cl}\,}

\newcommand{\Int}{{\rm int\,}}
\newcommand{\Inv}{\mbox{\rm Inv}}

\newcommand{\rmPhi}{\mathrm{\Phi}}
\newcommand{\rmOmega}{\mathrm{\Omega}}
\newcommand{\rmPsi}{\mathrm{\Psi}}
\newcommand{\rmGamma}{\mathrm{\Gamma}}

\newcommand{\rmXi}{\mathrm{\Xi}}
\newcommand{\rmTheta}{\mathrm{\Theta}}

\renewcommand{\star}{\mathop{\rm star }}

\newcommand{\id}{\mathop{\rm id }\nolimits}
 \usepackage{pb-xy}
\usepackage{extpfeil}  
 \usepackage{mathrsfs}
 \usepackage{empheq,yhmath}
\usepackage{pb-diagram}
 \usepackage{palatino}
 \usepackage{bm}
 \usepackage[cmtip,arrow]{xy}
 \usepackage{color}

\usepackage{amsfonts, amssymb, amsbsy, amsmath}
 \usepackage{blkarray}
 \usepackage{graphicx}
 \usepackage{float}
 \usepackage{epstopdf}
 \usepackage{algorithmic}
 \ifpdf
   \DeclareGraphicsExtensions{.eps,.pdf,.png,.jpg}
 \else
   \DeclareGraphicsExtensions{.eps}
 \fi
 \usepackage{mathrsfs}
 \usepackage{amscd,empheq} 
 \usepackage{yfonts,bm}

\usepackage{tikz}
\usepackage{tikz-cd}

\usepackage{aurical}
\usepackage[T1]{fontenc}

\usepackage{enumitem}
\setlist{noitemsep,topsep=1pt,parsep=1pt,partopsep=2pt}

\numberwithin{equation}{section}
 \journalname{ArXiv}
\begin{document}

\title{Priestley duality and\\  representations of recurrent dynamics 
}

\titlerunning{Priestley duality and  representations of recurrent dynamics}        



\author{William Kalies         \and
        Robert Vandervorst 
}


\institute{W.D. Kalies \at
              Department of Mathematics \\
              University of Toledo,  Toledo, USA\\
              \email{William.Kalies@UToledo.Edu}           
           \and
           R.C.A.M. Vandervorst \at
              Department of Mathematics\\
              VU University Amsterdam, The Netherlands\\
              \email{r.c.a.m.vander.vorst@vu.nl}
}

\date{Date: \today}

\maketitle

\begin{abstract}
For an arbitrary dynamical system there is a strong relationship between global dynamics and the order structure of an appropriately constructed Priestley space. 
This connection provides an order-theoretic framework for studying global dynamics.
In the classical setting, the chain recurrent set, introduced by C. Conley, \cite{Conley}, is an example of an ordered Stone space or Priestley space.
Priestley duality can be applied in the setting of dynamics on arbitrary topological spaces and  yields a notion of Hausdorff compactification of the (chain) recurrent set.

\end{abstract}

\tableofcontents

\begin{acknowledgements}
We like to thank A. Goldstein for some useful discussions in the early stages of this project.
The work of W.D.K. was partially supported by the Air Force Office of Scientific Research  under awards  FA9550-23-1-0011 and FA9550-23-1-0400.
\end{acknowledgements}

\section{Introduction}
\label{intro}

Recurrent versus nonrecurrent dynamics is a fundamental dichotomy for understanding the global structure of a dynamical system, \cite{poincare,Smale,Conley}.
Attractors and repellers are central to this dichotomy. They form a basis for robust decompositions as they capture asymptotic dynamics of regions in phase space. 
As such attractors and repellers  are the cornerstone for defining and characterizing recurrent dynamics. 

The action of a dynamical system on its phase space determines the orbits and the invariant sets as unions of orbits. An attractor is identified through an \emph{attracting neighborhood}, $U$, which captures the eventual forward images of its closure into its interior; the corresponding \emph{attractor} is the largest invariant set in $U$. This interplay between the action of the long-term dynamics and the topology of the phase space builds a global structure between invariant sets. The philosophy of this paper is to characterize this structure using order theory. 
A key feature of this approach is that recurrence can be characterized solely in terms of  attractors and repellers, so that once they are defined, the topological properties of the phase space play a role only in sharpening this characterization. 

The set of attractors, $\sAtt(\varphi)$, for a dynamical system, $\varphi$, forms a bounded, distributive lattice, \cite{LSoA1,LSoA2,LSoA3}. From an algebraic perspective, any bounded distributive lattice, $\sL$, is dual to ordered topological space. This dual is the spectrum, $\sfS \sL$, of the lattice,
which is a Priestley space, a poset equipped with a specific topology, the Priestley topology, that is compact, Hausdorff, and zero-dimensional.

Attractors and repellers form complementary attractor-repeller pairs. The collection of attractor-repeller pairs in a system, denoted as $\sARpair(\varphi)$, also forms a bounded, distributive lattice, see Section~\ref{ARpairsgencase}.
In the classical setting, every attractor has a unique dual repeller, but this need not be the case without some conditions on the topology of the phase space and properties of $\varphi$.

The \emph{recurrent set}, $\sR(\varphi)$, is defined in Section~\ref{recurrence} as the union of all complete orbits $\gamma_x$ such that
$\gamma_x\subset A\cup R$ for all attractor-repeller pairs $(A,R)$.
The \emph{recurrent components}, $\RC(\varphi)$, are the equivalence classes of the preorder on $\sR(\varphi)$ given by
$x\le x'$ if $x'\in A$ implies
 $x\in A$ for every $A \in \sAtt(\varphi)$, which induces an order on $\RC(\varphi)$.
There is a natural topology $\scrT_\sfS$ on $\RC(\varphi)$ and a map $\RC(\varphi)\to\sfS\sARpair(\varphi)$ that is a topological embedding and an order-embedding. This embedding provides a representation of the global dynamics within a compact order model, see Section~\ref{noncompact}.

We now describe the fundamental results of this paper.
When the phase space $X$ is compact and certain conditions are imposed on $\varphi$, such as being continuous and proper\footnote{A dynamical system $\varphi$ is proper if the maps $\varphi^t$ are \emph{closed} and their fibers $\varphi^{-t}(x)$ are \emph{compact} for all $x\in X$ and for all $t\ge 0$, cf.\ Sect.\ \ref{contdynnew}.}, 
the lattice of attractor-repeller pairs $\sARpair(\varphi)$ can be identified with the lattice of attractors $\sAtt(\varphi)$ so that $\sfS\sARpair(\varphi)$ is
 isomorphic to $\sfS\sAtt(\varphi)$.
In this setting, the recurrent set $\sR(\varphi)$ is defined as the 
intersection $\bigcap A \cup A^*$ over all attractors with $A^*$ 
 denoting the dual repeller of $A$. In addition to the topology $\scrT_\sfS$, the set of recurrent components $\RC(\varphi)$ can also be given the quotient topology $\scrT_\sim$ from the equivalence relation $\sim$ on 
 $\sR(\varphi)$ induced by the subspace topology from the phase space $X$. In this setting, $\scrT_\sfS$ and $\scrT_\sim$ are identical, and we obtain a homeomorphism and order isomorphism between $\RC(\varphi)$ and $\sfS\sAtt$.
 
\vskip.2cm
\noindent {\bf Theorem A. }(Thm.\ \ref{thm:homeo})
{\em
Let $\varphi$ be a continuous and proper dynamical system on
a compact space $(X,\scrT)$.
Then, the  space of   recurrent components $(\RC(\varphi),\scrT_\sim,\le)$ and the spectrum $(\sfS\sAtt(\varphi),\scrT_{\sfS\sAtt},\subseteq)$ are homeomorphic and order-isomorphic. 
}

\vskip.2cm

The conditions on $\varphi$ in Theorem A, ie.\ continuous and proper are satisfied for the following classes of dynamical systems:
\begin{description}
    \item[(i)]$\varphi$ is a continuous and invertible dynamical system on a compact  space $X$. Note that $\varphi$ is still proper without the \emph{compactness} requirement;
    \item[(ii)] $\varphi$ is a continuous  dynamical system on a compact, Hausdorff  space $X$;
    \item[(iii)] $\varphi$ is a continuous and closed dynamical system on a compact, $T_1$-space $X$.
\end{description}

The Priestley space $\sfS\sAtt(\varphi)$  consists of the prime ideals in $\sAtt(\varphi)$. A prime ideal is given by $I=h^{-1}(0)$, where $h\colon \sAtt(\varphi)\to {\bf 2}$ is a lattice homomorphism, with ${\bf 2}=\{0,1\}$, the two-point lattice $0<1$.
The recurrent components are retrieved from the Priestley space $\sfS\sAtt(\varphi)$ via the following formula:
\[
\rmPsi(I):= \left(\bigcap_{A\in I^c}A \right) \bigcap 
 \left( \bigcap_{A\in I}A^*\right)\in \RC(\varphi), 
\]
which represents the prime ideals as recurrent components in $X$.
Applying the above formula for ideals in a finite sublattice, $\sA\subset \sAtt(\varphi),$ yields  the Morse sets in a Morse representation $\sM\sA$, cf.\ \cite{LSoA3}. 
As a
 consequence of using lattice theory for describing  recurrence, $\RC(\varphi)$ can be characterized as a limit of Morse representations.
\vskip.2cm
\noindent {\bf Theorem B. }(Thm.\ \ref{inverselimthm})
{\em
Let $\varphi$ be a continuous and proper dynamical system on
a compact space $(X,\scrT)$. The space of recurrent components $(\RC(\varphi),\scrT_\sim,\le)$ is a profinite poset and 
\[
(\RC(\varphi),\scrT_\sim,\le) \cong \varprojlim \sM\sA,
\]
where the inverse limit is taken over the finite sublattices $\sA\subset \sAtt(\varphi)$.
}
\vskip.2cm
\noindent The inverse limit  includes the order structures as well as the topological structures.
Theorem~B  implies  that Morse representations are able to approximate all fine structure in $\RC(\varphi)$, even in the case that $\RC(\varphi)$ is uncountably infinite, cf.\ \cite{AHK1,AHK2}.
In the setting of an arbitrary topological space, the approach via attractors  or repellers does not yield an appropriate duality between attractors and repellers.  Instead, by using attractor-repeller pairs, we obtain the following result:
\vskip.2cm
\noindent {\bf Theorem C. }(Thm.\ \ref{cr7})
{\em
Let $\varphi$ be a dynamical system (not necessarily continuous) on 
a  topological space  $(X,\scrT)$. Then, the set of  recurrent components $(\RC(\varphi),\scrT_\sim)$ is a Hausdorff topological space, and there exists a continuous bijection from
$(\RC(\varphi),\scrT_\sim)$ to  $(\RC(\varphi),\scrT_\sfS)$. 
In particular, the map
$\rmPhi\colon (\RC(\varphi),\scrT_\sfS) \hookrightarrow (\sfS\sARpair(\varphi),\scrT_{\sfS\sARpair})$ is a topological order-embedding}.
\vskip.2cm
\noindent Theorem~C can be regarded as a compactification of the recurrent set $\RC(\varphi)$. In Section \ref{examplesofcomp}
 we provide some explicit examples exploring this Hausdorff compactification of the recurrent set.
Furthermore, as in Theorem~C, for an arbitrary space $X$ the ordered space of recurrent components can be defined via attractor-repeller pairs so that the topology and order on $\RC(\varphi)$ embed into the Priestley space $\sfS\sARpair(\varphi).$ 

Returning to the case of continuous and proper dynamical systems on compact topological spaces
 Sections \ref{ARpairsgencase}  and \ref{contdynnew}   establish 
 the lattice of attractors $\sAtt(\varphi)$ is established via the homomorphism $\Inv\colon \sANbhd(\varphi)\to\sAtt(\varphi)$, where $\sANbhd(\varphi)$ is the lattice of attracting neighborhoods for $\varphi$. 
The lattice of attracting neighborhoods $\sANbhd(\varphi)$  defines a preorder $(X,\le)$ in the same way as the recurrence preorder $(\sR(\varphi),\le)$. The equivalence classes are denoted by $\SC(\varphi)$ and are called the strong components of $\varphi$, see Section~\ref{attnbhds}. This yields the following cospan in the category of posets:
\begin{equation}
    \label{fundcospan156}
    X \xtwoheadrightarrow{~~\pi~~} \SC(\varphi) \xhookleftarrow{~~~\supseteq~~~} \RC(\varphi).
\end{equation}
In Section \ref{attnbhds} we explain how the cospan in \eqref{fundcospan156} defines a transitive relation $\scrR\subset X\times X$, and the set of points $x\in X$ with $(x,x)\in\scrR$ is the recurrent set $\sR(\varphi)$. 
In terms of Priestley spaces we obtain the commutative diagram:
\begin{equation*}
    \begin{diagram}
    \node{\beta X}\arrow{e,l}{\sfS\iota}\node{\sfS\sANbhd(\varphi)}\node{\sfS\sAtt(\varphi)}\arrow{w,l}{\sfS\omega}\\
    \node{X}\arrow{n,l,A}{i}\arrow{e,l}{\pi}\node{\SC(\varphi)}\arrow{n,r,A}{\rmXi}\node{\RC(\varphi)}\arrow{w,l}{\supseteq}\arrow{n,lr,<>}{\rmPhi}{\cong}
    \end{diagram}
\end{equation*}
where all maps are continuous and order-preserving.
The map $i$ is a topological embedding, $\rmXi$ is a continuous injection, and the map $\rmPhi$ is a homeomorphism.
There are several advantages to the above approach
as we now explain. 

In his monograph \cite{Conley}, Conley introduces the concept of chain recurrence for a flow on a compact, Hausdorff space $X$, 
 employing pseudo-orbit chains to establish an order relation 
$\scrC$ that captures essential characteristics of the dynamics of a flow 
$\varphi$. This order relation is referred to as the \emph{Conley relation}. He demonstrates in \cite[Assrt.\ 6.2.A]{Conley} that the set of chain recurrent points, derived from 
$\scrC$,
coincides with the set $\sR(\varphi)$.
The method of characterizing chain recurrence through chains was later more thoroughly developed by Akin, cf.\ \cite{Akin1,Akin2} and McGeehee, cf.\ \cite{Mcgehee}, aiming to simplify the dynamics through chain recurrence. Additionally, Hurley explores chain recurrence within noncompact metric spaces, cf.\ \cite{Hurley1,Hurley2}.
  Proposition \ref{prop:conley-order}  establishes that chain components align exactly with recurrent components $\RC(\varphi)$
with matching orders, when the former are well-defined. In Theorems \ref{comptoConleyrecrel} and \ref{comparetoCR}, we further illustrate that the relation 
$\scrR\subset X\times X$,
defined as by the cospan \eqref{fundcospan156}, is opposite to the Conley relation $\scrC$.
\vskip.2cm
\noindent {\bf Theorem D. }(Thm.\ \ref{comptoConleyrecrel} and Thm.\ \ref{comparetoCR})
{\em
If $\varphi$ is a continuous and invertible dynamical system on a compact, Hausdorff space $X$, then the relation $\scrR$ coincides with the opposite Conley relation $\scrC^{-1}$ (chain-recurrence relation).
}
\vskip.2cm

Applying the framework  discussed in this paper to, for example, trapping regions\footnote{See Section~\ref{contdynnew} and Remark~\ref{trappexist}.}, leads to a refined relation $\scrS\supseteq \scrR$
that offers a finer representation of the dynamics compared to the Conley relation, while the recurrent set and its order structure remain unchanged, see Remark~\ref{otherattnebhd}. 
Notably, our findings are applicable in arbitrary topological spaces $X$
without requiring a chain-based approach.

\section{Attractor-repeller pairs}
\label{ARpairsgencase}
We start with the definition of a dynamical system defined on a set $X$.
A one-parameter family $\varphi=\{\varphi^t\}_{t\in \T^+}$ of  maps  $\varphi^t\colon X \to X$ is a \emph{dynamical system}
if
\begin{description}
\item[(i)] $\varphi^0(x) = x$ for all $x\in X$;
\item[(ii)] $\varphi^t\bigl(\varphi^s(x)\bigr) = \varphi^{t+s}(x)$ for all $s,t\in \T^+$ and all $x\in X$,
\end{description}
where $\T$ is either $\Z$ or $\R$.
A dynamical system which can be defined so that the group property (ii) holds on all of $\T$  is  an \emph{invertible dynamical system}.
As stated in the introduction, a topology $\scrT$ on $X$ may be  chosen to distinguish orbits. 
A \emph{continuous} dynamical system $\varphi$ satisfies that additional requirement that the maps $\varphi^t\colon X\to X$ are continuous for all $t\in \T$.

\begin{remark}
For $t\in \R$, $\R^+$ one may assume that $\varphi(t,x):=\varphi^t(x)$ is a continuous function on $\R\times X$ or $\R^+\times X$.
The latter will not be assumed in this paper, and therefore a topology on $\T$ is not required.
\end{remark}

A set $S\subset X$ is \emph{forward invariant} if $\varphi^t(S)\subset S$ for all $t\ge 0$
and \emph{backward invariant} if $\varphi^{-t}(S)\subset S$ for all $t\ge 0$, where
 $\varphi^{-t}(S)$ denotes inverse
 image of $S$ under $\varphi^t$. A subset $S$ is \emph{invariant} if $\varphi^t(S) =S$ for all $t\ge 0$ and \emph{forward-backward invariant} if $\varphi^{-t}(S) =S$ for all $t\ge 0$. Furthermore define 
 $
 \Inv(U) = \bigcup\{S\subset U\mid S\text{~~invariant}\}
 $
 and  $\Inv^+(U) = \bigcup\{S\subset U\mid S\text{~~ forward invariant}\}$. See \cite{LSoA1} for more details.
All  notions of  invariance are independent of the topology on $X$. The invariant sets for $\varphi$ form a bounded, distributive lattice, denoted by $\sInvset(\varphi)$, with binary operations $S\cup S'$ and $S\wedge S' = \Inv(S\cap S')$\footnote{The wedge of two invariant sets $S$ and $S'$ can be interpreted as the union of all complete orbits $\gamma_x$ contained in both $S$ and $S'$. }
for all $S,S'\in \sInvset(\varphi)$, cf.\ \cite[Prop.\ 2.8]{LSoA1}. The latter is a complete lattice with 
\[
\bigcup_{i\in I} S_i \in \sInvset(\varphi),\quad \bigwedge_{i\in I} S_i = \Inv\Bigl( \bigcap_{i\in I} S_i \Bigr) \in \sInvset(\varphi),
\]
 for any family of invariant sets $\{S_i\}_{i\in I}$. The atoms of $\sInvset(\varphi)$ are the complete orbits $\gamma_x$ of $\varphi$. Every invariant set is a union of complete orbits.
 
 The following lemma gives a convenient criterion for the wedge of forward invariant sets and forward-backward invariant sets.
 \begin{lemma}
    \label{invofint}
    Let $S\in \sInvset^+(\varphi)$ and let $S'\in \sInvset^\pm(\varphi)$. Then,
    $\Inv(S\cap S') = \Inv(S)\cap S'$.
\end{lemma}
\begin{proof}
By \cite[Lem.\ 2.7]{LSoA1}, $\Inv(S\cap S') = \Inv\bigl(\Inv(S)\cap \Inv(S')\bigr)$, and by
\cite[Lem.\ 2.9]{LSoA1}, $\Inv(S)\cap S' \in \sInvset(\varphi)$.
Then,
\[
\begin{aligned}
    \Inv(S\cap S') &= \Inv\bigl(\Inv(S)\cap \Inv(S')\bigr)\subset \Inv\bigl(\Inv(S)\cap S'\bigr)\\
    &= \Inv(S)\cap S' \subset S\cap S',
\end{aligned}
\]
which implies, since $\Inv(S\cap S')$ is the maximal invariant set in $S\cap S'$,
that $\Inv(S\cap S')=\Inv(S)\cap S'$. 
    \qed
\end{proof}

\subsection{Attractors and repellers}
\label{attandrepnew}
Using a topology $\scrT$ on $X$, we define attracting neighborhoods and attractors.
\begin{definition}
\label{defnattfornoncomp}
A subset $U\subset X$ is a \emph{attracting neighborhood} if there exists a $\tau>0$ such that 
$\varphi^{t}(\cl U) \subset \Int U$  for all $t\ge \tau$.
A set $A\subset X$ is called an \emph{attractor} if there exists a attracting neighborhood $U$ such that $A=\Inv(\cl U)$.
\end{definition}

From the definition of attractor we obtain the following list of properties. We emphasize that at this point we do not require continuity for $\varphi$, nor any conditions on the topological space $X$. At a later stage additional requirements yield additional properties.

\begin{proposition}
\label{listofattpropsa}
The following properties hold for attractors and attracting neighborhoods:
\begin{description}
\item[(i)] An attractor satisfies $A=\Inv(\cl U)=\Inv(U) \subset \Int U$;
\item [(ii)] The attracting neighborhoods, 
denoted  $\sANbhd(\varphi)$,  have the structure of a bounded, distributive
lattice, i.e. $U,U'\in \sANbhd(\varphi)$, then $U\cup U', U\cap U'\in \sANbhd(\varphi)$;
\item [(iii)] The attractors are  denoted by $\sAtt(\varphi)$ and form a bounded, distributive lattice
with respect to the binary operations: 
\[
A\cup A' \quad \text{and}\quad A\wedge A':= \Inv(A\cap A'),~~ A,A'\in \sAtt(\varphi);
\]
\item [(iv)] $\Inv\colon \sANbhd(\varphi) \to \sAtt(\varphi)$, defined by $U\mapsto \Inv(U)$ is a surjective lattice homomorphism.
\end{description}
\end{proposition}

\begin{proof}
\noindent{\bf(i)} 
By definition $\Inv(U) \subset \Inv(\cl U)=A\subset \cl U$. Then by invariance, $A=\varphi^t(A)\subset \varphi^t(\cl U)\subset \Int U\subset U$ for all $t\ge \tau$, and therefore $A\subset \Inv(U)$, which shows that $\Inv(\cl U)=\Inv(U)$.

\noindent{\bf(ii)} Let $U,U'\in \sANbhd(\varphi)$. By the definition of attracting neighborhood we can choose
a common $\tau''=\max\{\tau,\tau'\}$. 
This yields,
$\varphi^{t}(\cl (U\cup U') )= \varphi^{t}(\cl U \cup \cl U') =  \varphi^{t}(\cl U)\cup \varphi^{t}(\cl U') \subset \Int U\cup \Int U'
\subset \Int(U\cup U')$ and 
$\varphi^{t}(\cl (U\cap U')) \subset \varphi^{t}(\cl U \cap \cl U') \subset \varphi^{t}(\cl U) \cap \varphi^{t}(\cl U') \subset \Int U\cap \Int U' = \Int(U\cap U')$ for all $t\ge \tau''$,
 which shows that intersection and union are binary operations on $\sANbhd(\varphi)$ and $\sANbhd(\varphi) \rightarrowtail \sSet(X)$ is a lattice embedding.

\noindent{\bf(iii)} 
By \cite[Lem.\ 2.7]{LSoA1} and the definition of $\wedge$ it follows that
 $A\wedge A' = \Inv(A\cap A') \subset \Inv(U\cap U')= \Inv(U)\wedge \Inv(U') =A\wedge A'$, which implies that $A\wedge A' = \Inv(U\cap U')\subset \Int U\cap \Int U'=\Int(U\cap U')$ so that $A\wedge A'$ is an attractor.

As for the union we argue as follows.
For every attracting neighborhood $U$ we can construct a forward invariant set 
 $V\supseteq U$ 
 such that $\Inv(V)=\Inv(U)$.
 Define $V: = \bigcup_{s\in [0,\tau]}\varphi^s(U)\supseteq U$. 
Then, 
$
\varphi^t(V) = \bigcup_{s\in [t,t+\tau]}\varphi^s(U)
\subset \bigcup_{s\in [\tau,\infty)}\varphi^s(U)\\\subset U\subset V
$
for all $t\ge \tau$.
For $t\in [0,\tau]$ we have,
\[
\begin{aligned}
\varphi^t(V) &= \bigcup_{s\in [t,t+\tau]}\varphi^s(U) = \bigcup_{s\in [t,\tau]}\varphi^s(U) \cup \bigcup_{s\in [\tau,t+\tau]}\varphi^s(U)\\
 &\subset \bigcup_{s\in [t,\tau]}\varphi^s(U) \cup U \subset  V \cup U = V,
\end{aligned}
\]
which establishes 
forward invariance of $V$. If $S=\Inv(V)$, then $S=\varphi^\tau(S)\subset \varphi^\tau(V)\subset U$, which proves that $\Inv(V)=\Inv(U)\subset \Int U\subset \Int V$.
Choose $\tau>0$ for both $U$ and $U'$.
Combined with \cite[Lem.\ 2.7]{LSoA1} and the fact that $\bigcup_{s\in [0,\tau]}\varphi^s(U\cup U')= V\cup V'$, we conclude that  
\[
\begin{aligned}
\Inv(U) \cup \Inv(U') &\subset \Inv(U\cup U')=\Inv(V\cup V')\\
&= \Inv(V)\cup \Inv(V') = \Inv(U) \cup \Inv(U'),
\end{aligned}
\]
 which proves that $\Inv(U\cup U') = \Inv(U)\cup \Inv(U')$.
Using the latter we obtain $A\cup A' = \Inv(U)\cup \Inv(U') = \Inv(U\cup U')\subset \Int U\cup \Int U'\subset \Int(U\cup U')$ and $A\cup A'$ is an attractor.

Consequently, the set $\sAtt(\varphi)$ is closed with respect to $\wedge$ and $\cup$, i.e.~$A\cup A'\in \sAtt(\varphi)$ and $A\wedge A':= \Inv(A\cap A')\in \sAtt(\varphi)$ for all $A,A'\in \sAtt(\varphi)$. Hence, $\sAtt(\varphi)$ is a bounded sublattice of $\sInvset(\varphi)$, cf.\ \cite[Prop.\ 2.8]{LSoA1}, which proves that $\sAtt(\varphi)$ is a bounded, distributive lattice.

\noindent{\bf(iv)} 
The identities in the proof of (iii) yield that $\Inv(U\cap U') = \Inv(U)\wedge \Inv(U')$ and $\Inv(U\cup U') = \Inv(U) \cup \Inv(U')$, which establishes 
$\Inv$ as a lattice homomorphism from $\sANbhd(\varphi)$ to $\sAtt(\varphi)$.
\qed
\end{proof}

Proposition \ref{listofattpropsa}(i)  justifies the definition   $A=\Inv(U)$ for attractor, which will be used as main definition. 
\begin{remark}
    \label{intandclforattnbhd}
    If $U$ is an attracting neighborhood, then so are $\cl U$ and $\Int U$.
\end{remark}
\begin{remark}
    \label{forwinvisolnbhd}
    The neighborhood $V$ constructed in (iii) in the above proof is a \emph{forward invariant, isolating neighborhood} $V$ for $A$, i.e. $\varphi^t(V)\subset V$ for all $t\ge 0$ and $A=\Inv(V)\subset \Int V$. 
    Stronger isolation properties are obtained when $\varphi$ is continuous: $\Inv(\cl V)\subset \Int V$. However, without further assumptions on $X$ such as compactness, $V$ need not be an attracting neighborhood, cf.\ Remark~\ref{trappexist}.
\end{remark}
\begin{definition}
\label{defnrep2}
A subset $U\subset X$ is a \emph{repelling neighborhood} if there exists a $\tau>0$ such that 
$\varphi^{-t}(\cl U) \subset \Int U$  for all $t\ge \tau$.
A set $R\subset X$ is called a \emph{repeller} if there exists a repelling neighborhood $U$ such that $R=\Inv^+(\cl U)$.
\end{definition}

As for attractors, various properties of repellers can be derived from the properties of $\Inv^+$ which are summarized in the following proposition. 

\begin{proposition}
\label{listofreppropsa}
The following properties hold for repellers and repelling neighborhoods:
\begin{description}
\item[(i)] A repeller satisfies $R=\Inv^+(\cl U)=\Inv^+(U) \subset \Int U$;
\item[(ii)] The repelling neighborhoods, 
denoted   $\sRNbhd(\varphi)$,  have the structure of a bounded, distributive
lattice, i.e. $U,U'\in \sRNbhd(\varphi)$, then $U\cup U', U\cap U'\in \sRNbhd(\varphi)$;
\item[(iii)] A repeller $R$ is a forward-backward invariant set and invariant if $\varphi$ is surjective;
\item[(iv)] The repellers are denoted by $\sRep(\varphi)$ which has the structure of a bounded, distributive lattice 
with respect to the binary operations: 
\[
R\cup R'\quad \hbox{and}\quad R\cap R',\quad R,R'\in \sRep(\varphi);
\]
\item[(v)] $\Inv^+\colon\sRNbhd(\varphi) \to \sRep(\varphi)$, defined by $U\mapsto \alpha(U)$ is
surjevctive lattice homomorphism;
\end{description}
\end{proposition}

\begin{proof}
\noindent{\bf(i)} By definition $\Inv^+(U) \subset \Inv^+(\cl U)=R\subset \cl U$. Then, by forward invariance, $\varphi^t(R)\subset R$ for all $t\ge 0$ and thus $R\subset \varphi^{-t}(R)$ for all $t\ge 0$. This implies
$R \subset\varphi^{-t}(R)\subset \varphi^{-t}(\cl U)\subset \Int U\subset U$ for all $t\ge \tau$, and therefore $R\subset \Inv^+(U)$, which shows that $\Inv^+(\cl U)=\Inv^+(U)\subset \Int U$.

\noindent{\bf(ii)} 
Let $U,U'\in \sRNbhd(\varphi)$. By the definition of repelling neighborhood we can choose
a common $\tau''=\max\{\tau,\tau'\}$. 
This yields,
$\varphi^{-t}(\cl (U\cup U') )= \varphi^{-t}(\cl U \cup \cl U') =  \varphi^{-t}(\cl U)\cup \varphi^{-t}(\cl U') \subset \Int U\cup \Int U'
\subset \Int(U\cup U')$ and 
$\varphi^{-t}(\cl (U\cap U')) \subset \varphi^{-t}(\cl U \cap \cl U') = \varphi^{-t}(\cl U) \cap \varphi^{-t}(\cl U') \subset \Int U\cap \Int U' = \Int(U\cap U')$, for all $t\ge \tau''$,
 which shows that intersection and union are binary operations on $\sRNbhd(\varphi)$ and $\sRNbhd(\varphi) \rightarrowtail \sSet(X)$ is a lattice embedding.

\noindent{\bf(iii)}
Similar to the proof of Proposition \ref{listofattpropsa}(iii), we construct a backward invariant set $V\supseteq U$ such that $\Inv^+(V)=\Inv^+(U)$.
Define  $V: = \bigcup_{s\in [0,\tau]}\varphi^{-s}(U)\supseteq U$. We obtain that $V$ is backward invariant and $\varphi^{-\tau}(V)\subset U$. If $S=\Inv^+(V)$, then $\varphi^{t}(S)\subset S$ for all $t\ge 0$ and in particular $S\subset\varphi^{-\tau}(S)$.
This implies that
$S\subset\varphi^{-\tau}(S)\subset \varphi^{-\tau}(V)\subset U$, which proves that $R=\Inv^+(U)=\Inv^+(V)\subset \Int U\subset \Int V$.
By \cite[Lem.\ 2.10]{LSoA1} we have that $R$ is forward-backward invariant.

\noindent{\bf(iv)}
If we adjust the proof of  \cite[Lem.\ 2.7]{LSoA1} for $\Inv^+$, then $\Inv^+(U\cap U') = \Inv^+(U) \cap \Inv^+(U')$ and therefore
$\Inv^+\colon \sRNbhd(\varphi) \to \sRep(\varphi)$ preserves $\cap$ if we equip $\sRep(\varphi)$ with 
$\wedge =\cap$.
The proof of \cite[Lem.\ 2.7]{LSoA1} also yields $\Inv^+(U\cup U') = \Inv^+(U) \cup \Inv^+(U')$, provided $U,U'$ are backward invariant.
If $U,U'$ are repelling neighborhoods, then we can choose a common $\tau>0$ and backward invariant sets $V\supseteq U$ and $V'\supseteq U'$ such that $\Inv^+(V)=\Inv^+(U)$, $\Inv^+(V') = \Inv^+(U'),$ and $\Inv^+(V\cup V') = \Inv^+(U\cup U')$. 
As before,
\[
\begin{aligned}
\Inv^+(U) \cup \Inv^+(U') &\subset \Inv^+(U\cup U')=\Inv^+(V\cup V')\\
&= \Inv^+(V)\cup \Inv^+(V') = \Inv^+(U) \cup \Inv^+(U'),
\end{aligned}
\]
This proves proves that $\sRep(\varphi)$ equipped with $\cap$ and $\cup$ is a bounded, distributive lattice.

\noindent{\bf (v)}
If $R=\Inv^+(U)$ and $R'=\Inv^+(U')$, then $U\cup U'\mapsto R\cup R'$ and $U\cap U'\mapsto R\cap R'$ showing that $\Inv^+$ is a surjective lattice homomorphism.
\qed
\end{proof}

\begin{remark}
    \label{backwinvisolnbhd}
    Given a repelling neighborhood $U$, an analogous construction
     provides a repelling neighborhood $V\supseteq U$ which is a \emph{backward invariant, isolating neighborhood} for $R$, i.e.\ $\varphi^{-t}(V)\subset V$ for all $t\ge 0$ and $R=\Inv^+(V)\subset \Int V.$ The existence of repelling neighborhoods that are backward invariant is discussed in Remark \ref{trappexist}. 
\end{remark}

\subsection{Duality}
\label{dualityforAR}
We now discuss a fundamental duality between attracting and repelling neighborhoods, which also yields a duality between attractors and repellers.
\begin{lemma}
\label{involattrep123}
The map $U\mapsto U^c$ defines an involutive lattice anti-isomorphism between $\sANbhd(\varphi)$ and $\sRNbhd(\varphi)$.
\end{lemma}

\begin{proof}
Let $U\in \sANbhd(\varphi)$, i.e. $\varphi^t(\cl U) \subset \Int U$ for all $t\ge \tau>0$.
Then, $\cl U^c = (\Int U)^c \subset [\varphi^t(\cl U)]^c$. By taking inverse images we obtain
\[
\varphi^{-t}(\cl U^c) \subset \varphi^{-t}\bigl(\varphi^t(\cl U)^c \bigr) = \bigl[ \varphi^{-t}\bigl(\varphi^t(\cl U) \bigr)\bigr]^c
\subset (\cl U)^c = \Int U^c~~\forall t\ge \tau,
\]
which establishes $U^c\in \sRNbhd(\varphi)$. Similarly, let $U\in \sRNbhd(\varphi)$,  i.e. 
$\varphi^{-t}(\cl U) \subset \Int U$ for all $t\ge \tau>0$.
Then,
\[
\cl U^c = (\Int U)^c \subset [\varphi^{-t}(\cl U)]^c = \varphi^{-t}\bigl((\cl U)^c\bigr) =\varphi^{-t}(\Int U^c).
\]
By taking the $\varphi^t$ image we obtain: $\varphi^t(\cl U^c) \subset \varphi^{t}\bigl(\varphi^{-t}(\Int U^c) \bigr)\subset \Int U^c$ for all $t\ge \tau$, which shows that $U^c\in\sANbhd(\varphi)$ and that the map 
 $U\mapsto U^c$ defines an involutive bijection. The anti-isomorphism is given by the De Morgan laws for $\cap $ and $\cup$.
\qed
\end{proof}

As stated above, duality between attractors and repellers is not guaranteed if $X$ is not compact. Nontrivial attracting and repelling neighborhoods may have empty attractors and repellers. This problem is solved by treating attractors and repellers as pairs. 
\begin{definition}
\label{ARpairdefn234}
A pair $P=(A,R)$ of subsets in $X$ is called an \emph{attractor-repeller pair} if there exist an attracting neighborhood $U$ such that $A=\Inv(U)$ and $R=\Inv^+(U^c)$. The attractor-repeller pairs are denoted by $\sARpair(\varphi)$.
\end{definition}

\begin{remark}
Attractor-repeller pairs are equivalently defined via repelling neighborhoods $U$ via $A=\Inv(U^c)$ and $R=\Inv^+(U)$.
\end{remark}

Attractor-repeller pairs also allow a lattice structure which is defined as follows. Let $P=(A,R)$, $P'=(A',R')$ be attractor-repeller pairs. Then,
\begin{equation}
\label{ARpairbins}
P\vee P' := \bigl( A\cup A', R\cap R'\bigr),\quad\quad P\wedge P' := \bigl( A\wedge A', R\cup R'\bigr).
\end{equation}
It follows from the lattice structures of $\sAtt(\varphi)$ and $\sRep(\varphi)$ that $\sARpair(\varphi)$ with the binary operations $\vee$ and $\wedge$ defined above is a bounded, distributive lattice.
In particular, $P\le P'$ if and only if $A\subset A'$ and $R'\subset R$. 
Define the homomorphisms $\varpi\colon \sANbhd(\varphi) \to \sARpair(\varphi)$ via $U\mapsto \bigl(\Inv(U),\Inv^+(U^c) \bigr)$ and $\varpi^*\colon \sRNbhd(\varphi) \to \sARpair(\varphi)$ via $U\mapsto \bigl(\Inv(U^c),\Inv^+(U) \bigr)$.
We obtain the following commutative diagram:
\begin{equation}
        \label{dualitydiagno34}
    \begin{diagram}\dgARROWLENGTH=4.5em
    \node{\sANbhd(\varphi)}\arrow[2]{e,lr,<>}{^c}{\cong}\arrow{se,r}{\varpi}\node{}\node{\sRNbhd(\varphi)}\arrow{sw,r}{\varpi^*}\\
    \node{}\node{\sARpair(\varphi)} 
    \end{diagram}
        \end{equation}
        cf.\ Diag.\ \eqref{dualitydiagno12}.
If we choose either attracting or repelling neighborhoods we have 
the span:
  \begin{equation}
     \label{fundmaps12}
\sSet(X) \xleftarrow{~~~\iota~~~}\sANbhd(\varphi)\xtwoheadrightarrow{\varpi}\sARpair(\varphi),
\end{equation}
where $\iota$ is inclusion and where $\sSet(X)$ is the Boolean algebra of subsets of $X$.
For  an attractor-repeller pair $P=(A,R)$  complete orbits $\gamma_x$ satisfy the following `asymptotics'.
 \begin{lemma}
 \label{asympattrep}
Let  $P=(A,R)$
be an attractor-repeller pair with $A=\Inv(U)$ and $R=\Inv^+(U^c)$ for some $U\in \sANbhd(\varphi)$. Let $x\in X\smin (A\cup R)$. For a complete orbit $\gamma_x$ there exists a $\tau>0$ such that $\gamma_x(t)\subset U$ and $\gamma_x(-t)\in U^c$ for all $t\ge \tau$.
 \end{lemma}

\begin{proof}
 For $x\in X\smin (A\cup R)$ there are two options: (i) $x\in U$ or (ii) $x\in U^c$.

\noindent{\bf(i)} 
By the definition of attracting neighborhood $\gamma_x(t)\subset U$ for all $t\ge \tau$. Since $A=\Inv(U)$, and $x\in U\smin A$, the orbit  $\gamma_x$ cannot be contained in $U$.
This implies that there exists a $\sigma<\tau$ such that $y=\gamma_x(\sigma)\in U^c$. By definition $\varphi^t\bigl(\gamma_x(\sigma-t)\bigr) = \gamma_x(\sigma)$. 
Since $U^c$ is a repelling neighborhood, $\gamma_x(\sigma-t)\in
\varphi^{-t}(\gamma_x(\sigma))\subset U^c$
for all $t\ge \tau$. 
If $t'=t-\sigma$, then $\gamma_x(-t')\in U^c$ for all $t'\ge \tau-\sigma$.
Without loss of generality we use $t$ instead of $t'$ and choose $\tau$ such that both inclusions hold.

\noindent{\bf(ii)}
Similarly, we can choose a $\tau>0$ such that $\gamma_x(-t)\in U^c$ for all $t\ge \tau$, which follows from the definition of repelling neighborhood.
Since $R=\Inv^+(U^c)$, and $x\in U^c\smin R$, the forward orbit $\gamma_x^+$ cannot be contained in $U^c$. Therefore,
there exists a $\sigma>0$ such that  $\gamma_x(\sigma)\subset U$, and
 $\gamma_x(t+\sigma)=\varphi^t\bigl(\gamma_x(\sigma)\bigr)\in U$
for all $t\ge \tau$. Let $t'=t+\sigma$. Then, $\gamma_x(t')\in U$ for $t'\ge \tau+\sigma$. As for (i), without loss of generality we use $t$ instead of $t'$ and choose $\tau$ such that both inclusions hold.
\qed
 \end{proof}

\begin{remark}
Define $\gamma_x^\tau = \bigcup_{t\ge \tau}\gamma_x(t)$ and $\gamma_x^{-\tau} = \bigcup_{t\ge \tau}\gamma_x(-t)$, with $\tau>0$ and $P$ as in the above lemma.
In the case that $\Inv(\cl\gamma_x^\tau)\neq \varnothing$, then 
$\Inv(\cl\gamma_x^\tau)\subset A\neq \varnothing$. The same applies to 
$\Inv^+(\cl\gamma_x^{-\tau}) \neq \varnothing$, in which case 
$\Inv^+(\cl\gamma_x^{-\tau}) \subset  R\neq \varnothing$.
\end{remark}
 
\section{Continuous dynamics on compact spaces}
\label{contdynnew}
In this section we explore extra structure when additional properties on $\varphi$, such as continuity, and/or topological properties on $X$ are assumed.

\begin{example}
\label{invnotomega}
Consider the interval $X=[0,\infty)\subset \R$ equipped with the co-finite topology $\scrT_{\rm cf}$ and map $f(x)=x+1$. Then, $f\colon X\to X$ is a continuous map on a compact space. The map $f$ is not closed since $f(X)=[1,\infty)$, which
 is not closed. The set $X$ is a repelling neighborhood with $\Inv^+(X)=X=\alpha(X)$.
Similarly, $X$ acts as an attracting neighborhood and $\Inv(X)=\varnothing \neq \omega(X) = X$. The open set $U=(0,\infty)$ is a precompact, attracting neighborhood with $\Inv(U) =\varnothing \neq \omega(U) = X$. 
The attractor-repeller pair associated to $U$ is the pair $(\varnothing,\varnothing)$. The lattice of attractor-repeller pairs consists of two elements $(\varnothing,X) < (\varnothing,\varnothing)$.
If we consider the standard metric topology $\scrT_{|\cdot|}$ on $X$, then
$f$ defines a continuous and proper dynamical system. With the metric topology $U$ is still an attracting neighborhood and $\Inv(U) = \omega(U) = \varnothing$.
\end{example}

\subsection{Continuous and proper systems}
\label{contcasearbX}
Example \ref{invnotomega} suggests that, under the assumption of continuity, a repeller can be characterized as the alpha-limit set of a repelling neighborhood, as proved in the following lemma. However, the analogous characterization is \emph{not} true  for attracting neighborhoods and attractors without additional conditions.

\begin{proposition}
\label{linktoalpha}
Let $\varphi$ be a continuous system, and
let $U\in \sRNbhd(\varphi)$ be a repelling neighborhood. Then, 
\[
R=\Inv^+(U)=\Inv^+(\cl U) =\alpha(\cl U)=\alpha(U)\subset \Int U,
\]
 and the repeller $R$ is a closed, forward-backward invariant set.
\end{proposition}

\begin{proof}
Let $U\in \sRNbhd(\varphi)$.
For  $s\ge t\ge 2\tau$ we have  $\varphi^{-s}(U) = \varphi^{-t+\tau}\bigl(\varphi^{t-s-\tau}(U) \bigr)$.
Since $s-t+\tau\ge \tau$, we have 
$
\varphi^{-s}(\cl U) = \varphi^{-t+\tau}\bigl(\varphi^{t-s-\tau}(\cl U) \bigr)\subset \varphi^{-t+\tau}(\cl U)
$
for all $s\ge t\ge 2\tau$.
Consequently, using that $t\ge 2\tau$ and continuity,
\[
\begin{aligned}
\alpha(\cl U) &= \bigcap_{t\ge 2\tau} \cl \bigcup_{s\ge t} \varphi^{-s}(\cl U)
\subset \bigcap_{t\ge 2\tau} \cl \varphi^{-t+\tau}(\cl U)
\subset \bigcap_{t\ge 2\tau} \varphi^{-t+\tau}(\cl U)\\
&\subset \bigcap_{t'\ge \tau} \varphi^{-t'}(\cl U)\subset \Int U\subset U.
\end{aligned}
\]
Using Proposition \ref{lem:props-al-om14}(vii), 
$\alpha(\cl U) \subset \alpha(\alpha(\cl U)) \subset \alpha(U)$.
Moreover, $\alpha(U) \subset \alpha(\cl U)$, which proves that $\alpha(\cl U) = \alpha(U)$.
Since
$ \alpha(U) \subset \Int U\subset U$, it follows that
$\alpha\bigl( \alpha(U)\bigr) \subset \alpha(U)\subset \alpha\bigl( \alpha(U)\bigr)$,  
which proves  $\alpha\bigl( \alpha(U)\bigr) = \alpha(U)$. 
Since $\alpha(U)\subset U$, the result follows from Proposition~\ref{lem:props-al-om14}(viii).
\qed
\end{proof}

\begin{remark}
By the properties of alpha-limit sets, if $U$ is precompact, then $U\neq \varnothing$ implies that $R=\Inv^+(U)\neq \varnothing$, cf.\ Proposition \ref{lem:props-al-om14}(ii).
In particular, if $X$ is compact, then $R=\Inv^+(U)$ is nonempty whenever $U$ is nonempty. In contrast,
Example \ref{invnotomega} provides an example of a nonempty, attracting neighborhood for which the corresponding attractor is the empty set even though $X$ is compact.
\end{remark}

In order to relate attractors to the omega-limit sets we consider continuous and \emph{proper} dynamical systems. 
\begin{definition}
    \label{properdynsys}
A dynamical system is  \emph{proper}\footnote{For a detailed discussion on proper maps see \cite[Ch.\ 1, Sect.\ 10]{bourbaki}.} if: 
\begin{description}
    \item[(i)] $\varphi$ has compact fibers, i.e.~$\varphi^{-t}(x)$ is compact for all $x\in X$ and for all $t>0$;
    \item[(ii)] $\varphi$ is closed,  i.e. $\cl\varphi^t(U)=\varphi^t(\cl U)$ for all $U\subset X$ and for all $t>0$.
\end{description}
\end{definition}

\begin{proposition}
    \label{contclosedom}
    Let $\varphi$ be a continuous and proper system. If $U\in \sANbhd(\varphi)$, then
\[
A=\Inv(U) = \omega(U) \subset \Int U,
\]
and the attractor $A$ is a closed, invariant set. 
\end{proposition}
\begin{proof}
By definition, an attracting neighborhood is \emph{eventually} forward invariant, i.e. $\varphi^t(U)\subset U$ for all $t\ge \tau$ for some $\tau>0$.
By Propositions \ref{listofattpropsa}(i) and  \ref{invofomega123}, $A=\Inv(U)=\Inv(\cl U) = \omega(U)\subset \Int U$.
\qed
\end{proof}

Example \ref{ordertopnoninv} provides an example that without properness Proposition \ref{contclosedom} is not true in general.
Here we list some settings which imply that a dynamical system $\varphi$ on $X$ is proper, cf.\ Remark~\ref{speccaseforom}.
\begin{description}
    \item[(i)]$\varphi$ is a continuous and invertible. Note that \emph{no} conditions on the topology of $X$ are necessary;
    \item[(ii)] $\varphi$ is continuous with $X$ compact and Hausdorff;
    \item[(iii)] $\varphi$ is a continuous and closed with $X$ compact and $T_1$.
\end{description}
  
\noindent For continuous and invertible systems we also have additional invariance properties on repellers due to the time-reversal symmetry.
\begin{proposition}
\label{linktoboth}
Let $\varphi$ be a continuous and invertible dynamical system.
Then, 
both attractors and repellers are closed, invariant sets and are given by omega-limit and alpha-limit sets respectively.
\end{proposition}
\begin{proof}
For a continuous and invertible system the maps $\varphi^t\colon X\to X$ are homeomorphisms, and therefore closed maps. Since $\varphi^{-t}(x)=\{y\}$ are singleton sets they are compact. This proves that $\varphi$ is proper. By Proposition \ref{contclosedom} attractors are closed, invariant sets given by the omega limit set. By time reversal $t\mapsto -t$ the same applies to repellers.
    \qed
\end{proof}

\begin{remark}
Given a nonempty attracting or repelling neighborhood in $X$, the corresponding omega-limit or alpha-limit set as in Proposition \ref{linktoboth} is not neccessarily  nonempty. 
If $X$ is a compact, then omega-limit and alpha-limit sets of nonempty attracting and repelling neighborhoods are nonempty.
\end{remark}

\begin{remark}
    For a continuous and closed  systems, attractors are closed, invariant sets.
\end{remark}

\subsection{Compact spaces and duality}
\label{compspcase}
If $X$ is compact, there is stronger duality between attractors and repellers. 
\begin{lemma}
\label{nested45}
Let $U_0 \supset U_1 \supset U_2 \supset \cdots$ be a nested sequence of nonempty, compact subsets of $X$. Suppose $\bigcap_{n\ge 0} U_n \subset \Int V$. Then, there exists $n_*\ge 0$ such that $U_n \subset \Int V$ for all $n\ge n_*$.
\end{lemma}

\begin{proof}
Suppose not, then $U_n \cap (\Int V)^c \not =  \varnothing$ for all $n$.
By definition $(\Int V)^c = \cl V^c$ and  is therefore a compact set in $X$.\footnote{The statement of the lemma also holds in arbitrary topological spaces since $\cl V^c$ is closed and thus $U_n\cap \cl V^c$ is compact.} This implies that $W_n:=U_n\cap \cl V^c$ is compact and nonempty for all $n\ge 0$.
By the Cantor Intersection Theorem $\bigcap_n U_n \not = \varnothing$. Since the sets $W_n$ also form a nested sequence of of nonempty, compact subsets of $X$,  we also have $\bigcap_n W_n \not = \varnothing$.
Now,
\[
\varnothing \not =  \bigcap_n W_n  \subset  \bigcap_n U_n \subset \Int V.
\]
On the other hand  $\varnothing\not=\bigcap_n W_n\subset \cl V^c = (\Int V)^c$, a contradiction.
\qed
\end{proof}

\begin{lemma}
\label{att1234}
Let $\varphi$ be a continuous and proper dynamical system with $X$ compact. 
Let $U,U'\in \sANbhd(\varphi)$ with $\omega(U) = \omega(U')$. Then, $\alpha(U^c) =\alpha(U'^c)$.
Similarly, if $U,U'\in \sRNbhd(\varphi)$ with $\alpha(U)=\alpha(U')$, then
$\omega(U^c) = \omega(U'^c)$.
\end{lemma}
\begin{proof}
Let $U,U'\in \sANbhd(\varphi)$.
Since $\omega( U) = \omega(U')$, Proposition \ref{contclosedom} implies that $\omega(U')\subset \Int U$ and $\omega(U)\subset \Int U'$. 
Consider the nested family of nonempty, compact sets
$U_t=\cl\rmGamma^+_t(U)=\cl\bigl(\bigcup_{s\ge t}\phi^s(U)\bigr)$. 
Since  $\omega(U) =\bigcap_{t\ge 0} U_t$ is contained in the open
set $\Int U'$, Lemma \ref{nested45} yields the existence of a $\tau>0$ such that $U_t \subset \Int U'$ for all $t\ge \tau$. Therefore,  $\varphi^t(\cl U) \subset \cl\phi^t(U) \subset U_t \subset \Int U'$ for all $t\ge \tau$. 
Similarly, $\varphi^t(\cl U')\subset \Int U$ for all $t\ge \tau$, for a common choice of $\tau>0$.
This implies that
$\varphi^{-\tau}(\cl U^c)\subset \Int U'^c$   and 
$\varphi^{-\tau}(\cl U'^c)\subset \Int U^c$, cf.~Lemma \ref{involattrep123}. 
Therefore, since $U^c,U'^c\in \sRNbhd(\varphi)$,
\[
\alpha(U^c)=\alpha(\cl U^c) = \alpha\bigl(\varphi^{-\tau}(\cl U^c)\bigr) \subset \alpha(\Int U'^c) \subset \alpha(\cl U'^c)=\alpha(U'^c),
\]
and similarly $\alpha( U'^c) \subset \alpha(U^c)$.
Summarizing, we established that $\alpha( U^c) = \alpha( U'^c)$.
For repelling neighborhoods the proof is similar using the fact that 
$\alpha(U)=\alpha(U')$ if and only if $\alpha(\cl U)=\alpha(\cl U')$.
\qed
\end{proof}

\begin{remark}
    Note that for a continuous and proper system on a compact  space, an attracting neighborhood is equivalently characterized by $\omega(U)\subset \Int U$. Indeed, by Proposition \ref{contclosedom}, if $U$ is an attracting neighborhood then $\omega(U)\subset \Int U$. Conversely, as in the proof of
    Lemma \ref{att1234}, the latter condition yields $\varphi^t(\cl U)\subset \Int U$ for all $t\ge \tau$ for some $\tau>0$.
    Similarly, repelling neighborhoods are characterized by $\alpha(U)\subset \Int U$.
\end{remark}

\begin{remark}
    \label{trappexist}
    Given an attracting neighborhood $U$, then $\omega(U)\subset \Int U$. By compactness of $X$,  $\cl \bigcup_{s\in [\tau,\infty)} \varphi^s(U) \subset \Int U$, cf.\ Lemma~\ref{att1234}.
     By Remark \ref{forwinvisolnbhd}, it follows that $V=\bigcup_{s\in [0,\tau]}\varphi^s(U)\supseteq U$ is  forward invariant, 
     and by continuity
\[
\varphi^\tau(\cl V) \subset \cl\varphi^\tau(V) =
\cl \bigcup_{s\in [\tau,2\tau]}\varphi^s(U)
\subset \cl \bigcup_{s\in [\tau,\infty)}\varphi^s(U)
 \subset \Int U\subset \Int V.
\]
 which proves that $V$ is a \emph{trapping region} for $A$. This is stronger than forward invariant, isolating neighborhood and stronger than attracting neighborhood. 
By Lemma \ref{involattrep123}, $V^c\subset U^c$ is a \emph{repelling region} for $A^*$. 
Therefore, a repelling neighborhood contains a repelling region for the same repeller. Note that this is stronger than the existence of a backward invariant, isolating neighborhood as in 
Remark~\ref{backwinvisolnbhd}.
\end{remark}

The duality Lemma \ref{att1234} allows us 
to introduce the notions of dual repeller and dual attractor.
\begin{definition}
\label{defn:duals1}
Let $\varphi$ be a continuous and proper dynamical system on a compact space.
For an attractor
$A\in \sAtt(\phi)$, with attracting neighborhood $U$,
the \emph{dual repeller} 
of $A$  is defined by
\begin{equation}
    \label{dualatt2323}
A^* =\omega(U)^* =\alpha(U^c) \in \sRep(\varphi).
\end{equation}
For a repeller
$R\in \sRep(\phi)$, with repelling neighborhood $U$,
the \emph{dual  attractor}\  of $R$  is defined by
\begin{equation}
    \label{dualrep3434}
R^* =\alpha(U)^* =\omega(U^c)\in \sAtt(\varphi).
\end{equation}
\end{definition}

The pairs $(A,A^*)$ and $(R^*,R)$ are 
attractor-repeller pairs as in Definition~\ref{ARpairdefn234}.
Under the conditions in Definition~\ref{defn:duals1}, for a pair
$(A,R)\in\sARpair(\varphi)$, $R$ is uniquely determined by $A$ and vice-versa, by Lemma~\ref{att1234}. 
Example~\ref{invnotomega} provides a counterexample in the case that $X$ is not compact, i.e.~two pairs $(A,R)$ and $(A,R')$ with $R\neq R'$.

From the definitions in \eqref{dualatt2323} and \eqref{dualrep3434}, the duality maps $A\mapsto A^*$ and $R\mapsto R^*$ are involutive bijections.
By Lemma \ref{att1234} the notions of dual repeller and dual attractor are well-defined.  The wedge on $\sAtt(\varphi)$ can also be given by $A\wedge A'=\omega(A\cap A')$.
 obtain the following commuative diagram:
 \begin{equation}
        \label{dualitydiagno12}
    \begin{diagram}
    \node{\sANbhd(\varphi)}\arrow{e,lr,<>}{^c}{\cong}\arrow{s,l}{\omega} \node{\sRNbhd(\varphi)}\arrow{s,r}{\alpha} \\
    \node{\sAtt(\varphi)} \arrow{e,lr,<>}{^*}{\cong} \node{\sRep(\varphi)}
    \end{diagram}
        \end{equation}
    and  all maps are lattice homomorphisms, cf.\ \cite{LSoA1}.
\begin{lemma}
    \label{pairsandatt}
    Let $\varphi$ be a continuous and proper system with $X$ compact, then
    \[
\sAtt(\varphi) \cong \sARpair(\varphi),
    \]
via the isomorphism $A \mapsto (A,A^*)$.
\end{lemma}
\begin{proof}
Lemma \ref{att1234} states that $\omega(U)=\omega(U')$
if and only if $\alpha(U^c) = \alpha(U'^c)$, which proves the lemma.
    \qed
\end{proof}

If we use  attractor-repeller pairs in compact spaces, we obtain the following statement concerning the asymptotic behavior of orbits.
\begin{lemma}
 \label{asympattrep1221}
 Let $\varphi$ be a continuous and proper system with $X$ compact.
Let  $P=(A,A^*)$ 
be an attractor-repeller pair and 
let $x\in X\smin (A\cup A^*)$. For a complete orbit $\gamma_x$ it holds that $\omega(x)\subset A$ and $\alphaOg\subset A^*$.
 \end{lemma}

 \begin{proof}
Let $U\in\sANbhd(\varphi)$. By Lemma \ref{asympattrep} there exists a $\tau>0$ such that $\gamma_x(-\tau)\subset U^c$. Then,
$\varphi^s\bigl( \gamma_x(-\tau-s)\bigr) =\gamma_x(-\tau)\in U^c$ and thus
$\gamma_x(-\tau-s) \in \varphi^{-s}\bigl( \gamma_x(-\tau)\bigr)\subset \varphi^{-s}(U^c)$ for all $s\ge 0$.
By the definition of orbital alpha-limit set this implies
\[
\begin{aligned}
\alphaOg &= \bigcap_{t\ge 0} \cl \bigcup_{s\ge t} \gamma_x(-s)
=\bigcap_{t\ge \tau} \cl \bigcup_{s\ge t} \gamma_x(-s)
= \bigcap_{t'\ge 0} \cl \bigcup_{s'\ge t'} \gamma_x(-s'-\tau)\\
&\subset \bigcap_{t'\ge 0} \cl \bigcup_{s'\ge t'} \varphi^{-s'}(U^c)=A^*.
\end{aligned}
\]
Moreover $\gamma_x(\tau)\in U$  implies $\omega(x)=\omega(\gamma_x(\tau))\subset \omega(U)=A$.
     \qed
 \end{proof}

\begin{lemma}
    \label{dualsviacapcup}
    Let $\varphi$ be a continuous and proper system with $X$ compact.
     Let $P=(A,A^*)$ be a an attractor-repeller pair.
     Then,
     \[
     A^*=\bigcap \bigl\{ U^c~|~U\in \sANbhd(\varphi)~\hbox{with~~}\omega(U) =A \bigr\}.
     \]
\end{lemma}
\begin{proof}
By Lemma \ref{asympattrep1221}, $A^*= \{x\in X\mid \omega(x)\cap A=\varnothing\}$.
 Define $\Bb(A) = \bigcup \bigl\{ U~|~U\in \sANbhd(\varphi)~\hbox{with~~}\omega(U) =A \bigr\}$, which is the \emph{basin of attraction of $A$}.
Then, $\Bb(A) \supseteq X\smin A^*$. Indeed, if $x\in X\setminus A^*$ , then $\omega(x) \subset A$ and $x\in \{x\}\cup U\in \sANbhd(\varphi)$, which yields $x\in \Bb(A)$. 
Moreover, if $x\in \Bb(A)$ then $\omega(x) \subset A$ which implies that $\Bb(A) \subset X\smin A^*$.
    \qed
\end{proof}

\section{Recurrence}
\label{recurrence}

In the prelude of this section and Section \ref{reccompsect} we start with a definition of recurrence that applies to arbitrary dynamical systems.
Based on the lattice of attractor-repeller pairs for a dynamical system, we
define the \emph{recurrent set} $\sR(\varphi)$ as the union of all complete orbits $\gamma_x$ such that
$\gamma_x\subset A\cup R$ for all attractor-repeller pairs $P=(A,R)$, and thus

\begin{equation}\label{eqn:RS}
\sR(\varphi)=\bigwedge_{P\in\sARpair(\varphi)}A\cup R \in \sInvset(\varphi).  
\end{equation}

This definition is equivalent to $x\in \sR(\varphi)$ if and only if there exists a complete orbit   $\gamma_x\subset A\cup R$ for all $P=(A,R)$. The latter is equivalent to  $\gamma_x\subset A$ or $\gamma_x\subset R$ for all $P=(A,R)$.
Indeed, suppose $\gamma_x\subset A\cup R$ and $y\in \gamma_x$. If $y\in R$, then
$\gamma_x\subset R$, since $R$ is forward-backward invariant. This implies $\gamma_x\subset A$ or $\gamma_x\subset R$.

Note that if $x\in \sR(\varphi)$, then 
 $x\in A$ or $x\in R$ for all attractor-repeller pairs $P\in \sARpair(\varphi)$, but the converse need not hold without further conditions.
However, the recurrent set is invariant by definition.
 We emphasize that no conditions on $\varphi$ or $X$ are assumed in this definition.
\begin{definition}
\label{pro}
Define a preorder on $\sR(\varphi)$ as follows: $x\le x'$ if and only if for every $P\in \sARpair(\varphi)$,  
$x'\in A$ implies $x\in A$, or equivalently, 
$x \in R$ implies $x'\in R$. 
The preorder is denoted as $\bigl(\sR(\varphi),\le\bigr)$.
\end{definition}

For a  \emph{finite} sublattice $\sAR\subset\sARpair(\varphi)$   define 
\begin{equation}\label{eqn:RSnow1}
\sR(\varphi;\sAR) = \bigwedge_{P\in \sAR} A\cup R \in \sInvset(\varphi) 
\end{equation}
On $\sR(\varphi;\sAR)$ we define a preorder as before: $x\le x'$ if and only for every $P\in \sAR$, $x'\in A$ implies $x\in A$, or equivalently
$x \in R$ implies $x'\in R$. 

\subsection{Recurrent components}
\label{reccompsect}
From Definition \ref{pro}, for $x,x'\in \sR(\varphi)$ both $x\le x'$ and $x'\le x$ if and only if either $x,x'\in A$ or $x,x'\in R$
for every $P\in\sARpair(\varphi)$. 
This condition defines an equivalence relation on $\sR(\varphi)$, denoted by $x\sim x'$, and the equivalence classes are denoted by $\xi=[x] $. This relation also defines a partial equivalence relation on $X$.

\begin{definition}\label{def:crs}
The equivalence classes $\xi\subset \sR(\varphi)$ of $\sim$
are the {\em  recurrent components} of $\varphi$. The set of recurrent components is denoted by $\RC(\varphi) = \sR(\varphi)/_\sim$. In particular,
\[
\sR(\varphi)=\bigcup_{\xi\in\RC(\varphi)} \xi.
\]
\end{definition}

By construction $(\RC(\varphi),\le)$ is a poset and the partial order can be characterized as follows: $\xi \le \xi'$ if and only if for every $P\in\sARpair(\varphi)$,
 $\xi'\subset A$ implies $\xi\subset A$, or equivalently, 
 $\xi \subset R$ implies $\xi'\subset R$.

\begin{lemma}
    \label{charofcompofR}
   Let $\xi\in \RC(\varphi)$ be a recurrent component. Then,
    \begin{equation}
        \label{charofcompofRform}
        \xi = \Biggl(\bigwedge_{\xi\subset A} A \Biggr)\cap \Biggl(\bigcap_{\xi\subset R} R\Biggr),
    \end{equation}
    and $\xi$ is invariant.
\end{lemma}
\begin{proof}
If $x\in \sR(\varphi)$, then for every $\gamma_x\subset \sR(\varphi)$, $\gamma_x\subset A$ if and only if $x\in A$, and $\gamma_x\subset R$ if and only if $x\in R$, for all $P\in \sARpair(\varphi)$. This implies in particular that $x\sim x'$ for all $x'\in \gamma_x$. Consequently, if  $x\in \xi$ 
and $\gamma_x\subset \sR(\varphi)$, then  $\gamma_x\subset \xi$ which proves that $\xi$ is invariant.

By definition $\xi\subset \bigcap\bigl\{ A\mid \xi\subset A\bigr\}\cap \bigcap\bigl\{ R\mid \xi\subset R\bigr\}$, and since $\xi$ is invariant, we have that 
\[
\xi\subset \Inv\Bigl(\bigcap\bigl\{ A\mid \xi\subset A\bigr\}\cap \bigcap\bigl\{ R\mid \xi\subset R\bigr\}\Bigr) = \bigwedge\bigl\{ A\mid \xi\subset A\bigr\}\cap \bigcap\bigl\{ R\mid \xi\subset R\bigr\},
\]
where the latter equality follows from Lemma \ref{invofint}.
If $x'\in \bigwedge\bigl\{ A\mid \xi\subset A\bigr\}\cap \bigcap\bigl\{ R\mid \xi\subset R\bigr\}$, then due to the invariance, there exists a complete orbit $\gamma_{x'}$ such that 
$\gamma_{x'}\subset A$ for all $P$ with $\xi\subset A$, and $\gamma_{x'}\subset R$ for all $P$ with $\xi\subset R$. Therefore, $x'\in \xi$ which proves the equality in \eqref{charofcompofRform}.
    \qed
\end{proof}

\begin{lemma}\label{lem:wedge-cap-forAR}
Let $\xi\in\RC(\varphi)$. Then, for $P=(A,R)\in \sARpair(\varphi)$,
\begin{description}
\item[{\rm (i)}] If $\xi\cap A\neq\varnothing$,
 then $\xi\subset A$. Likewise, if $\xi\cap R\neq\varnothing$,
 then $\xi\subset R$;
\item[{\rm (ii)}] $x\in\xi$ if and only if there exists a complete orbit $\gamma_x\subset \xi$; 
\item[{\rm (iii)}] For all $\xi\in\RC(\varphi)$ and all $P,P'\in\sARpair(\varphi)$ we have
\[
\xi\cap(A\wedge A')=\varnothing \Longleftrightarrow 
\xi\cap(A\cap A')=\varnothing.
\]
\end{description}
\end{lemma}

\begin{proof}
\noindent{\bf(i)}
Let $x\in\xi\cap A$, then $\xi=[x] \subset A$ by definition of $\sim$, and the same holds for repellers.

\noindent{\bf(ii)}
Since the recurrent component $\xi$ is invariant, we have that
$x\in \xi$ if and only if there exists a complete orbit $\gamma_x\subset \xi$. 

\noindent{\bf(iii)}
By the invariance of $\xi$ we have that $\xi\subset A\cap A'$ if and only if $\xi\subset A\wedge A'$.
\qed
\end{proof}

We now establish some properties concerning the relationship between 
recurrent components and attractors.
In particular, $\xi < \xi'$ if and only if $\xi \le \xi'$ and there exists a
$P\in\sARpair(\varphi)$ such that $\xi \subset A$ and
$\xi'\subset R$. 
Moreover, $\xi$ and $\xi'$ are unordered, denoted by  $\xi\Vert \xi'$, if and only if 
$\xi\not\le\xi'$ and $\xi'\not\le\xi$. The latter is equivalent to the existence distinct 
attractor-repeller pairs $P\neq P'$ 
such that $\xi\subset A$ and $\xi'\subset R$,
and $\xi'\subset A'$ and $\xi\subset R'$, i.e.
$\xi \subset A,R'$ and $\xi' \subset R,A'$.

In the case of $\sR(\varphi;\sAR)$ the equivalence classes of the preorder yield the partially ordered set
$\RC(\varphi;\sAR) := \sR(\varphi;\sAR)/_\sim$.

\begin{remark}
    \label{equivexprforrec}
    In Section \ref{emdofPr} we show $\sR(\varphi;\sAR)$  is given by
\begin{equation}\label{eqn:RS2}
\sR(\varphi;\sAR) = \bigwedge_{P\in \sAR} A\cup R =\bigcap_{P\in\sAR}A\cup R.  
\end{equation}
The expression in \eqref{eqn:RS2} implies, since  $\sAR$ is finite, that the components $\xi\in \RC(\varphi;\sAR)$ are given as the intersection of an attractor and a repeller, cf.\ Rem.\ \ref{invisint}. As such,  $\RC(\varphi;\sAR)$ is the \emph{Morse representation} subordinate to $\sAR$, cf.\ Sect.\ \ref{profcc}, Sect.\ \ref{emdofPr} and \cite[Sect.\ 8.1]{LSoA3}.
\end{remark}

\subsection{Recurrence for continuous dynamics on compact spaces}
\label{reccontcomp}

Equations~\eqref{eqn:RS} and~\eqref{eqn:RSnow1} defining $\sR(\varphi)$ and $\RC(\varphi)$
provide a general characterization of recurrence that applies to all dynamical systems. In the case where a system has some topological properties, a more classical characterization can be obtained. In this section, we give alternative definitions of $\sR(\varphi)$ and $\RC(\varphi)$, which we show coincide with those in Section~\ref{reccompsect}.

In particular, in a compact space omega-limit sets of nonempty sets are nonempty. This motivates the alternative definition of the recurrent set, using Lemma \ref{pairsandatt}, which states that the attractor-repeller pairs are determined by the attractors. In Section \ref{subsec:OI} we show that the definition of recurrent set below is identical to the definition in \eqref{eqn:RS}.

\vskip 6pt
\noindent\textbf{From this section through the end of Section \ref{attnbhds}, we assume that $\varphi$ is a continuous and proper dynamical system  and  $X$ is a compact space.}
\vskip 6pt
 
 Let $\varphi$ be a continuous  and proper dynamical system on a compact space $X$.
Define the \emph{recurrent set} $\sR(\varphi)$ as the set of all points $x\in X$ for which $x\in A\cup A^*$ for all  attractors $A\in \sAtt(\varphi)$, and thus
\begin{equation}\label{eqn:RS1221}
\sR(\varphi)=\bigcap_{A\in\sAtt(\varphi)}A\cup A^*,   
\end{equation}
which is equivalent to $x\in \sR(\varphi)$ if and only if $x\in A$ or $x\in A^*$ for all $A\in \sAtt(\varphi)$.

As before, both $x\le x'$ and $x'\le x$ if and only if either $x,x'\in A$ or $x,x'\in A^*$
for every $A\in\sAtt(\varphi)$. 
This   defines an equivalence relation 
and the set of  recurrent components is denoted by  $\RC(\varphi) =\sR(\varphi)/_\sim$. The order on the  recurrent components is obtained as before.

\begin{lemma}\label{lem:wedge-cap}
Let $\xi\in\RC(\varphi)$. Then,
\begin{description}
\item[{\rm (i)}] If $\xi\cap A\neq\varnothing$,
 then $\xi\subset A$. Likewise, if $\xi\cap A^*\neq\varnothing$,
 then $\xi\subset A^*$;
\item[{\rm (ii)}] $x\in\xi$ if and only if $\omega(x)\subset \xi$;
\item[{\rm (iii)}] For all $\xi\in\RC(\varphi)$ and all $A,A'\in\sAtt(\varphi)$ we have
\[
\xi\cap(A\wedge A')=\varnothing \Longleftrightarrow 
\xi\cap(A\cap A')=\varnothing.
\]
\end{description}
\end{lemma}

\begin{proof}
\noindent{\bf(i)}
Let $x\in\xi\cap A$, then $\xi=[x] \subset A$ by definition of $\sim$, and the same holds for repellers. 

\noindent{\bf(ii)}
Let $x\in \xi$, then for all $A\in \sAtt(\varphi)$ either $x\in A$ or $x\in A^*$. Since attractors and repellers are forward invariant and closed,
it follows that $\omega(x)\subset A$ whenever $x\in A$, and $\omega(x)\subset A^*$ whenever $x\in A^*.$ Hence $\omega(x)\subset\xi$, which proves $x\in\xi$ implies $\omega(x)\subset \xi$. Since recurrent components are disjoint, and both $x$ and $\omega(x)$ lie in some component, $x\in\xi$ if and only if $\omega(x)\subset\xi$.

\noindent{\bf(iii)}
Now, suppose $x\in\xi\cap(A\cap A')$.
By (i), $x \in \xi \subset A\cap A'.$ Since $X$ is compact, $\omega(x)\neq\varnothing,$ and by (ii), $\omega(x) \subset \xi.$ Again, since attractors are forward invariant and closed, $\omega(x)\subset\omega(A\cap A')= A\wedge A'$.
Consequently, $\xi\cap (A\wedge A')\neq\varnothing$.
The converse implication is immediate.
\qed
\end{proof}

\begin{theorem}
\label{prop:RCinv}
Let $\varphi$ be a continuous  and proper dynamical system on a compact topological space $(X,\scrT)$. Then, 
every recurrent component $\xi\in\RC(\varphi)$ is invariant. In particular,
$\sR(\varphi)$ in an invariant set.
\end{theorem}
The proof Theorem \ref{prop:RCinv} is provided in Section \ref{subsec:OI}.

As before let $\sA\subset\sAtt(\varphi)$ be a \emph{finite} sublattice, and define 
\begin{equation}\label{eqn:RS33}
\sR(\varphi;\sA)=\bigcap_{A\in\sA}A\cup A^*.  
\end{equation}
The equivalence classes  yield 
$\RC(\varphi;\sA) := \sR(\varphi;\sA)/_\sim$.
The restricted recurrent sets with respect to a finite sublattice $\sA$ will play a prominent role in Section \ref{profcc} in order to derive properties of  recurrence such as invariance.

\begin{remark}
 For a continuous and invertible dynamical system 
 the invariance of $\sR(\varphi)$ is immediate, since images of intersections are equal to intersection of images.
\end{remark}

\begin{remark}
   The fact that the expression for $\sR(\varphi)$ in \eqref{eqn:RS1221} is invariant by Theorem~\ref{prop:RCinv} shows that the definition of the recurrent set in \eqref{eqn:RS1221} coincides with the expression in \eqref{eqn:RS} whenever $\varphi$ is a proper system on a compact space, see Section~\ref{emdofPr}. 
\end{remark}

\subsection{The Conley  relation}
\label{conleychainrecsect}
In \cite[Defn.\ 6.1]{Conley} Conley defines a  transitive relation on $X$ which captures certain characteristics of the gradient behavior of a continuous and invertible dynamical system $\varphi$ on a compact, Hausdorff space. 
This relation induces a partial order on $\RC(\varphi)$, which is not mentioned in \cite{Conley}. We show that this induced order is the opposite of the  order $(\RC(\varphi),\le)$, cf.\ Sect.\ \ref{reccontcomp}, for spaces $X$ for which Conley's relation is well-behaved, such as compact, Hausdorff spaces, c.f.\ \cite{Conley}, \cite{patrao}.
We start by recalling 
Conley's construction. 
\begin{definition}
\label{chain}
Let $\scrU$ be an open covering of $X$. Given $\tau > 0$,
a \emph{$(\scrU,\tau)$-chain} from $x$ to $x'$ is a   sequence
of points
$\{x_1,\cdots, x_{n+1}\}$ with $x_1=x$ and $x_{n+1} = x'$, and a sequence of times $\{t_1,\cdots, t_n\}$ with $t_i\ge \tau$, such that for every $i=1,\cdots,n$ there exist open sets $U_i\in \scrU$ such that
$\varphi^{t_i}(x_i), x_{i+1} \in U_i$.
Define the {\em Conley  relation} $\scrC\subset X\times X$ as follows. A pair $(x,x') \in \scrC$ if and only if for every open covering $\scrU$ and every $\tau>0$ there exists a $(\scrU,\tau)$-chain from $x$ to $x'$.
\end{definition}

By definition $\scrC$ is a transitive relation on $X$, cf.\ \cite[Assrt.\ 6.1.A]{Conley}.
A point $x$ is called \emph{chain recurrent} if $(x,x)\in \scrC$. The set of all chain recurrent points is exactly $\sR(\varphi)$ by  \cite[Assrt.\ 6.2.A]{Conley}. Furthermore, the {\em chain components} are the equivalence classes of the relation defined by $x\sim_{\scrC}x'$ if and only if $(x,x')\in\scrC$ and $(x',x)\in\scrC.$
Finally, the relation $\scrC$ induces the following partial order on the chain components $[x]_{\sim_\scrC}\le_\scrC[x']_{\sim_\scrC}$ if $(x,x')\in\scrC.$ This is well-defined by transitivity of $\scrC.$
It is proved in \cite[Assrt.\ 6.1.A]{Conley} that $\scrC$ is a closed relation provided that $X$ is a compact, Hausdorff space. Under these conditions the following equivalence is true.

\begin{proposition}\label{prop:conley-order}
Let $\varphi$ be a continuous and invertible dynamical system on a compact, Hausdorff space $X$. 
The chain components identically correspond to the recurrent components in $\RC(\varphi).$
Moreover, the partial order $\le$ on $\RC(\varphi)$ from Definition \ref{def:crs} is opposite to the order $\le_{\scrC}$ induced on $\RC(\varphi)$ by the Conley  relation.
\end{proposition}

\begin{proof}
Following \cite[Sect.\ 6.1]{Conley}, for $x\in X$ define $\rmOmega(x)=\{x'~|~(x,x')\in \scrC\}.$ 
By definition, $(x,x')\in\scrC$ if and only if $x'\in\rmOmega(x)$. 
From  \cite[Assrt.\ 6.1.C]{Conley}, 
\begin{equation}\label{eqOmega12}
\rmOmega(x)=\bigcap_{A\supseteq \omega(x)}A,
\end{equation}
and thus $(x,x')\in \scrC$ if and only if $x'\in A$ for all attractors $A$ such that $\omega(x)\subset A$.
Note that if $x\in\sR(\varphi),$ then $x\in\rmOmega(x)$ by Lemma~\ref{lem:wedge-cap}(ii).
Thus $x$ and $x'$ lie in the same chain component if and only if both $x,x'\in\rmOmega(x)$ and $x,x'\in\rmOmega(x')$.
For $x\in \sR(\varphi)$  Lemma~\ref{lem:wedge-cap}(ii) implies that  $x$ and $\omega(x)$ lie in the same recurrent component, which is equivalent to $x\in A$ if and only if $\omega(x)\subset A.$ Therefore 
\begin{equation}\label{eqOmega}
\rmOmega(x)=\bigcap_{A\ni x}A \quad \text{for } x\in \sR(\varphi).
\end{equation}
Consequently, $x$ and $x'$ lie in the same chain component if and only if $x\in A \iff x'\in A$ for all $A\in\sAtt(\varphi),$ i.e.\ $x$ and $x'$ lie in the same recurrent component in $\RC(\varphi)$. Hence the chain components coincide with  recurrent components.

We now restrict $\scrC$ to $\sR(\varphi)$, and by transitivity this induces a well-defined partial order  on the recurrent components $\RC(\varphi)$.
Then, for $\xi\in \RC(\varphi)$ denote $\{\xi'~|~(\xi,\xi')\in \scrC\}$ again by $\rmOmega(\xi)$.
By definition, $\xi\le_\scrC \xi'$ if and only if $\xi'\in \rmOmega(\xi).$
From \eqref{eqOmega}, it follows by transitivity that
$$
\rmOmega(\xi)=\bigcap_{A\supseteq \xi}A.
$$
Therefore, $\xi'\in \rmOmega(\xi)$ if and only if $\xi\subset A$ implies $\xi'\subset A$ for all $A\in\sAtt(\varphi)$, i.e. $\xi'\le\xi$.
\qed
\end{proof}

We continue with the relation between lattice theory and  recurrence in the case of a continuous and proper dynamical system on a compact topological space.
We show that the partial order defined on $\RC(\varphi)$ can be derived from the lattice theoretic properties of the lattice of attractors $\sAtt(\varphi)$.

\section{Profinite properties of  recurrent components}
\label{profcc}
In this section $\varphi$ is a continuous and proper dynamical system on a compact space $X$,
which allows us to use the attractors $\sAtt(\varphi)$ to characterize attractor-repeller pairs by Lemma~\ref{pairsandatt}. 
For bounded, distributive lattices there exists a powerful duality between order, topology, and algebra.

For the lattice of  attractors, $\sAtt(\varphi)$,  a
subset $I\subset \sAtt(\varphi)$  is a prime ideal if and only if $I=h^{-1}(0)$ for some lattice homomorphism $h\colon \sAtt(\varphi)\to {\bf 2}:=\{0,1\}$. The set of prime ideals, the spectrum $\sfS\sAtt(\varphi)$, is a poset with partial order given by set-inclusion of prime ideals.
The map
\begin{equation}\label{jmap}
\begin{aligned}
\jmath\colon \sAtt(\varphi) & \to \sSet\bigl(\sfS\sAtt(\varphi)\bigr) \\
A &\mapsto \jmath(A) =  \{I\in \sfS\sAtt(\varphi)~|~A\not \in I\}
\end{aligned}
\end{equation}
 is a lattice homomorphism, cf.\ \cite[Lem.\ 10.20]{DP02}. 
Moreover, $\jmath(A)$ is a down-set, i.e.~$\jmath(A)\in\sO\bigl(\sfS\sAtt(\varphi)\bigr)$. 
The collection 
\[
\cB := \bigl\{\jmath(A) \smin \jmath(A')~|~ A,A'\in \sAtt(\varphi)\bigr\},
\]
are the convex sets in $\sfS\sAtt(\varphi)$. We choose $\cB$ as the basis of a topology, $\scrT_{\sfS\sAtt}$, called the {\em Priestley topology}, and $\sfS\sAtt(\varphi)$ is called a \emph{Priestley space}. The Priestley topology is compact, Hausdorff, zero-dimensional, and order-separating.
The Priestley representation theorem states that
the lattice
$\sAtt(\varphi)$ is isomorphic to the lattice of clopen down-sets in the Priestley space $\bigl(\sfS\sAtt(\varphi),\scrT_{\sfS\sAtt}\bigr)$, and the map $\jmath\colon\sAtt(\varphi) \to \sO^{\rm clp}\bigl(\sfS\sAtt(\varphi)\bigr)$ defines the isomorphism of lattices, cf.\ \cite{DP02,Roman}. These properties are do \emph{not} depend on the topology on $X$.
 
For $\sA\in\sub_F\sAtt(\varphi)$,
the collection of spectra $\sfS\sA$ forms an inverse system of posets over $\sub_F\sAtt(\varphi)$, which is a directed set via inclusion. Indeed, for finite sublattices $\sA\subset \sA'$ we denote the inclusions by  $\iota_{\sA\sA'}\colon \sA\rightarrowtail \sA'$, and by Birkhoff duality we have the maps $\iota_{\sA\sA'}^{-1}\colon \sfS\sA'\to\sfS\sA$ given by $I'\mapsto I=\iota_{\sA\sA'}^{-1}(I') = \{A\in \sA\mid A\in I'\}=I'\cap \sA$. Then,
\begin{description}
    \item [(i)] $\sA \subset \sA'\subset \sA''$ implies $\iota^{-1}_{\sA\sA''} = \iota^{-1}_{\sA\sA'}\circ \iota^{-1}_{\sA'\sA''}$;
    \item [(ii)] $\iota^{-1}_{\sA\sA} = \id$ for all $\sA\in \sub_F\sAtt(\varphi)$,
\end{description}
which makes $\{(\sfS\sA,\subset)\}_{\sA\in \sub_F\sAtt(\varphi)}$ an inverse system of posets over $\sub_F\sAtt$.
By the results in \cite{speed1,speed2} we have that $\sfS\sAtt(\varphi)$ is a \emph{profinite poset} 
 and is given as
the inverse limit: 
\[
\sfS\sAtt(\varphi) \cong \varprojlim_{\sA\in \sub_F\sAtt} \sfS\sA.
\]
The latter is both an order-isomorphism and a homeomorphism.\footnote{From the above results it also follows that $\sAtt(\varphi) \cong \varinjlim_{\sA\in \sub_F\sAtt} \sA$.}
In this section we make this isomorphism more explicit in terms of representations of $\sfS\sAtt(\varphi)$.

\subsection{Representation of finite spectra via Morse sets}
\label{Morsesets12}

As in the case of finite sublattices the
the prime ideals for $\sA\in \sub_F\sAtt(\varphi)$ are related to the prime ideals in $\sAtt(\varphi)$. 

\begin{lemma}
\label{restrpi}
$I$ is a prime ideal in $\sA$ if and only if $I=I'\cap \sA$ for some prime ideal $I'$ in
$\sAtt(\varphi)$. 
\end{lemma}

\begin{proof}
As in the finite case denote the inclusion $\sA\subset \sAtt(\varphi)$ by the lattice embedding
$\iota\colon \sA \rightarrowtail \sAtt(\varphi)$.
By Priestley duality the induced map between the Priestley duals
is a surjective order-preserving map $\iota^{-1}\colon \sfS\sAtt(\varphi) \twoheadrightarrow \sfS\sA$ and is given by
\[
\iota^{-1}(I') = \{ A\in \sA~|~A\in I'\} = I'\cap \sA,
\]
which 
 completes the proof.
\qed
\end{proof}

Consider the map $\rmPsi_{\sA}\colon\sfS\sA \to\RC(\varphi;\sA)$ given by
\[
I \mapsto \left(\bigcap_{A\in I^c}A \right) \bigcap 
 \left( \bigcap_{A\in I}A^*\right),\quad I\in \sfS\sA.
\]
Note that if $x\in\rmPsi_{\sA}(I)$, then $\rmPsi_{\sA}(I)=[x]\in\RC(\varphi;\sA)$. 
Indeed, by definition of $\sim$ and $\rmPsi_{\sA}$, we have
$x,x'\in\rmPsi_{\sA}(I)$ implies that $x,x'\in A$ for all
$A\notin I$ and $x,x'\in A^*$ for all $A\in I.$ 
Since $I^c$ is a prime filter, we have that $A:=\min I^c \in \sJ\sA$, the join-irreducible elements in $\sA$ (see \cite{Roman,DP02} for a definition), and $I^c = \up A$.
Similarly, $I$ is a prime ideal, and $A':=\max I\in \sJ^*\sA$, the meet-irreducible elements in $\sA$, (see \cite{Roman,DP02} for a definition), and $I = \down A'$.
By the expression for $\rmPsi_{\sA}(I)$
we have that $\rmPsi_{\sA}(I) = A\cap A'^*$, since $\bigcap_{A\in I}A^*=\bigl(\bigcup_{A\in I}A\bigr)^*$. Due to the above characterization of $I$ and $I^c$, 
we have that the unique immediate predecessor to the join-irreducible attractor $A$ is given by $A^\pred = A\wedge A'$.
Indeed, $A\wedge A'$ is a predecessor for both $A$ and $A'$ and thus $A\wedge A'\in I$. 
Since $A\in \sJ(\sA),$ we have that $A\wedge  A'\subseteq  A^\pred\subset A$ and thus $A^\pred\not\in I^c$. Moreover, $I\cup I^c=\sA$
which implies that $A^\pred\in I$ and consequently $A^\pred\subset A'$. Therefore, $A^\pred  \subset A\wedge A'$, and
$A\cap (A^\pred)^* = A\cap (A\wedge A')^*= A\cap (A^*\cup A'^*) = A\cap A'^*$.
For $I \in \sfS\sA$,
\begin{equation*}
    \label{transtojoin}
\rmPsi_{\sA}(I) = A\cap (A^\pred)^* \not = \varnothing.
\end{equation*}
The results in \cite{LSoA3} now imply that $\RC(\varphi;\sA) = \sM\sA$, the \emph{Morse representation subordinate to $\sA$}, 
and $\rmPsi_{\sA}\colon\sfS\sA \to\RC(\varphi;\sA)$ is an order-isomorphism.

Using Lemma \ref{restrpi}, we can now give an explicit description of the dual map of the embedding $\sA\subset \sAtt(\varphi)$.
\begin{corollary}
\label{surj}
Let $\sA\subset \sAtt(\varphi)$ be finite sublattice formalized by the lattice embedding
$\iota\colon \sA \rightarrowtail \sAtt(\varphi)$.
Then, $\pi_{\sA}=\rmPsi_{\sA}\circ \iota^{-1}\colon\sfS\sAtt(\varphi) \twoheadrightarrow \RC(\varphi;\sA)$ is the induced order-preserving surjection  
 given by
\begin{equation}
    \label{finiteidealchar}
\pi_{\sA}(I)=\left(\bigcap_{A\in I^c\cap \sA}A \right) \bigcap 
 \left( \bigcap_{A\in I\cap \sA}A^*\right) \neq \varnothing,\quad I\in \sfS\sAtt(\varphi).
\end{equation}
\end{corollary}

\begin{proof}
The characterization of the prime ideals in Lemma  \ref{restrpi}  combined with
the definition of $\rmPsi_{\sA}$ yields the desired result. 
\qed
\end{proof}

\subsection{Properties of inverse limits of finite spectral components}

The isomorphisms between $\sfS\sA$ and $\RC(\varphi;\sA)$
for $\sA\in \sub_F\sAtt(\varphi)$ yield the following correspondences: 
\begin{equation}
    \label{profinforatt}
\sfS\sAtt(\varphi) \cong \varprojlim  \sfS\sA \cong \varprojlim  \RC(\varphi;\sA) = \varprojlim \sM\sA.
\end{equation}
Since the latter is the inverse limit is  of subsets of $X$, it can be characterized in terms of intersections.
We start with a general  version of the Cantor's intersection theorem.

An \emph{inverse system of subsets} $\{U_\alpha\}_{\alpha\in \cA}$, with $\cA$ a directed set,
is a collection of subsets   $U_\alpha \subset X$ such that
\begin{equation}
    \label{ISclosed}
U_\beta \subset U_\alpha \quad \hbox{for }\quad \alpha \le \beta.
\end{equation}

\begin{lemma}
\label{CIT}
Let $X$ be a compact topological space, and
let $\{U_\alpha\}_{\alpha\in \cA}$ be an inverse system of nonempty, closed subsets of $X$
defined by \eqref{ISclosed}. Then,
\[
\varprojlim U_\alpha \cong \bigcap_{\alpha\in \cA} U_\alpha \not = \varnothing.
\]
\end{lemma}

\begin{proof}
A topological space $X$ is compact if and only if $\bigcap_{\alpha\in \cA} U_\alpha \not = \varnothing$ for every collection of closed subsets $\{U_\alpha\}_{\alpha\in \cA}$ with the finite intersection property\footnote{A collection $\{U_\alpha\}_{\alpha\in \cA}$ has the finite intersection property if
every finite intersection $\bigcap_{\alpha\in \cA'} U_\alpha$ is nonempty.}, cf.\ \cite[Sect.\ 26]{Munkres}.
Since $\cA$ is a directed set, and  all the sets $U_\alpha$ are nonempty, we have that for every pair $\alpha,\beta\in \cA$ there exists a $\gamma\in \cA$
with $\alpha\le \gamma$ and $\beta \le\gamma $
such that $U_\gamma \subset U_\alpha$ and $U_\gamma \subset U_\beta$. Therefore,
\[
U_\gamma \subset U_\alpha \cap U_\beta \not = \varnothing.
\]
Consequently, every finite intersection $\bigcap_{\alpha\in \cA'} U_\alpha$ is nonempty, and therefore  $\{U_\alpha\}_{\alpha\in \cA}$ satisfies the finite intersection property. The compactness of $X$ then implies that $\bigcap_{\alpha\in \cA} U_\alpha$ is nonempty.

The inverse limit $\varprojlim U_\alpha$ defined above then consists of points $\bar x\in \prod_{\alpha\in \cA} U_\alpha$ such that $\pi_\alpha(\bar x) = \pi_\beta(\bar x)$, i.e.
points $\bar x$ such that $\pi_\alpha(\bar x) = x \in U_\alpha$ for all $\alpha$.
In other words $x \in \bigcap _\alpha U_\alpha$. The map $x \mapsto \bar x = (\cdots,x,\cdots)$ gives the required bijection, cf.\ \cite{hochster}.
\qed
\end{proof}

Fix $I\in \sfS\sAtt(\varphi)$. 
From Corollary \ref{surj} we have that for $\sA\in \sub_F\sAtt(\varphi)$
\[
\pi_{\sA}(I) = \left(\bigcap_{A\in I^c\cap \sA}A \right) \bigcap 
 \left( \bigcap_{A\in I\cap \sA}A^*\right) \not = \varnothing.
\]
By construction $\pi_{\sA'}(I) \subset \pi_{\sA}(I)$ for all  $\sA\subset \sA'$  which 
provides the desired inverse system of nonempty, closed subsets of $X$. 
From Lemma \ref{CIT} we have following result, cf.\ \cite[Sect.\ 2.5]{Ingram}, \cite[pp.\ 153]{hochster}.

\begin{lemma}
\label{nonemptylimit}
The map $\rmPsi\colon\sfS\sAtt(\varphi)\to\sSet(X)$ is given by
\begin{equation}
    \label{nonempty}
\rmPsi(I):= \left(\bigcap_{A\in I^c}A \right) \bigcap 
 \left( \bigcap_{A\in I}A^*\right) = \bigcap_{\sA\in \sub_F\sAtt} \pi_{\sA}(I)  \not = \varnothing.
\end{equation}
\end{lemma}

\begin{proof}
By definition $\rmPsi(I) \subset \pi_{\sA}(I)$, and thus
$\rmPsi(I) \subset \bigcap_{\sA\in \sub_F\sAtt} \pi_{\sA}(I)$.
Define $\3_A:= \{\varnothing,A,\omega(X)\} \in \sub_F\sAtt$. Then,
\[
\bigcap_{\sA\in \sub_F\sAtt} \pi_{\sA}(I) \subset
\bigcap_{A\in \sAtt(\varphi)} \pi_{\3_A}(I),
\]
since $\{\3_{\sA}\}_{A\in \sAtt(\varphi)}\subset \sub_F\sAtt(\varphi)$.
For $I\in \sfS\sAtt(\varphi)$ we have that 
\[
I\cap \3_A = \begin{cases}
    \varnothing  & \text{ if  } A\not \in I \\
    \{\varnothing,A\}  & \text{ if  } A\in I.
\end{cases}
\]
This implies $\pi_{\3_A}(I) = A\cap X = A$ if $A \in I^c$
and $\pi_{\3_A}(I) = X\cap X \cap A^*= A^*$ if $A \in I$.
For the above intersection this gives
\[
\bigcap_{A\in \sAtt(\varphi)} \pi_{\3_A}(I) = \left(\bigcap_{A\in I^c}A \right) \bigcap 
 \left( \bigcap_{A\in I}A^*\right) = \rmPsi(I),
 \]
 which proves the expression in \eqref{nonempty}.
 The nonemptiness of the set $\pi_{\sA}(I)$ follows from 
 Corollary \ref{surj}, and  Lemma \ref{CIT} implies $\bigcap_{\sA\in \sub_F\sAtt} \pi_{\sA}(I)  \not = \varnothing$.
 \qed
\end{proof}

\subsection{Invariance}
\label{invarianceforHaus}
The invariance of attractors  yields
 invariance properties for the poset of recurrent components $\RC(\varphi)$.
\begin{lemma}
\label{invfinite}
Each $\xi\in \RC(\varphi;\sA)$ is an invariant set.
\end{lemma}

\begin{proof}
Since $\sA$ is a finite lattice, a prime ideal is of the form $I = (\up A)^c$, where $A\in \sJ\sA$
a join-irreducible attractor. Prime filters are of the form $I^c = \up A$.
From the formula for
$\xi = \pi_{\sA}(I)$ we have that $\min I^c = \bigcap_{A\in I^c}A$ is a join-irreducible attractor and therefore invariant. The intersection of dual repellers is forward-backward invariant. By Equation \eqref{finiteidealchar}, $\xi$ is  the intersection of an invariant set with a forward-backward invariant set, which is invariant by \cite[Lemma 2.9]{LSoA1}.
\qed
\end{proof}

Lemma \ref{nonemptylimit} shows that the nonemptiness of the components $\pi_{\sA}(I)$ is carried over to the limit.
A similar property can be proved for the invariance  of the sets $\pi_{\sA}(I)$.
Let $(\cA,\le)$ be a directed set and let 
$\{S_\alpha\}_{\alpha\in \cA}$ be a collection of closed, invariant sets such that
\[
 S_\beta \subset S_\alpha \quad \hbox{for}\quad \alpha \le \beta,
\]
which makes it an inverse system.

\begin{theorem}
    \label{inverssysofinvsetsappl}
    Let $\varphi$ be a continuous and proper dynamical system on a compact space $X$, and let $\{S_\alpha\}_{\alpha\in \cA}$ be an inverse system of closed invariant sets. Then, $\bigcap_{\alpha\in \cA} S_\alpha$ is invariant.
\end{theorem}

\begin{proof}
Since $\varphi$ is proper, the maps $\varphi^t$ have compact fibers for all $t\ge 0$. Therefore,
  by the invariance of the sets $S_\alpha$ for all $\alpha \in \cA$, Theorem \ref{inverssysofinvsets}  implies that 
$\varphi^t\bigl( \bigcap_{\alpha\in \cA} S_\alpha \bigr)
=\bigcap_{\alpha\in \cA} \varphi^t(S_\alpha) = \bigcap_{\alpha\in \cA} S_\alpha$ for all $t\ge 0$, which proves the 
  invariance of $\bigcap_{\alpha\in \cA} S_\alpha$.
    \qed
\end{proof}

\begin{corollary}
\label{invariance}
\begin{equation*}
 \rmPsi(I)=\left(\bigcap_{A\in I^c}A \right) \bigcap 
 \left( \bigcap_{A\in I}A^*\right) = \bigcap_{\sA\in \sub_F\sAtt} \pi_{\sA}(I)  \not = \varnothing
\end{equation*}
is invariant, which shows that the definition of $\sR(\varphi)$ in \eqref{eqn:RS1221} coincides with \eqref{eqn:RS}.
\end{corollary}

\begin{proof}
The collection $\{\pi_{\sA}(I)\}$, $\sA\in \sub_F\sAtt(\varphi)$ is an inverse system of compact, closed invariant sets, and
  Lemma \ref{nonemptylimit} shows that the intersection is nonempty. 
  Invariance follows from Theorem \ref{inverssysofinvsetsappl}.
\qed
  \end{proof}

\begin{remark}
    The conclusion of Lemma \ref{nonemptylimit} shows that $\bigcap_{\sA\in \sub_F\sAtt} \pi_{\sA}(I)$ is a profinite set. This property is used to establish nonemptyness and invariance. In Section \ref{priso} we show that the latter is in fact $\RC(\varphi)$ and isomorphic to $\sfS\sAtt(\varphi)$. This extends the abstract result that Priestley spaces are profinite to include partial order and Priestley topology. 
\end{remark}

\section{Priestley isomorphisms}
\label{priso}
In this section
we continue to assume that $\varphi$ is a continuous and proper dynamical system on a compact space $X$.
We now show that the poset of recurrent components $\RC(\varphi)$ is related to the  Priestley space  $\sfS\sAtt(\varphi)$.
The relation can be expressed with explicit maps and topologies.

\subsection{$\RC(\varphi)$ and $\sfS\sAtt(\varphi)$ are order-isomorphic}\label{subsec:OI}
For the lattice of attractors $\sAtt(\varphi)$ the spectrum is the Priestley space denoted by $\bigl(\sfS\sAtt(\varphi),\scrT_{\sfS\sAtt}\bigr)$.
Define the map $\rmPhi\colon\RC(\varphi) \to\sfS\sAtt(\varphi)$ by
\begin{equation}
    \label{Phimap}
\xi \mapsto \{A\in\sAtt(\varphi)~|~\xi\cap A=\varnothing\}
=\{A\in\sAtt(\varphi)~|~\xi\not\subset A\}.
\end{equation}
The last equality is due to Lemma~\ref{lem:wedge-cap}(i).

\begin{lemma}\label{lem:Phi}
The map $\rmPhi\colon\RC(\varphi) \to\sfS\sAtt(\varphi)$ is well-defined.
\end{lemma}

\begin{proof}
Let $\xi\in \RC(\varphi)$. Define the function $h_\xi\colon \sAtt(\varphi) \to {\bf 2}$ by
\[
h_\xi(A) = \begin{cases}
  0    & \text{ if } \xi\cap A=\varnothing\\ 
   1   & \text{ if } \xi\subset A,
\end{cases}
\] 
which uses Lemma~\ref{lem:wedge-cap}(i) to describe the alternative.
The definition implies that $h_\xi(\varnothing) =0$ and $h_\xi(\omega(X)) = 1$.
As for meets and joins of attractors, we argue as follows. Consider $A\cup A'$. Then, $\xi\cap (A\cup A') = (\xi\cap A)\cup(\xi\cap A') = \varnothing$ if and only if both $\xi\cap A=\varnothing$ and  $\xi\cap A'=\varnothing$,
and $\xi\subset  A\cup A'$ if and only if $\xi\subset A$ or  $\xi\subset A'$.  This shows $h_\xi(A\cup A')
= \max\{h_\xi(A),h_\xi(A')\}$.

Consider $A\wedge A'$. Then, by Lemma \ref{lem:wedge-cap}(iii), $\xi\cap (A\wedge A')=\varnothing$ if and only if  $\xi\cap (A\cap A')=\varnothing$.
 Thus, $h_\xi(A\wedge A')=0$ if and only if 
 $\xi\cap A=\varnothing$ or  $\xi\cap A'=\varnothing$, and
$h_\xi(A\wedge A')=1$ if $\xi\subset A\wedge A'$ which is equivalent to $\xi\subset A\cap A'$. Consequently, $h_\xi(A\wedge A')=1$ if and  only if
 $\xi\subset A$ and   $\xi\subset  A'$.
 This shows that $h_\xi(A\wedge A')
= \min\{h_\xi(A),h_\xi(A')\}$, and therefore $h_\xi$ is a lattice homomorphism.

By definition $h_\xi^{-1}(0)$ is a prime ideal of $\sAtt(\varphi)$, and we observe that
$\rmPhi(\xi) = h_\xi^{-1}(0)\in \sfS\sAtt(\varphi)$, which proves that $\rmPhi$ is a well-defined map.
\qed
\end{proof}

In \eqref{nonempty} we introduced the expression $\rmPsi(I)$ for $I\in \sfS\sAtt(\varphi)$, which defines 
 the map $\rmPsi\colon\sfS\sAtt(\varphi) \to\RC(\varphi)$
given by 
\begin{equation}
    \label{Psimap}
 I \mapsto \left(\bigcap_{A\in I^c}A \right) \bigcap 
 \left( \bigcap_{A\in I}A^*\right).
\end{equation}

By Lemma \ref{nonemptylimit} we know that $\rmPsi(I)\neq\varnothing$ for all $I\in\sfS\sAtt(\varphi)$.
Note that if $x\in\rmPsi(I)$, then $x\in \sR(\varphi)$.
Indeed, by the definitions of $\sim$ and $\rmPsi$, we have
$x,x'\in\rmPsi(I)$ implies that $x,x'\in A$ for all
$A\in I^c$ and $x,x'\in A^*$ for all $A\in I$,
and thus $\rmPsi(I)\subset \xi$ for some $\xi\in \RC(\varphi)$.
Moreover, 
every $x\in \xi$ satisfies $x\in \left(\bigcap_{A\in I^c}A \right) \bigcap 
 \left( \bigcap_{A\in I}A^*\right)=\rmPsi(I)$. 
 Therefore,
 $\rmPsi(I)=\xi \in\RC(\varphi)$ is a recurrent component.

\begin{lemma}\label{thm:bijection}
The map $\rmPhi\colon\RC(\varphi) \to\sfS\sAtt(\varphi)$ is bijective
with $\rmPhi^{-1}=\rmPsi.$
\end{lemma}
\begin{proof}
We need to show that $\rmPhi\circ\rmPsi=\id_{\sfS\sAtt}$ and $\rmPsi\circ\rmPhi=\id_{\RC}$.
Let $\xi\in\RC(\varphi)$. Then, $I=\rmPhi(\xi)=\{A\in\sAtt(\varphi)~|~\xi\not\subset A\}\in \sfS\sAtt(\varphi)$.
Therefore, 
\[
\rmPsi(I)=\left(\bigcap_{A\in I^c}A \right) \bigcap 
 \left( \bigcap_{A\in I}A^*\right)
 =\left(\bigcap_{\xi\subset A}A \right) \bigcap 
 \left( \bigcap_{\xi\cap A=\varnothing}A^*\right)=\xi.
\]
By definition $\xi\subset\rmPsi(I)$ so that  the last equality follows from the fact that $\rmPsi(I)$ is a recurrent component.
Finally,
let $I\in\sfS\sAtt(\varphi)$ and $\xi=\rmPsi(I).$ Then
$\rmPhi(\xi)=\{A\in\sAtt(\varphi)~|~\xi\not\subset A\}=I$ by 
definition of $\rmPsi$.
\qed
\end{proof}

\begin{corollary}
    \label{cornonempty}
    $\sR(\varphi) \neq\varnothing$.
\end{corollary}
\begin{proof}
Since $\varnothing,\omega(X)\in \sAtt(\varphi)$, it follows that $\sfS\sAtt(\varphi)\neq \varnothing$. By Lemma \ref{thm:bijection} this implies that $\RC(\varphi)\neq \varnothing$ from which the statement follows. 
    \qed
\end{proof}

\begin{theorem}\label{cor:iso}
$\RC(\varphi)$ and $\sfS\sAtt(\varphi)$ are  order-isomorphic.
\end{theorem}
\begin{proof}
We must show that $\xi\le\xi' $ if and only if $\rmPhi(\xi)\subseteq \rmPhi(\xi' )$.
Let $\xi\le\xi' $. Then, for every $A\in\sAtt(\varphi)$,
we have one of three possibilities. Set $\xi=[x]$ and $\xi'=[x']$.
Recall that $\rmPhi(\xi) = \{A\in\sAtt(\varphi)~|~[x]\cap A=\varnothing\}$
and $\rmPhi(\xi') = \{A\in\sAtt(\varphi)~|~[x']\cap A=\varnothing\}$.
Then,
\begin{description}
    \item [(i)] $x,x'\in A  \iff A\notin\rmPhi(\xi) \hbox{ and } A\notin\rmPhi(\xi' )$;
    \item [(ii)] $x,x'\in A^*  \iff A\in\rmPhi(\xi)\hbox{ and } A\in\rmPhi(\xi' )$;
    \item [(iii)] $x\in A \hbox{ and } x'\in A^*  \iff A\notin\rmPhi(\xi) \hbox{ and } 
A\in\rmPhi(\xi' )$,
\end{description}
which shows that $\xi\le\xi' $ if and only if $\rmPhi(\xi)\subseteq\rmPhi(\xi' ).$
\qed
\end{proof}

\begin{proof}[of Theorem \ref{prop:RCinv}]
By Lemma \ref{thm:bijection} every recurrent component $\xi$ is given by $\xi = \rmPsi(I)$. Invariance follows from  Corollary \ref{invariance}.
\qed
\end{proof}

\subsection{$\RC(\varphi)$ and $\sfS\sAtt(\varphi)$ are homeomorphic}\label{subsec:H}

The poset $\bigl( \RC(\varphi),\le\bigr)$, as a profinite ordered space, has the natural the topology of a Priestley space  via the bijection $\rmPsi$. That is,
letting $\scrT_{\sfS\sAtt}$ denote the Priestley topology on $\sfS\sAtt(\varphi)$, then $\scrT_\sfS:=\rmPsi\scrT_{\sfS\sAtt}$ defines a Priestley topology on $\RC(\varphi)$ and
makes $\rmPsi$ and $\rmPhi$ homeomorphisms. The following proposition is an immediate consequence.

\begin{proposition}
\label{priestRC}
The Priestley spaces $ \bigl( \RC(\varphi),\scrT_\sfS \bigr)$ and $(\sfS\sAtt(\varphi),\scrT_{\sfS\sAtt} )$ are homeomorphic, i.e.  homeomorphic and order-isomorphic.
Moreover, the lattice $\sO^{\rm clp}\bigl(\RC(\varphi)\bigr)$ of clopen downsets in $\RC(\varphi)$ is isomorphic to $\sAtt(\varphi)$.
\end{proposition}

The Priestley topology on $\RC(\varphi)$ can be dynamically characterized via the  Morse sets of the dynamical system. 
Recall that a set $M\subset X$ is a \emph{Morse set} if $M=A\cap R$ for some attractor $A\in \sAtt(\varphi)$ and repeller $R\in \sRep(\varphi)$. 
Morse sets are invariant, cf.\ \cite[Lem.\ 8]{LSoA3} and form the bounded meet-semilattice $\sMorse(\varphi)$ with $M\wedge M' :=\Inv(M\cap M')$.
For a given $M\in \sMorse(\varphi)$ define the set
\begin{equation}
    \label{viaMorse1}
\zeta_{M} := \{ \xi\in \RC(\varphi)~|~ \xi \subset M\} = \bigl(\sR(\varphi)\cap M\bigr)/_\sim,
\end{equation}
where the latter characterization follows from Lemma \ref{lem:wedge-cap}(i). Indeed, $\xi\cap M \neq \varnothing$ is equivalent to $\xi\cap A\neq\varnothing$ and 
$\xi\cap A'^*\neq\varnothing$ which implies $\xi\subset A$ and
$A'^*\subset \xi$, cf.\  Lem.\ \ref{lem:wedge-cap}(i). Therefore, $\xi\subset M$.
We obtain the following characterization of the Priestley topology on the recurrent components $\RC(\varphi)$.
\begin{lemma}
\label{viaMorse}
The topology $\scrT_\sfS$ is generated by the basis
\[
\cB_{\sfS} := \{\zeta_{M}~|~ M\in \sMorse(\varphi)\},
\]
which are the clopen convex sets in $\bigl(\RC(\varphi),\scrT_\sfS\bigr)$.
\end{lemma}
\begin{proof}
A basic (open/closed) set for $\scrT_{\sfS\sAtt}$ is given by $\jmath(A)\smin \jmath(A')$ for some $A,A'\in \sAtt(\varphi)$ and thus is the set of all prime ideals $I\in \sfS\sAtt(\varphi)$ such that  $A\not\in I$ and $A'\in I$. In particular, by the definition in \eqref{Psimap}, $\rmPsi(I)\subset A$ and $\rmPsi(I)\subset A'^*$. Therefore,
$\xi=\rmPsi(I)\subset M =A\cap A'^*$. Conversely, if $\xi\subset M =A\cap A'^*$, then by the definition in \eqref{Phimap},
$A\not\in \rmPhi(\xi)=I$. Similarly, since $\xi\subset A'^*$, it follows from Lemma \ref{lem:wedge-cap}(i) that $\xi\not\subset A'$, and thus $A'\in I$. Summarizing we have that $\rmPsi\bigl(\jmath(A)\smin \jmath(A') \bigr) = \zeta_{M}$.
\qed
\end{proof}

\begin{lemma}
\label{inrecset}
Let $U\in \sANbhd(\varphi)$ with $\omega(U) = A$. If $x\in U\cap \sR(\varphi)$, then $x\in A$. Similarly, if $x\in U^c\cap \sR(\varphi)$, then $x\in A^*$.
\end{lemma}

\begin{proof}
By definition, $x\in \sR(\varphi)$ if and only if $x\in A$ or $x\in A^*$ for all attractors $A\in \sAtt(\varphi)$. If $x\in U\supset A$, then $x\not\in U^c$ and $U^c\supset A^*$, and thus
$x\in U\cap \sR(\varphi)$ implies that $x\in A$. Similarly, $x\in U^c\cap \sR(\varphi)$ implies that $x\in A^*$.
\qed
\end{proof}

Let $\scrT$ denote the topology on $X$, and consider the subspace topology on $\sR(\varphi)$. Then $\RC(\varphi)$ can be given the quotient topology $\scrT_\sim$ induced by the equivalence relation in Definition \ref{def:crs}. The following theorem shows that $\scrT_\sim$ is the same as the topology $\scrT_\sfS$.

\begin{theorem}\label{thm:homeo}
$(\RC(\varphi),\scrT_\sim)$  and $(\sfS\sAtt(\varphi),\scrT_{\sfS\sAtt} )$
are homeomorphic.
\end{theorem}
\begin{proof}
Let  $\pi\colon(\sR(\varphi),\scrT)\to (\RC(\varphi),\scrT_\sim)$ be the quotient map given by $x\mapsto   [x]=\xi$ and consider the commutative diagram:
\begin{equation}
\label{rc44}
\begin{diagram}
\node{\sR(\varphi)}\arrow{s,l,A}{\pi}\arrow{se,l}{g}\\
\node{\RC(\varphi)}\arrow{e,l,J}{\rmPhi}\node{\sfS\sAtt(\varphi)}
\end{diagram}
\end{equation}
with composition $g:=\rmPhi\circ \pi$.
Since $\sR(\varphi)$ is a closed subset of a compact space, it is compact.
By Lemma \ref{thm:bijection} the map $\rmPhi$ is bijective, and thus $g$ is a surjective map.
If $g$ is continuous, then since $\sR(\varphi)$ is compact, and $\sfS\sAtt(\varphi)$ is Hausdorff, $g$ is a closed map and hence a quotient map. By \cite[Corollary 22.3]{Munkres}, $\rmPhi$ is a homeomorphism and $\RC(\varphi)$ is a compact, Hausdorff space homeomorphic to
$\sfS\sAtt(\varphi)$.
It remains to show that $g$ is continuous.

A basic open set in $\sfS\sAtt(\varphi)$ is given by 
$C=j(A)\smin j(A')$
for some pair of attractors $A,A'\in \sAtt(\varphi)$. Then, using \eqref{viaMorse1} and Lemma \ref{viaMorse},
\begin{equation*}
\begin{aligned}
g^{-1}(C)&=\pi^{-1}(\rmPhi^{-1}(C))\\
&= \pi^{-1}\bigl(\rmPsi\bigl(j(A)\smin j(A')\bigr)\bigr)\\
&= \pi^{-1}(\zeta_M)
= M\cap \sR(\varphi),
\end{aligned}
\end{equation*}
where the intersection $M=A\cap A'^*$ is a Morse set.
Since $A$ is an attractor, there is an open attracting neighborhood $U$ 
such that $\omega (U)=A$, cf.\ Sect.\  Rem.\ \ref{intandclforattnbhd}. 
Likewise, since $A'^*$ is a dual repeller, there is an open repelling neigborhood  $U'^c\in \sRNbhd(\varphi)$ by choosing $U'$ a closed attracting neighborhood, cf.\  Rem. \ref{intandclforattnbhd}, 
such that $\omega (U')=A'$, cf.\  Lem.\ \ref{dualsviacapcup}. 
Consequently, $M\subset U\cap U'^c$ and $U\cap U'^c$ is open  in $X$ and thus $W:=U\cap U'^c\cap\sR(\varphi)$ is  open in  $\sR(\varphi)$.  We have $g^{-1}(C)\subseteq W$ by construction.  

Conversely,  let $x\in W$. 
By Lemma \ref{inrecset} we have that 
\[
x\in A\cap A'^*\cap \sR(\varphi) = M\cap \sR(\varphi) = g^{-1}(C),
\]
and  hence $W\subset g^{-1}(C)$.
Combining both inclusions yields $W=g^{-1}(C)$,
which is an open set. This shows that $g$ is continuous, and thus proves that $g$ is a quotient map, and $\rmPhi$ is a homeomorphism.
\qed
\end{proof}

\begin{theorem}
    \label{inverselimthm}
The space of recurrent components $(\RC(\varphi),\scrT_\sim,\le)$ is a profinite poset and 
\[
(\RC(\varphi),\scrT_\sim,\le) \cong \varprojlim \sM\sA,
\]
where the projective limit is over the finite sublattices $\sA\subset \sAtt(\varphi)$.
\end{theorem}

\begin{proof}
   Combine Theorem \ref{thm:homeo} with \eqref{profinforatt}.
   \qed
\end{proof}

\subsection{Decomposition}
\label{decompthm}
The prime ideals are linked to the asymptotic behavior of $\varphi$ which yields a version of Conley's Decomposition Theorem, cf.\ \cite{Conley}.
\begin{theorem}
\label{Conleydecomp}
For every $x\in X$ there exists a unique  $\xi_+\in \RC(\varphi)$ such that $\omega(x)\subset \xi_+$.
If $x\in X\smin\sR(\varphi)$ allows a complete orbit $\gamma_x$, then there exist a  $\xi_-\in\RC(\varphi)$, unique to the orbit $\gamma_x$, with 
$\xi_+<\xi_-$ such that 
\[
\omega(x)\subset \xi_+\quad\text{and}\quad \alphaOg \subset \xi_-.
\]
\end{theorem}
\begin{proof}
For a point $x\in X$ it holds that $x\in X\smin A^*$ or $x\in A^*$ for all attractors $A\in \sAtt(\varphi)$.
Consequently, $\omega(x)\subset A$ whenever $x\in X\smin A^*$, or $\omega(x)\subset A^*$ whenever $x\in A^*$, cf.\  Lem.\ \ref{asympattrep1221}. 
Fix $x\in X$ and define the map  
\begin{equation}
    \label{hom1toasymp}
h_x(A) = \begin{cases}
  0    & \text{ if } \omega(x)\subset A^*\\ 
   1   & \text{ if } \omega(x)\subset A.
\end{cases}
\end{equation}
Note that $h_x(\varnothing)=0$ and $h_x(\omega(X)) = 1$.
For $h_x(A\cup A')$ we have the value $0$ whenever $\omega(x)\subset (A\cup A')^*=A^*\cap  A'^*$ which is realized if both $\omega(x)\subset A^*$ and $\omega(x)\subset A'^*$.
The value $1$ is attained whenever $\omega(x)\subset A\cup A'$
which is realized if  $\omega(x)\subset A$ or $\omega(x)\subset A'$. We conclude 
$h_x(A\cup A') = \max\{h_x(A),h_x(A')\}=h_x(A)\vee h_x(A')$.
For $h_x(A\wedge A')$ we have the value $0$ whenever $\omega(x)\subset (A\wedge A')^*=A^*\cup  A'^*$ which is realized if  $\omega(x)\subset A^*$ or $\omega(x)\subset A'^*$.
The value $1$ is attained whenever $\omega(x)\subset A\wedge A'$
which is realized if both $\omega(x)\subset A$ and $\omega(x)\subset A'$.\footnote{By the idempotency of $\omega$ we have that $\omega(x)\subset A\wedge A'$ if and only if $\omega(x)\subset A\cap A'$, cf.\ Lem.\ \ref{lem:wedge-cap}(iii). Similarly, $\alphaOg\subset A\wedge A'$ if and only if $\alphaOg\subset A\cap A'$ since $(A\cap A') \cap (A^*\cup A'^*)=\varnothing$.}
We conclude that
$h_x(A\wedge A') = \min\{h_x(A),h_x(A')\}=h_x(A)\wedge h_x(A')$.
The map $h_x\colon\sAtt(\varphi) \to {\bf 2}$ is therefore a lattice homomorphism, and 
$
I_+=h_x^{-1}(0) = \bigl\{ A\in \sAtt(\varphi)~|~ \omega(x)\subset A^*\bigr\}
$
is a prime ideal. Consequently,
\[
\omega(x) \subset \left(\bigcap_{A\in I_+^c}A \right) \bigcap 
 \left( \bigcap_{A\in I_+}A^*\right) =: \xi_+ \in \RC(\varphi).
\]

Let $\gamma_x$ be a complete orbit for $x\in X\smin\sR(\varphi)$.
Similar to omega-limit sets, we have that $\alphaOg\subset A$ or $\alphaOg\subset A^*$ whenever $x\in A$, or $\alphaOg\subset A^*$ whenever $x\in X\smin A$. 
In this case the complete orbit is the `parameter' and we define
\begin{equation}
    \label{hom2toasymp}
h_{\gamma_x}(A) = \begin{cases}
  0   & \text{ if } \alphaOg\subset A^*\\ 
   1   & \text{ if } \alphaOg\subset A.
\end{cases}
\end{equation} 
As before,
the map $h_{\gamma_x}$ is  a lattice homomorphism, and 
\[
I_-=h_{\gamma_x}^{-1}(0) = \bigl\{ A\in \sA~|~ \alphaOg\subset A^*\bigr\}
\]
is a prime ideal. Consequently,
\[
\alphaOg \subset \left(\bigcap_{A\in I_-^c}A \right) \bigcap 
 \left( \bigcap_{A\in I_-}A^*\right) =: \xi_- \in \RC(\varphi).
\]
To complete the proof, it remains to show that $\xi_+<\xi_-$, which follows from the definition of the partial order on $\RC(\varphi)$.
Indeed, for attractors $A$ for which $x\in X\smin(A\cup A^*)$ we have $\omega(x)\subset A$ and $\alphaOg\subset A^*$. 
For attractors $A$ for which $x\in A$   we have $\omega(x)\subset A$ and $\alphaOg\subset A$ or $\alphaOg\subset A^*$, and 
for attractors $A$ for which $x\in A^*$   we have $\omega(x)\subset A^*$ and $\alphaOg\subset A^*$.
For $\xi_+$ and $\xi_-$ this implies $\xi_+<\xi_-$.
\qed
\end{proof}

\begin{remark}
If we consider a finite sublattice $\sA\subset \sAtt(\varphi)$, or any sublattice, then the decomposition theorem applies to the restricted recurrent set $\sR(\varphi;\sA)$. In the case that $\sA$ is finite, this results in a Morse representation, cf.\ \cite{Conley}, \cite{LSoA3}.
\end{remark}

If we combine Theorem \ref{Conleydecomp} with the partial order on $\RC(\varphi)$, then we can define a transitive relation $\scrR\subset X\times X$. 

\begin{definition}
    \label{defnofrelR}
    A pair $(x,x') \in \scrR$ if and only there exist $\xi_-\le \xi'_+$   and a complete orbit $\gamma_{x}$     such that $\alphaOg\subset \xi_-$ and  $\omega(x') \subset \xi'_+$, cf.\ Fig.\ \ref{fig:compforR}.
\end{definition}

\begin{figure}
\centering
  \includegraphics[width=2in,height=0.8in]{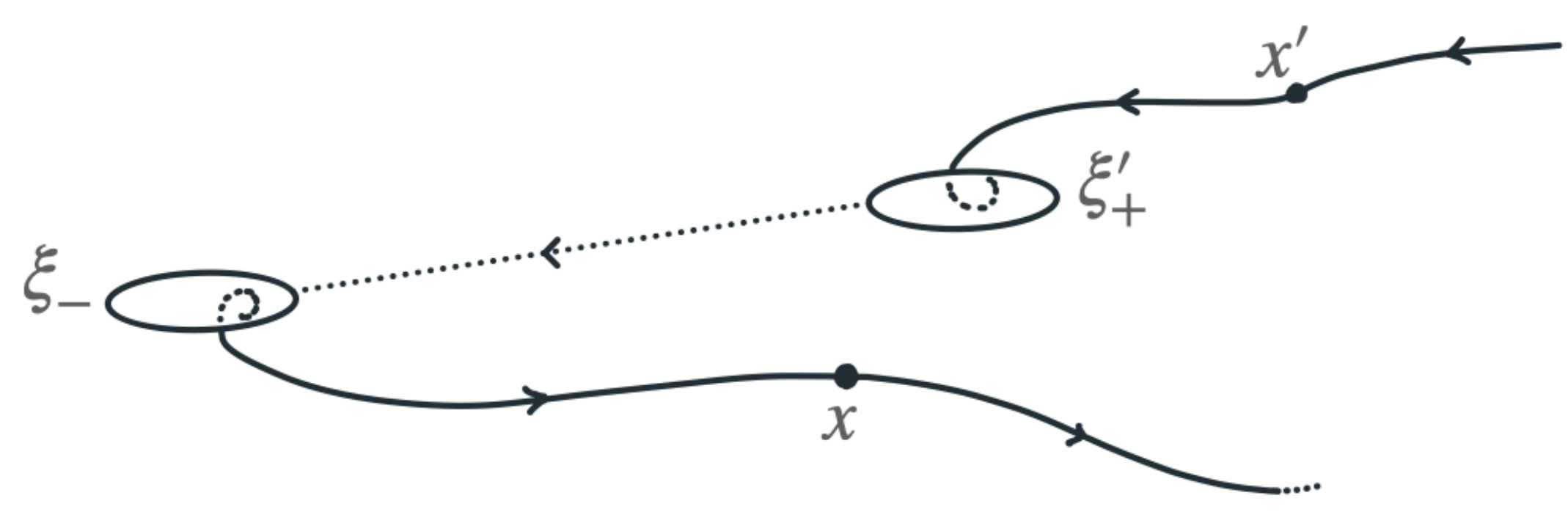}
  \caption{The relation $(x,x')\in \scrR\subset X\times X$  via the partial order on $\RC(\varphi)$ as in Definition~\ref{defnofrelR}.}
  \label{fig:compforR}
\end{figure}
\begin{remark}
    \label{inomega}
   Note that for a pair $(x,x')\in \scrR$ it necessarily holds  that $x\in \omega(X),$
    since the definition of $\scrR$ requires complete orbits $\gamma_x$ at $x$.
    For $x,x'\in\sR(\varphi)$ we have $(x,x')\in\scrR$ if and only if $\xi_-=[x]\le [x']=\xi'_+$.
\end{remark}

\begin{theorem}
    \label{comptoConleyrecrel}
    Let $\varphi$ be a flow on a compact, Hausdorff space.
Then,
    the relation $\scrR$ coincides with the opposite Conley  relation $\scrC^{-1}$.
\end{theorem}
\begin{proof}
Let $x,x'\in X$. Since $\varphi$ is a flow on a compact space, alpha-limit sets and omega-limit sets exist and are nontrivial for both $x$ and $x'$. By Theorem \ref{Conleydecomp}   $\alpha(x)\subset \xi_-=\rmPsi(I)$ and $\omega(x')\subset \xi'_+=\rmPsi(I')$. 
From Lemma \ref{asympattrep1221}
we conclude that $\alpha(x)\subset A$ if and only $x\in A$. Therefore, by Proposition~\ref{prop:conley-order}, $I\subset I' \iff x\in A$ for all $A\supseteq \omega(x')$, which proves that $x\in \rmOmega(x')$ using
 \cite[Assrt.\ 6.1.C]{Conley} and the formula in $\eqref{eqOmega12}$.
The fact that $(x,x')\in \scrC^{-1}$ if and only if $x\in \rmOmega(x')$ completes the proof.
    \qed
\end{proof}

\begin{remark}
    \label{onlyequiv}
     Note that $x\sim x'$ in $\scrR$ if and only if $x,x'\in \xi\in \RC(\varphi)$. Indeed, $(x,x')\in\scrR$ and $(x',x)\in \scrR$ imply that $\xi_+=\xi'_-$ and
    $\xi'_+=\xi_-$.
    Since $\xi_+\le\xi_-$ and $\xi'_+\le\xi'_-$, we have $\xi:=\xi_-=\xi_+=\xi'_-=\xi'_+$.
    This implies $x,x'\in \xi\in \RC(\varphi)$, since for points in $\omega(X)\smin \sR(\varphi)$ the inequalities are strict.
\end{remark}

\section{Attracting neighborhoods and the fundamental cospan}
\label{attnbhds}

In this section
we continue to assume that $\varphi$ is a continuous and proper dynamical system on a compact space $X$.
We extend the order on the recurrent components $\RC(\varphi)$ to an order on all of $X$ using the lattice of attracting neighborhoods. 
In Definition \ref{defnofrelR} we use Theorem \ref{Conleydecomp} to define the transitive relation $\scrR$ using the asymptotic behavior of orbits, which retrieves the Conley relation in the case of flows, cf.\ Thm.\ \ref{comptoConleyrecrel}. 
In this section we use the duality theory for
 the lattices $\sANbhd(\varphi)$ and $\sAtt(\varphi)$ and the homomorphism
 $\omega\colon\sANbhd(\varphi)\to\sAtt(\varphi)$ in order to define $\scrR$ in an alternative manner. The advantage of the latter approach is that it is based on order theoretic properties of attractors and attracting neighborhoods, and the method extends beyond compact spaces, cf.\ Sect.\ \ref{noncompact}. The method also allows for a refinement of the relation $\scrR$, cf.\ Rem.\ \ref{otherattnebhd}.

\subsection{Strong components}
\label{strongcomp123}
Recall that $\sANbhd(\varphi)$ denotes the bounded, distributive lattice of attracting neighborhoods. The map $\omega\colon\sANbhd(\varphi)\to\sAtt(\varphi)$ given by $U\mapsto \omega(U)$ is a surjective lattice homomorphism.
The duals of attracting neighborhoods are the repelling neighborhoods, $\sRNbhd(\varphi)$, and the duality is given by $U\mapsto U^c$. Since $X = U\cup U^c$
we can define a preorder on $X$ directly.

\begin{definition}
\label{pro123}
Define a preorder on $X$ as follows: $x\le  x'$ if and only if for every $U\in\sANbhd(\varphi)$,
 $x'\in U$ implies $x\in U$, or equivalently $x \in U^c$ implies $x'\in U^c$. 
The preorder is denoted as $(X,\le )$.
\end{definition}

From the above defined preorder it follows that $x\le  x'$ and $x'\le  x$ if and only if either $x,x'\in U$ or $x,x'\in U^c$
for every $U\in\sANbhd(\varphi)$. This defines the  equivalence classes in $X$. The
  equivalence relation on $X$ is denoted by $x\sim  x'$ and the equivalence classes are again denoted by $\xi=[x]$.

\begin{definition}\label{def:str123}
The equivalence classes of $\sim $
are called the {\em  strong components} of $\varphi$ and are denoted by $\SC(\varphi) = X/_{\!{\sim }}$,
and the partial order on $\SC(\varphi)$ is again denoted by $\le $.
\end{definition}

Note that as in the case of recurrent components, we have that $\xi\cap U \not = \varnothing$ is equivalent to $\xi \subset U$, and similarly $\xi\cap U^c \not = \varnothing$ is equivalent to $\xi \subset U^c$.
By construction $\SC(\varphi)$ is a poset, and the partial order can be characterized as follows: $\xi \le \xi'$ if for every $U\in\sANbhd(\varphi)$
we have that $\xi'\subset U$ implies $\xi\subset U$, or equivalently $\xi \subset U^c$ implies $\xi'\subset U^c$.

\begin{lemma}
    \label{charofpreo}
    Let $(X,\le)$ be the preorder defined in Definition \ref{pro123}.
    If $x\le x'$ and $x\neq x'$, then for every $U\in \sANbhd(\varphi)$, $x\in U^c$ implies $x'\in \alpha(U^c)\neq \varnothing$.
\end{lemma}

\begin{proof}
Let $U\in \sANbhd(\varphi)$ with $x\in U^c$ and $\alpha(x)\subset \alpha(U^c)=A^*\neq \varnothing$, and suppose that $x'\in U^c\smin A^*$. Then, Lemma \ref{dualsviacapcup}
yields the existence of an attracting neighborhood $U'\in \sANbhd(\varphi)$ such that $\alpha(U'^c)=A^*$ and $x'\notin U'^c$.
Define $V^c=\{x\}\cup U'^c$, then $\alpha(V^c)=A^*\subset \Int U'^c\subset
\Int U'^c\cup \Int\{x\}\subset \Int V^c$, and therefore, $V\in \sANbhd(\varphi)$ with $x\in V^c$ and $x'\notin V^c$.
Consequently, $x\in V^c$ does not imply that $x'\in V^c$,
which contradicts $x\le x'$. Therefore, $x'\in A^*=\alpha(U^c)\neq \varnothing$ for all $U\in \sANbhd(\varphi)$ for which $x\in U^c$.
    \qed
\end{proof}

\begin{lemma}
    \label{containment1}
    Let $A\in \sAtt(\varphi)$  and  $\gamma_x$ a complete orbit with $\alphaOg\subset A$.
    Then $\gamma_x\subset A$.
\end{lemma}
\begin{proof}
Let $V\supseteq A$ be a trapping region for $A$, cf.\ Rem.\ \ref{trappexist}. Then, for every orbital alpha-limit point $y\in \alphaOg$, there exists a net $\{\gamma_{x}(-t_\beta)\}$ with  $\gamma_{x}(-t_\beta)\to y$ as $t_\beta\to \infty$, i.e. there exists a $\beta_0$ such that $\gamma_{x}(-t_\beta)\in V$ for all $\beta_0\le \beta$. 
Suppose, $\gamma_{x}(-t_{0})\not\in V$ 
for some $t_0\in\R$. 
Then, there exists a $t_\beta>t_0$ such that $\gamma_{x}(-t_{\beta}) \in V$. Moreover, by forward invariance
$\gamma_{x}(-t_{0})=\varphi^{t_{\beta}-t_0}\bigl( \gamma_{x}(-t_{\beta})\bigr) \in V$, which contradicts the assumption that $\gamma_{x}(-t_{0})\not\in V$. We conclude that $\gamma_{x}\subset V$ and thus $\gamma_{x}\subset A$ since $A=\Inv(V)$. 
    \qed
\end{proof}

\begin{theorem}
    \label{comparetoCR}
    Let  $\scrR\subset X\times X$ be the relation defined in Definition \ref{defnofrelR}. For $x\neq x'$  and $x\in\omega(X)$,
  $(x,x')\in \scrR$ if and only if $x\le x'$.
  \end{theorem}
\begin{proof}
Let $(x,x')\in \scrR$ with $x\neq x'$. Then, by  Definition \ref{defnofrelR} there exist recurrent components $\xi_-,\xi'_+\in \RC(\varphi)$ such that 
$
\xi_-\le \xi'_+
$
 and a complete orbit $\gamma_{x}$
with $\omega(x')\subset \xi'_+$ and $\alphaOg\subset \xi_-$ 
By definition  $\rmPhi(\xi_-) = I_-\subset I'_+=\rmPhi(\xi'_+)$, and 
\begin{equation}
    \label{assum1}
\omega(x') \subset \xi'_+= \left(\bigcap_{A\in {I'_+}^c}A \right) \bigcap 
 \left( \bigcap_{A\in I'_+}A^*\right),
\end{equation}
and 
\begin{equation}
    \label{assum2}
\alphaOg \subset \xi_-=\left(\bigcap_{A\in {I_-}^c}A \right) \bigcap 
 \left( \bigcap_{A\in {I_-}}A^*\right). 
\end{equation}
Therefore, $\omega(x')\subset A$ for all $A\in {I'_+}^c$, and since ${I'_+}^c\subset I_-^c$, we conclude that 
$\alphaOg\subset A$ for all $A\in {I'_+}^c$.
Since the 
orbital alpha-limit set $\alphaOg$ is contained in $A$ for all attractors $A\in {I'_+}^c$, Lemma \ref{containment1} implies that the orbit $\gamma_{x}$ is contained in all $A\in {I'_+}^c$. Thus 
$x\in A\subset U$ for all $U\in \sANbhd(\varphi)$ with $\omega(U)=A\in {I'_+}^c$.
Now, suppose $(x,x')\in \scrR$ with $x\neq x'$ and $x'\in U$ for some $U\in \sANbhd(\varphi)$. Then $\omega(x')\subset \omega(U)=A$, and $A\in {I'_+}^c$ by \eqref{assum1}. By the previous considerations also $x\in U$.
Consequently,
 $x'\in U$ implies $x\in U$ for all $U\in \sANbhd(\varphi)$, which yields $x\le x'$.

As for the converse, we argue as follows.  
Suppose $x\le x'$ 
with $x\neq x'$  and $x\in\omega(X)$.
By assumption, for every $U\in \sANbhd(\varphi)$, $x\in U^c$ implies
$\alphaOg\subset \alpha(U^c)=A^*$, and $x'\in A^*\neq \varnothing$ by Lemma \ref{charofpreo}. 
By the forward invariance of $A^*$,
we deduce that $\omega(x')\subset A^*$. 
By Theorem \ref{Conleydecomp} there exists a recurrent component $\xi'_+$ such that $\omega(x')\subset \xi'_+$, 
and there exists a recurrent component $\xi_-$ such that $\alphaOg\subset \xi_-$.
By the inclusions for $\omega(x')$ and $\alphaOg$ in \eqref{assum1} and \eqref{assum2}, we conclude that if $x\in U^c$, then $\alphaOg\subset A^*\in I_-$, and
$\omega(x')\subset A^*\in  I'_+$, which implies $I_-\subset I'_+$. Using the isomorphism $\rmPsi$ in via Theorem \ref{cor:iso} this proves that $\xi_-\le \xi'_+$.
    \qed
\end{proof}

\begin{remark}
\label{ext1212}
    If $\omega(X)=X$, then 
    Theorem \ref{comparetoCR} implies that the reflexive closure  $\scrR^=$ is equal to the preorder $(X,\le)$ and yields the partial order $(\SC(\varphi),\le)$. 
    If $\omega(X)\neq X$, then $(X,\le)$ extends $\scrR^=$, but
    coincides for all pairs $(x,x')$ with $x\in \omega(X)$. In some cases one can prove that $\scrR^=$ and $(X,\le)$  coincide. For instance assume 
    that for every $x\not\in \omega(X)$ there exists a $\tau>0$ such that $\varphi^{-\tau}(x)=\varnothing$. Then,
    $\alpha(x)=\varnothing$ for all $x\not\in \omega(X)$. Let $x\not\in \omega(X)$ and $x'\in X$. By assumption $V^c=\{x\}$ is a repelling neighborhood, and
    $V\in\sANbhd(\varphi)$ with $\alpha(V^c)=\varnothing$,   $x\in V^c$ and $x'\notin V^c$, which proves that $x\not\le x'$ for all $x'\in X$.
In a $T_1$-topological a similar statement can be proved. Since $\{x\}$ is closed, $X\smin\{x\}$ is an attracting neighborhood for $\omega(X)$, and thus $V^c=\{x\}$ is a repelling neighborhood with $\alpha(V^c)=\alpha(x)=\varnothing$.This again shows that $x\not\le x'$ for all $x\not\in \omega(X)$ and $x'\in X$.
\end{remark}

\begin{remark}
    \label{no-equiv}
    Consider the case $x\in \omega(X)\smin \sR(\varphi)$, $x'\not\in \omega(X)$, and $x\le x'$. By assumption $x\not \in A\cup A^*$ for some $A\in \sAtt(\varphi)$.  By Theorem \ref{comparetoCR} this implies $\alphaOg\subset A^*$ and $\omega(x')\subset A^*$. 
    Indeed, $x\le x'$ if and only if $\alphaOg\subset \xi_- \le \xi'_+\supseteq \omega(x')$. Therefore, 
     $A\in I_-\subset I'_+$, 
    which yields $\omega(x')\subset A^*$.
    Let $U\in \sANbhd(\varphi)$ with $\omega(U)=A$.  Since $\omega(x')\subset A^*$, Lemma \ref{asympattrep1221} implies that 
    $x'\in A^*\subset U^c$ and $\omega(x)\subset A$.
    Suppose $x'\le x$. Then, by Lemma \ref{charofpreo}, $x'\in U^c$ implies
     that $x\in \alpha(U^c)=A^*$, and therefore $\omega(x)\subset A^*$, a contradiction, which proves that $x'\not\le x$. This shows, in particular, that an equivalence class $\xi \in\SC(\varphi)$ is either contained in $\omega(X)$ or in $X\smin\omega(X)$.
\end{remark}

The reflexive points in $\scrR$ correspond exactly with the recurrent points, which may be encoded by the cospan of posets:
\begin{equation}
    \label{emb1}
   X \xtwoheadrightarrow{~~} \SC(\varphi) \xhookleftarrow{~~~~~~~} \RC(\varphi),
\end{equation}
which is called the \emph{recurrence cospan}. 
The inclusion as sets follows from the definitions of $\RC(\varphi)$ and $\SC(\varphi)$.  Indeed, by Remark \ref{onlyequiv} an equivalence class for $(\sR(\varphi),\le)$ defines an equivalence class for $(X,\le)$. If $\xi\le \xi'$ for some $\xi,\xi'\in \RC(\varphi)$, then
$\xi'\subset A$ implies $\xi\subset A$ for all $A\in \sAtt(\varphi)$, and consequently $\xi' \subset U$
implies $\xi\subset U$ for all $\omega(U) = A$ and all $A\in \sAtt(\varphi)$, which proves that $\xi\le \xi'$
in $\RC(\varphi)$ if and only if $\xi \le \xi'$ in $\SC(\varphi)$.
The set $X$ in \eqref{emb1} is regarded as unordered set. The preorder $(X,\le )$ in Definition \ref{pro123} can be obtained from $x\mapsto [x]$ as follows: $x\le  x'$ if and only if $[x] \le  [x']$.

 Since the reflexive closure of $\scrR$ can be defined via the preorder $(X,\le)$, cf.\ Thm.\ \ref{comparetoCR} and Rem.\ \ref{ext1212}, we can utilize the cospan in \eqref{emb1}
to provide an alternative for definition $\scrR$, which is solely based on the order theoretic information of the homomorphism $\omega\colon\sANbhd(\varphi)\twoheadrightarrow\sAtt(\varphi)$:

\begin{description}
    \item[(i)]  $(x,x') \in \scrR$, $x\neq x'$ if and only if  $[x] \le  [x']$ in $\SC(\varphi)$;
      \item[(ii)] $(x,x)\in \scrR$ if and only if $[x]\in \RC(\varphi)$.
\end{description}

This approach yields an order-theoretic method to define the Conley relation which extends beyond the setting of Hausdorff flows in \cite{Conley} to arbitrary dynamical systems with no conditions on the topology of the phase space.
In  Section \ref{noncompact}, we discuss a Hausdorff compactification of $\scrR$ which allows a theory of recurrence in this more general setting.

\begin{remark}
\label{singlepoint}
By Remark \ref{onlyequiv} for every class $\xi \in \SC(\varphi)\smin \RC(\varphi)$ with $\xi=[x]$ for some $x\in \omega(X)$, it holds that $\xi = \{x\}$,  cf.\ Ex.\ \ref{examcomp}.
\end{remark}

\begin{remark}
    \label{otherattnebhd}
    If use the lattice of trapping regions, $\sANbhd_+(\varphi)$, then the lattice of attractors remains unchanged, and we obtain Diagram \eqref{dualitydiagno12} with $\sANbhd(\varphi)$ and $\sRNbhd(\varphi)$ replaced by trapping regions $\sANbhd_+(\varphi)$
    and repelling regions $\sRNbhd_-(\varphi)$ respectively. This implies that \eqref{emb1} only changes the partial order on $\SC(\varphi)$ as an extension of the order defined in Definition \ref{pro123}. 
   Indeed, if we define $(X,\le)$ using trapping regions, the forward invariance has the following implication. Suppose $x'\in U\in \sANbhd_+(\varphi)$,
    then every point $x=\varphi^t(x')$, for some  $t\ge 0$, is contained in $U$.
    This implies that $x\le x'$. In contrast, using attracting neighborhoods, for two points $x,x'$ satisfying $x=\phi^\tau(x')$ for some $\tau>0$ and $x\in \omega(X)\smin \sR(\varphi)$, and thus $x'\in X\smin\sR(\varphi)$, it holds that $\omega(x')=\omega(\varphi^\tau(x'))=\omega(x)$. Suppose $x\le x'$, then there exist $\xi_+<\xi_-\le \xi'_+$ such that $\omega(x')=\omega(x)\subset\xi_+< \xi_-\le \xi'_+$ and $\omega(x')\subset\xi'_+ $,
    which is a contradiction, and therefore
    $x\not\leq x'$.
    The induced relation on $X$ is denoted by $\scrS$ and $\scrR\subset \scrS\subset X\times X$. This relation contains information about individual orbits, whereas the Conley relation does not. The relation $\scrS$ may regarded  as a description of the action of $\varphi$ that discards the notion of time.
\end{remark}

\subsection{Order properties of $\sfS\sANbhd(\varphi)$}
\label{orderprops12}
Similar to  the map $\rmPhi$ in \eqref{Phimap}, we define the map $\rmXi\colon\SC(\varphi) \to\sfS\sANbhd(\varphi)$
given by
\[
\xi \mapsto \{U\in\sANbhd(\varphi)~|~\xi\cap U=\varnothing\}=\{U\in\sANbhd(\varphi)~|~\xi\not\subset U\}.
\]

\begin{lemma}\label{lem:Xi}
The map $\rmXi\colon\SC(\varphi) \to\sfS\sANbhd(\varphi)$ is well-defined.
\end{lemma}
\begin{proof}
Follows the same argument as the proof of Lemma \ref{lem:Phi}.
\qed
\end{proof}

Similar to the definition of $\rmPsi$ in \eqref{Psimap}, we define the map $\rmTheta\colon\sfS\sANbhd(\varphi) \to\sSet(X)$ given by
\[
 I \mapsto \left(\bigcap_{U\in I^c}U \right) \bigcap 
 \left( \bigcap_{U\in I}U^c\right).
\]
By the definition of $\sim$ and $\rmTheta$, we have
$x,y\in\rmTheta(I)$ implies that $x,y\in U$ for all
$U\in I^c$ and $x,y\in U^c$ for all $U\in I$. Therefore, if $\rmTheta(I)\neq\varnothing$, then $\rmTheta(I)\in \SC(\varphi)$. 
Note that $\rmTheta(I)$ may be void for some $I\in \sfS\sANbhd(\varphi)$, since we cannot apply Lemma \ref{nonemptylimit} and the Cantor intersection theorem in this setting.

\begin{theorem}\label{thm:bijection12}
The map $\rmXi\colon\SC(\varphi) \to\sfS\sANbhd(\varphi)$ is injective
and the left inverse is given by $\rmXi^{-1}\big|_{\rmTheta(\sfS\sANbhd)}=\rmTheta.$
\end{theorem}
\begin{proof}
We need to show that 
$\rmTheta\circ\rmXi=\id_{\SC}$.
Let $\xi\in\SC(\varphi)$. Then $I=\rmXi(\xi)=\{U\in\sANbhd(\varphi)~|~\xi\not\subset U\}.$
Thus, 
\[
\rmTheta(I)=\left(\bigcap_{U\notin I}U \right) \bigcap 
 \left( \bigcap_{U\in I}U^c\right)
 =\left(\bigcap_{\xi\subset U}U \right) \bigcap 
 \left( \bigcap_{\xi\not\subset U}U^c\right)=\xi.
\]
By definition $\xi\subset\rmTheta(I)$ so that  the last equality follows from the fact that $\rmTheta(I)$ is a strong component.
\qed
\end{proof}

\begin{corollary}\label{cor:iso12}
$\rmXi\colon\SC(\varphi) \hookrightarrow\sfS\sANbhd(\varphi)$ is an   order-embedding.
\end{corollary}
\begin{proof}
As in the proof of Theorem \ref{cor:iso}, one can show that $\xi\le\xi'$
     if and only if $\rmXi(\xi)\le \rmXi(\xi')$. The arguments are the same as in the proof of Theorem \ref{cor:iso}.
     \qed
\end{proof}

For a topology on $\SC(\varphi)$ there are various choices. One is the quotient topology induced by $(X,\scrT)$ and is denoted by $(\SC(\varphi),\scrT_\sim)$.
Another natural topology is defined by the subspace topology of $\rmXi(\SC(\varphi))$ in
$\sfS\sANbhd(\varphi)$ and is denoted by $(\SC(\varphi),\scrT_\sfS)$.
The latter makes 
\[
\rmXi\colon (\SC(\varphi),\scrT_{\sfS})\xhookrightarrow{~~~~~~} (\sfS\sANbhd(\varphi),\scrT_{\sfS\sANbhd}),
\]
a topological order-embedding.
As before, we  compare both topologies.
For the strong components we have the following commutative diagram:
\begin{equation}
\label{rc88}
\begin{diagram}
\node{X}\arrow{s,l,A}{\pi}\arrow{se,l}{f}\\
\node{\SC(\varphi)}\arrow{e,l,J}{\rmXi}\node{\sfS\sANbhd(\varphi)}
\end{diagram}
\end{equation}

A basic open set in $\sfS\sANbhd(\varphi)$ is given by 
$C=j(U)\smin j(U')$
for some pair of attracting neighborhoods $U,U'\in \sANbhd(\varphi)$. Then, $f^{-1}(C) =
f^{-1}\bigl( j(U)\smin j(U')\bigr) = U\cap U'^c=T$.
Regardless of the choice of either open or closed attracting neighborhoods, or arbitrary attracting neighborhoods, the set $T$ is not open in general. We therefore cannot conclude continuity of $f$ in general. Consequently $\rmXi\colon (\SC(\varphi),\scrT_{\sim})\to (\sfS\sANbhd(\varphi),\scrT_{\sfS\sANbhd})$ is generally not continuous  and the quotient space $(\SC(\varphi),\scrT_{\sim})$ is not Hausdorff in general.
Therefore the spaces $(\SC(\varphi),\scrT_{\sim})$ and $(\SC(\varphi),\scrT_{\sfS})$ are not  homeomorphic in general.
This is in sharp contrast with the recurrent components $\RC(\varphi)$.

In order to view the cospan \eqref{emb1} topologically we need to reconsider the commutative diagram in \eqref{rc88}.
The relevance of the the topological space $(X,\scrT)$ and the continuity of $\varphi$ is reflected in the lattices $\sANbhd(\varphi)$ and $\sAtt(\varphi)$. 
The natural step is to regard $X$ as discrete space. 
With the discrete topology on $X$,  the maps $f$ and $\pi$ are continuous.
 to the lattice embedding $\iota\colon \sANbhd(\varphi) \rightarrowtail\sSet(X)$ is 
the continuous, order-preserving surjection $\sfS\iota\colon \beta X \twoheadrightarrow \sfS\sANbhd(\varphi)$, where
the space $\beta X \cong \sfS\sSet(X)$ is the \v{C}ech-Stone compactification of the discrete space $X$. 
Since $\beta X$ is compact and $\sfS\sANbhd(\varphi)$ is Hausdorff, the map $\sfS\iota$ is a quotient map and thus closed, cf.\ \cite[Sect.\ 2.2]{gehrke2023topological}.
The latter implies that $\pi\colon X\to \SC(\varphi)$ is a quotient map with $\SC(\varphi)$ equipped with the topology $\scrT_\sfS$.

For the set $\sfS\sSet(X)$ define the map $i\colon X \to \sfS\sSet(X)$ given by
\[
x\mapsto \{U\in \sSet(X)\mid x\notin U \},
\]
which is a well-defined map. Moreover, define the  map $j\colon \sfS\sSet(X) \to \sSet(X)$ given by
\[
 I \mapsto \left(\bigcap_{U\in I^c}U \right) \bigcap 
 \left( \bigcap_{U\in I}U^c\right).
\]
As before,  $j(I)$ may be void for some $I\in \sfS\sSet(X)$.
By construction the map
$i\colon X\hookrightarrow \sfS\sSet(X)\cong \beta X$ is a topological embedding with dense image. 
In the following theorem we assume that $X$ has the discrete topology and $\SC(\varphi)$ is equipped with the topology $\scrT_\sfS$.

\begin{theorem}\label{fundcomp12}
The    diagram 
\begin{equation}
    \label{compdiag12}
    \begin{diagram}
    \node{\beta X}\arrow{e,l}{\sfS\iota}\node{\sfS\sANbhd(\varphi)}\node{\sfS\sAtt(\varphi)}\arrow{w,l}{\sfS\omega}\\
    \node{X}\arrow{n,l,A}{i}\arrow{e,l}{\pi}\node{\SC(\varphi)}\arrow{n,r,A}{\rmXi}\node{\RC(\varphi)}\arrow{w,l}{\supseteq}\arrow{n,lr,<>}{\rmPhi}{\cong}
    \end{diagram}
\end{equation}
commutes and all maps are continuous and  order-preserving.
\end{theorem}
\begin{proof}
We start with the left part of the diagram. Note that $\sfS\iota =\iota^{-1}$, and $\iota^{-1}(I) = \{U\in I\mid U\in \sANbhd(\varphi)\}$.
Let $x\in X$. Then, $x \mapsto I=\{U\in \sSet(X)\mid x\notin U\} \mapsto \{U\in \sANbhd(\varphi)\mid x\notin U\}$. On the other hand $x\mapsto [x]\mapsto 
\{U\in \sANbhd(\varphi)\mid [x]\cap U=\varnothing\}$. If $x\notin U$, then $[x]\not\subset U$, which shows commutativity.
As for the right square, we argue as follows. The map $\sfS\omega=\omega^{-1}$ is given by
$I\mapsto \omega^{-1}(I) = \bigcup_{A\in I} \{U\in \sANbhd(\varphi)\mid \omega(U) =A\}$.
Therefore, $\xi \mapsto \rmPhi(\xi) = \{A\in \sAtt(\varphi)\mid \xi\cap A=\varnothing\}
\mapsto \{U\in \sANbhd(\varphi)\mid \omega(U)=A\in I, \xi\cap A=\varnothing\}$,
which is equal to $\{U\in \sANbhd(\varphi)\mid \xi\cap U=\varnothing\} = \rmXi(\xi)$.
    \qed
\end{proof}

\begin{remark}
\label{altConleydecomp}
Note that
using the partial order on $\SC(\varphi)$, 
    given $\xi_+<\xi<\xi_-$,
    we obtain $\rmXi(\xi_+)\subsetneq\rmXi(\xi)\subsetneq\rmXi(\xi_-)$ in $\sfS\sANbhd(\varphi)$ via the order embedding $\rmXi$.
    The latter is an order on the prime ideal space which serves as a compactification of $\SC(\varphi)$, cf.\ Sect.\ \ref{compti12}.
\end{remark}


\section{ Recurrence in the noncompact setting}
\label{noncompact}
The focus of this paper is  dynamical systems on compact spaces. Due to the compactness of the phase space $X$, and mild conditions on $\varphi$  as a continuous and proper system, there is a full duality between attractors and  repellers. 
Attractors are compact, invariant sets, 
which implies that the recurrent set is also a compact, invariant set. 
The set of recurrent components forms a compact Hausdorff space with respect to the induced quotient topology ($X$  is \emph{not} required to be Hausdorff!). 

If the compactness hypothesis on $X$ is dropped,
we can still define attracting and repelling neighborhoods as well as trapping and repelling regions, as in Section~\ref{ARpairsgencase}. However, since omega-limit sets may be empty, and not even invariant necessarily, the existence of nontrivial attractors is not guaranteed. 
    In Section~\ref{reccompsect}, we have used lattice theory to develop a theory of recurrence that applies to all dynamical systems. In Sections~\ref{reccontcomp}, \ref{profcc}-\ref{attnbhds}, we characterize this approach in the case of systems on a compact space and establish agreement with standard definitions of (chain) recurrence when they apply. In this section, we combine the general results in Section~\ref{reccompsect} and the order-theoretic methods in Sections~\ref{profcc}-\ref{attnbhds} to obtain a compactification of the recurrent set as well as a compactification of the dynamical system in terms of order models.

\vskip 6pt
\noindent\textbf{The results in this section hold for
 dynamical systems with no conditions on $\varphi$ or the topology of the phase space 
$X$. }

\subsection{Embedding $\RC(\varphi)$ in $\sfS\sARpair(\varphi)$}
\label{emdofPr}
The results in this section hold for
 dynamical systems with no conditions on $\varphi$ and on the topology of the phase space 
$X$. 
Recall the definition of $\sR(\varphi)$ in \eqref{eqn:RS} and $\RC(\varphi)$ as the poset of partial equivalence classes in Definition \ref{def:crs}.
For $\xi\in \RC(\varphi)$ we define the following set:
\[
\begin{aligned}
\rmPhi(\xi) &= \{P\in \sARpair(\varphi) \mid \xi\cap A=\varnothing\};\\
&= \{P\in \sARpair(\varphi) \mid \xi\subset R\}.
\end{aligned}
\]
Associated with $\rmPhi(\xi)$ we define the map $h_\xi\colon\sARpair(\varphi)\to {\bf 2}$ by 
\[
h_\xi(P) = \begin{cases}
  0    & \text{ if } \xi\cap A=\varnothing;\\ 
   1   & \text{ if } \xi\subset A.
\end{cases}
\]

\begin{lemma}
\label{cr2}
The map $h_\xi$ is a lattice homomorphism for every $\xi\in \RC(\varphi)$, and thus $\rmPhi(\xi) = h_\xi^{-1}(0)$ is a prime ideal for $\sARpair(\varphi)$ for all $\xi \in \RC(\varphi)$.
\end{lemma}

\begin{proof}
The proof is identical to the proof of Lemma \ref{lem:Phi}
by employing Lemma \ref{lem:wedge-cap-forAR}.
\qed
\end{proof}

We can interpret $\rmPhi(\xi)$ as a map between posets. Let $\sfS\sARpair(\varphi)$ be the poset of prime ideals  ordered by inclusion.

\begin{lemma}
\label{cr3}
The map 
$\rmPhi\colon \RC(\varphi) \to \sfS\sARpair(\varphi)$  
is an order-embedding.
\end{lemma}

\begin{proof}
Suppose $\xi\le \xi'$, ie.\ for all $R\in \sARpair(\varphi)$, $\xi\subset R$ implies $\xi'\subset R$. 
Then $\rmPhi(\xi) \subset \rmPhi(\xi')$,
which proves that $\rmPhi$ is order-preserving. Moreover, $\rmPhi(\xi)\subset \rmPhi(\xi')$ if and only
if $\xi\le \xi'$, since $\rmPhi(\xi)\subset \rmPhi(\xi')$ implies that $\xi\le \xi'$.
It remains to show that $\rmPhi$ is injective.
Let $I\in \sfS\sARpair(\varphi)$ be a prime ideal and define, motivated by 
\eqref{nonempty}, the set
\[
\rmPsi(I) := \Biggl( \bigwedge_{P\in I^c} A\Biggr) \cap  \Biggl( \bigcap_{P \in I} R\Biggr).
\]
Let $I=\rmPhi(\xi)\in \sfS\sARpair(\varphi)$, then as in the proof of Lemma \ref{thm:bijection},
\[
\rmPsi(I) 
= \Biggl( \bigwedge_{P\in I^c} A\Biggr) \cap  \Biggl( \bigcap_{P \in I} R\Biggr) = \Biggl( \bigwedge_{\xi \subset A} A\Biggr) \cap  \Biggl( \bigcap_{\xi\subset R} R\Biggr) =\xi,
\] 
where the latter equality follows from Lemma \ref{charofcompofR}.
We conclude that $\rmPsi$ defines a left-inverse for $\rmPhi$, and thus $\rmPhi$ is injective. 
\qed
\end{proof}

\begin{remark}
    \label{invisint}
    For every prime ideal $I=\rmPhi(\xi)\in \sfS\sAR$ with $\xi\in \RC(\varphi;\sAR)$, the proof of Lemma \ref{cr3} implies that the expression $\rmPsi(I)$ is the recurrent component $\xi$.
    As in Section \ref{Morsesets12}, $\bigcap_{P\in I^c}A$ is invariant, and therefore
    $\xi= \left(\bigcap_{P\in I^c}A \right) \cap 
 \left( \bigcap_{P\in I}R\right)$,  which proves \eqref{eqn:RS2}.
\end{remark}

\subsection{Compactification via  Priestley spaces}
\label{compti12}

 space of prime ideals  $\sfS\sARpair(\varphi)$ is a Priestley space that is compact, Hausdorff, and zero-dimensional.
Let $\sfS_0\sARpair(\varphi):=\rmPhi\bigl(\RC(\varphi)\bigr)$ with the subspace topology $\scrT_{\sfS_0\sARpair}$ of $\sfS\sARpair(\varphi)$. Then $\scrT_\sfS=\rmPsi \scrT_{\sfS_0\sARpair}$ is a topology so that $(\RC(\varphi),\le,\scrT_\sfS)$ is an ordered topological space.
Another topology on $\RC(\varphi)$ is the quotient topology $\scrT_\sim$ induced by the projection $\pi$
so that $(\RC(\varphi),\le,\scrT_\sim)$ is also an ordered topological space.
We now compare these two topologies.

Consider the diagram:
\begin{equation*}
\label{rc5}
\begin{diagram}
\node{\sR(\varphi)}\arrow{s,l,A}{\pi}\arrow{se,l,A}{g}\\
\node{\RC(\varphi)}\arrow{e,l,<>}{\rmPhi}\node{\sfS_0\sARpair(\varphi) }\arrow{e,l,J}{}\node{\sfS\sARpair(\varphi)}
\end{diagram}
\end{equation*}
where  
$\rmPhi$ is a bijection. By construction $\pi$ is a quotient map, and 
if $g$ is continuous, then $\rmPhi\colon \RC(\varphi) \to \sfS_0\sARpair(\varphi)$ is a bijective, continuous map, but not necessarily a homeomorphism. Since $\sfS\sARpair(\varphi)$ is Hausdorff, so is $\sfS_0\sARpair(\varphi)$, and therefore $(\RC(\varphi),\scrT_\sim)$ is a Hausdorff space, cf.\ \cite[Corollary 22.3]{Munkres}.
 If $\RC(\varphi)$ is compact, then $\rmPhi$ is a homeomorphism, cf.\ Theorem~\ref{thm:homeo}.

\begin{lemma}
\label{cr6a}
Let $U\in \sANbhd(\varphi)$ with $\Inv(U) = A$. If $x\in U\cap \sR(\varphi)$, then $x\in A$. Similarly, if $x\in U^c$  and $(A,R)\in\sARpair(\varphi)$, then $x \in R$.
\end{lemma}

\begin{proof}
If $x\in \sR(\varphi)$, then $x\in A$ or $x\in R$ for all $P\in \sARpair(\varphi)$. If $x\in U$, then $x\not \in
R$ and therefore $x\in A$. Since $U^c\cap A=\varnothing$,  $x\in U^c$ implies that $x\in R$.
\qed
\end{proof}

\begin{lemma}
\label{cr6}
The map $g\colon \sR(\varphi) \to \sfS_0\sARpair(\varphi)$ is continuous.
\end{lemma}
\begin{proof}
For the Priestley space $\sfS\sARpair(\varphi)$, the Priestley topology  consists of basic open sets of the form
$B=\{ I\in \sfS\sARpair(\varphi)\mid P\in I,~P'\not \in I\}$ for some $P,P'\in \sARpair(\varphi)$. To show that $g$ is continuous we prove that $g^{-1}(B)$ is open in $X$ and thus in $\sR(\varphi)$.
For $g^{-1}(B)$ we have:
\begin{equation*}
\begin{aligned}
g^{-1}(B)&=\pi^{-1}\bigl(\rmPhi^{-1}(B)\bigr)\\
&=\pi^{-1}\left(\{\xi\in \RC(\varphi) ~|~\xi\cap A=\varnothing,\xi\subset A'\}\right)\\
&=\{x\in \sR(\varphi)~|~x\in A\cap R'\},
\end{aligned}
\end{equation*}
due to the characterization of $\rmPhi^{-1}$ given by $\rmPsi$. Consider   attracting neighborhoods $U,U'$ with $A = \Inv(U)$ and $A'=\Inv(U')$ such that $U$ is open and $U'$ closed. By definition $W = U\cap U'^c\cap \sR(\varphi)$ is open and 
$g^{-1}(B) \subset W$. Let $x\in W$, then
by Lemma \ref{cr6a}, $x\in A\cap R'\cap \sR(\varphi)$ and thus $W\subset g^{-1}(B)$. Combining 
both inclusions yields $g^{-1}(B) = W$, which shows that $g^{-1}(B)$ is open, and therefore $g$ is continuous.
\qed
\end{proof}

Summarizing, the following theorem compares the induced quotient topology $\scrT_\sim$ and the topology $\scrT_\sfS$.
Note that a map between ordered topological spaces is called a \emph{topological order-embedding} if it both an order-embedding and an embedding of topological spaces.
\begin{theorem}
\label{cr7}
The space of recurrent components $(\RC(\varphi),\scrT_\sim)$ is a Hausdorff space, and there exists a continuous bijection from
$(\RC(\varphi),\scrT_\sim)$ to $(\RC(\varphi),\scrT_\sfS)$. In particular, the map
$\rmPhi\colon (\RC(\varphi),\scrT_\sim) \hookrightarrow (\sfS\sARpair(\varphi),\scrT_{\sfS\sARpair})$ is a continuous injection, and moreover
$\rmPhi\colon (\RC(\varphi),\scrT_\sfS) \hookrightarrow (\sfS\sARpair(\varphi),\scrT_{\sfS\sARpair})$ is a topological order-embedding.
 If the poset  $\RC(\varphi)$ is compact, then $\rmPhi\colon (\RC(\varphi),\scrT_\sim) \hookrightarrow (\sfS\sARpair(\varphi),\scrT_{\sfS\sARpair})$ is a topological order-embedding.
\end{theorem}
From Theorem \ref{cr7} and Lemma \ref{cr3} we have that
\[
(\RC(\varphi),\le,\scrT_\sim) \xhookrightarrow{~~~~~~}(\RC(\varphi),\le,\scrT_\sfS)\xhookrightarrow{~~~~\rmPhi~~~~} (\sfS\sARpair(\varphi),\subset,\scrT_{\sfS\sARpair})
\]
is a continuous order-injection, and  $\rmPhi$ is topological order-embedding. 
As described in the introduction,
$\scrT_\sfS$ is the natural choice for a topology on $\RC(\varphi)$. The topology on the phase space characterizes the asymptotic behavior of the system and determines the attractors and repellers. From that point, order theory provides a topological order-embedding of $\bigl(\RC(\varphi),\scrT_\sfS \bigr)$ into the compact Hausdorff space $\sfS\sARpair(\varphi)$. We now explore the how this compactification can describe dynamics.  

The definitions of the preorder $(X,\le)$ and the poset of strong components
$\SC(\varphi)$ use the attracting neighborhoods and are identical to Definitions \ref{pro123} and \ref{def:str123}. In the compact and proper, continuous case Remark \ref{onlyequiv} establishes that equivalence classes in $\sR(\varphi)$ are also equivalence classes in $\SC(\varphi)$. In the general setting the following result holds:
\begin{lemma}
    \label{compareequivclasses}
    If $x\sim x'$ in $\sR(\varphi)$, then $x\sim x'$ in $(X,\le)$.
\end{lemma}
\begin{proof}
Two points $x,x'\in \sR(\varphi)$ are equivalent if for all $P\in \sARpair(\varphi)$,
$x\in A^*$ if and only if $x'\in A^*$, or equivalently $x'\in A$ if and only if $x\in A$.
Since $A\subset U\in \sANbhd(\varphi)$ and $A^*\subset U^c$ for all $\omega(U)=A$,
it follows that $x\in U^c$ if and only if $x'\in U^c$, which proves that $x\sim x'$ in $(X,\le)$.
    \qed
\end{proof}

Lemma \ref{compareequivclasses} implies that every equivalence class in $\RC(\varphi)$ corresponds to a unique equivalence class in $\SC(\varphi)$ which yields the order-embedding
$\RC(\varphi)\hookrightarrow\SC(\varphi)$.
Dualizing Diagram \ref{dualitydiagno34}
 follows along the same argument as in Section \ref{attnbhds} and
yields the diagrams

\begin{equation}
\label{rc1010}
\begin{diagram}
\node{\sfS\sANbhd(\varphi)}\node{\sfS\sARpair(\varphi)}\arrow{w,l,L}{\sfS\varpi}\\
\node{\SC(\varphi)}\arrow{n,l,J}{\rmXi}\node{\RC(\varphi)}\arrow{w,l,L}{}\arrow{n,r,J}{\rmPhi}
\end{diagram}
\end{equation}
where the map $\varpi$ is defined in \eqref{dualitydiagno34}.
 The left commutative square in \eqref{compdiag12} also fits in \eqref{rc1010}, and the analysis of the above diagram is the same as in Section \ref{attnbhds}. Diagram \eqref{rc1010} can be used to define a relation $\scrR$, as displayed in Section \ref{strongcomp123}, generalizing the Conley relation.

\subsection{Examples and discussion of compact order models}
\label{examplesofcomp}
In this section, we observe that the top rows in Diagram \eqref{rc1010} maybe regarded as a compactification of the dynamical system in terms of an order model. 
We start with an example to illustrate the compactification.
\begin{example}
    \label{examcomp}
    Let $X=\R$ and $\varphi(t,x) = xe^t$, with $t\in \R$. Figure \ref{fig:comp}[top-left] illustrates the flow  with the poset of recurrent components $\RC(\varphi) =\{0\}$. The space $\sfS\sARpair(\varphi)$
    yields the compactification of $\RC(\varphi)$. For a point $x<0$ it holds that $\alpha(x) =\{0\}$ and $\omega(x) = \varnothing$. 
    Moreover, the strong component of $x$ is the singleton set, $\xi=\{x\}$,  which can be argued as follows.
    Suppose $x\sim x'$ with $x\neq x'$. This is equivalent to $x\in U$ if and only if $x'\in U$,
    or equivalently $x\in U^c$ if and only if $x'\in U^c$ for all $U\in \sANbhd(\varphi)$.
    Suppose $x\in U$ and $0\not\in U$. Define $\tilde U=U\smin\{x'\}$. Then, $\tilde U^c = U^c\cup \{x'\}$ so that
    $x\not\in \tilde U^c$ and $x'\in \tilde U^c$. Moreover, $\alpha(\tilde U^c) = 
    \alpha(U^c)\cup \alpha(x') = \alpha(U^c)\subset \Int U^c\subset \Int\tilde U^c$, which makes $\tilde U$ an attracting neighborhood,  a contradiction.
    This implies that $x\sim x'$ if and only if $x=x'$, so that $\xi=\{x\}.$
    The prime ideal corresponding to $\xi$ is given by
    $\rmXi(\xi)=I_x= \{U\in \sANbhd(\varphi)\mid \xi\cap U=\varnothing\}= \{U\in \sANbhd(\varphi)\mid x\not\in U\}\in \sfS\sANbhd(\varphi)$. 
        
    The origin as recurrent point corresponds to the prime ideal
    $J_3\in \sfS\sARpair(\varphi)$ given by 
    \[
    \begin{aligned}
    J_3 &= \rmPhi(\{0\}) = \bigl\{ P\in \sARpair(\varphi)\mid 0\not\in A\bigr\}\\
    &= \bigl\{(\varnothing,X), (\varnothing,L_-), (\varnothing,L_+),(\varnothing,0)\bigr\}
        \end{aligned}
    \]
    Moreover, the corresponding prime ideal in $\sfS\sANbhd(\varphi)$ is given by
    \[
    \begin{aligned}
      I_3 &= \varpi^{-1}(J_3) = \bigl\{U\in \sANbhd(\varphi)\mid \bigl(\omega(U),\alpha(U^c)\bigr)=(A,R)\in J_3\bigr\}\\
      &= \bigl\{ U\in \sANbhd(\varphi)\mid 0\not\in U\bigr\} = \rmXi(\{0\}).
    \end{aligned}
    \]
    Define the prime ideal $J_1 =\{(A,R)\in \sARpair(\varphi)\mid x\in R\}\in \sfS\sARpair(\varphi)$. It is important to emphasize that $J_1$ is \emph{not} in the image of $\rmPhi$, and the prime ideals are given by
    $J_1=\bigl\{(\varnothing,X), (\varnothing,L_-) \bigr\}$ and $I_1=\varpi^{-1}(J_1)$. Then $I_1$ consists of attracting neighborhoods $U=\varnothing$ and $\alpha(U^c)=L_-$, which implies that  $I_1\subsetneq I_x$.
    If $0\in U$, then $U=X$ and therefore $I_x\subsetneq I_3$.
    Consequently, 
    $I_1\subsetneq I_x\subsetneq I_3$. The case $x>0$ follows along the same lines. The above procedure provides a compactified order model for the dynamical system $\varphi$, as shown in Figure~\ref{fig:comp} 
\end{example}

\begin{figure}[h!]
  \includegraphics[width=\linewidth]{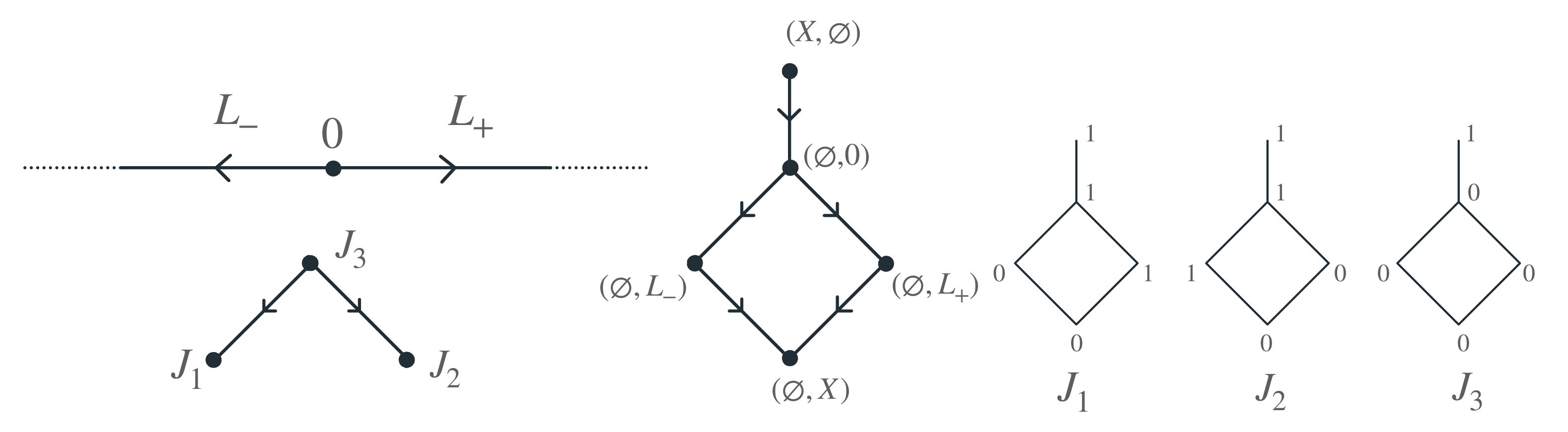}
  \caption{[top-left] The flow $\varphi$ in Example~\ref{examcomp}. [middle] The lattice of attractor-repeller pairs $\sARpair(\varphi)$. [right] The prime ideals in $\sfS\sARpair(\varphi)$. [bottom-left] The poset $\sfS\sARpair(\varphi)$ .}
  \label{fig:comp}
\end{figure}

The top row in \eqref{rc1010} defines a transitive relation $\overline\scrR$ on
the prime ideals $I\in \sfS\sANbhd(\varphi)$ as per Section \ref{strongcomp123}:
\begin{description}
    \item[(i)]  $(I,I') \in \overline\scrR$ for $I\neq I'$ if and only if  $I \le  I'$ in $\sfS\sANbhd(\varphi)$;
    \item[(ii)] $(I,I)\in \overline\scrR$ if and only if $I=\varpi^{-1}(J)$, $J\in \sfS\sARpair(\varphi)$.
\end{description}
As before, the reflexive closure of $\overline\scrR$ yields the poset  $\sfS\sANbhd(\varphi)$.

\begin{example}
    \label{parallel}
    Consider the flow $\varphi^t(x) = x+t$ for $x\in X=\R$.
    As in Example \ref{examcomp}, the lattice of attractor-repeller pairs and its prime ideal space are given in Figure \ref{fig:parallel}[middle]. The ideals $J_1$ and $J_2$ are computed using similar reasoning as in Example \ref{examcomp}. Figure \ref{fig:parallel}[left] shows the translation flow and the Hasse diagram of the compactified relation $\overline \scrR$. Points $x\in \R$ correspond to ideals $I_x$ which relate, or are `asymptotic to'
    the ideals $I_1$ and $I_2$. The relation $\overline\scrR$ characterizes the compactification of the system. The Conley relation in this case is void:  $\scrC=\varnothing$. 
\end{example}

\begin{figure}[h!]
  \includegraphics[width=\linewidth]{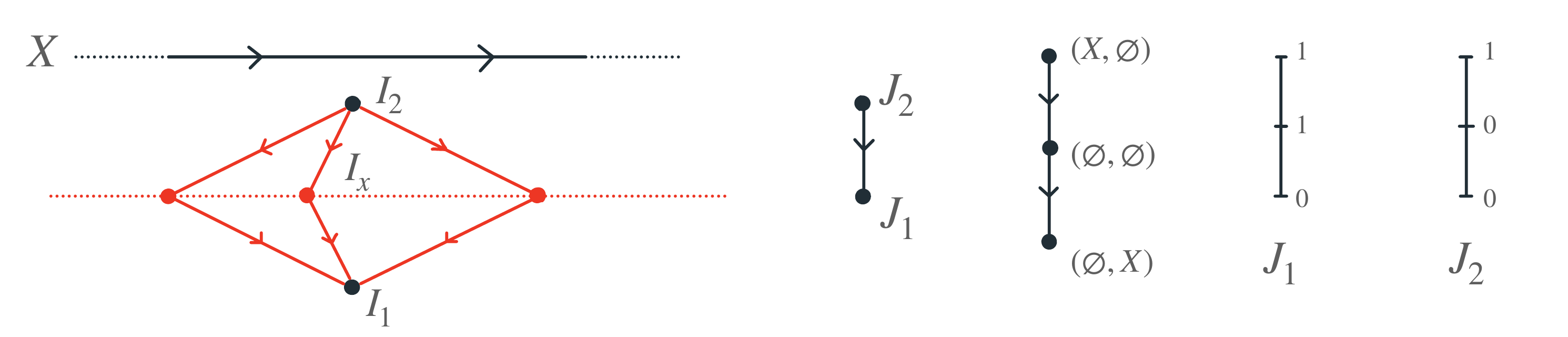}
  \caption{[top-left] The flow $\varphi$ in Example~\ref{parallel}. [bottom-left] The Hasse diagram of the relation $\overline\scrR$. [middle] The poset $\sfS\sARpair(\varphi)$ and  the lattice of attractor-repeller pairs $\sARpair(\varphi)$. [right] The prime ideals in $\sfS\sARpair(\varphi)$ .}
  \label{fig:parallel}
\end{figure}

We now give an outline  without proofs on how $\overline\scrR$ can be interpreted as a compactification of $\scrR$ defined
in a similar way in Section  \ref{strongcomp123}. The compactification procedure is a discussion and will be studied in more detail in future work.

In general $x\in X\smin \sR(\varphi)$ corresponds to a strong component $\xi=[x]$. The associated prime ideal is given by 
$I_x= \{U\in \sANbhd(\varphi)\mid \xi\cap U=\varnothing\}$ contained in $\sfS\sANbhd(\varphi)$. 
Define the homomorphism $h_x\colon \sARpair(\varphi)\to{\bf 2}$ given by
\[
h_x(P) = \begin{cases}
  0    & \text{ if } x \in R;\\ 
   1   & \text{ if } x\not\in R,
\end{cases}
\] 
which is equivalent to the definition in \eqref{hom1toasymp}.  In this framework, no limit sets are required.
The latter defines a 
prime ideal $J_+= h_x^{-1}(0) =\{ P\in \sARpair(\varphi)\mid x\in R\}\in \sfS\sARpair(\varphi)$. This yields a prime ideal $I_+\in \sfS\sANbhd(\varphi)$ given by 
\[
I_+=\varpi^{-1}(J_+) = \{ U\in \sANbhd(\varphi)\mid \bigl( \omega(U),\alpha(U^c)\bigr)=(A,R)\in J_+\}.
\]
As in Example \ref{examcomp}, it is important to emphasize that $J_+$ is not necessarily in the image of $\rmPhi$. 
For $U\in I_+$ it holds that $x\in R\subset U^c$, i.e. $x\not\in U$,  and thus $I_+\subsetneq I_x$. This `compactifies' the asymptotic behavior of a forward orbit $\gamma_x^+$. 
If $\omega(x)\neq \varnothing$, then $J_+ = \rmPhi(\xi_+)$ for some $\xi_+\in \RC(\varphi)$. 

Suppose $x$ allows a complete orbit $\gamma_x$.
As before we define the homomorphism $h_{\gamma_x}\colon \sARpair(\varphi)\to{\bf 2}$ given by
\[
h_{\gamma_x}(P) = \begin{cases}
  0    & \text{ if } \gamma_x\cap A=\varnothing;\\ 
   1   & \text{ if } \gamma_x\subset A,
\end{cases}
\] 
which is equivalent to the definition in \eqref{hom2toasymp}. From $h_{\gamma_x}$ we have the
prime ideal $J_-= h_{\gamma_x}^{-1}(0) =\{ P\in \sARpair(\varphi)\mid \gamma_x\cap A=\varnothing\}\in \sfS\sARpair(\varphi)$ and $I_- = \varpi^{-1}(J_-)$. For $U\in I_x$ it holds that $x\in U^c$, and thus $\gamma_x\cap A=\varnothing$ for $A=\omega(U)$. Consequently, $U\in I_-$ and thus $I_x\subsetneq I_-$.
If $x$ allows a complete orbit $\gamma_x$ such that $\alphaOg\neq \varnothing$, then 
$J_- = \rmPhi(\xi_-)$ for some $\xi_-\in \RC(\varphi)$.

Summarizing, given $x\in X\smin\sR(\varphi)$, then
\begin{equation}
    \label{fundinclforasymp}
I_+\subsetneq I_x\subsetneq I_-,
\end{equation}
which is a compactified version of Conley's decomposition theorem, cf.\ Theorem~\ref{Conleydecomp}. In the compact case this follows from the map $\rmXi$ as explained in Remark \ref{altConleydecomp}. Equation \eqref{fundinclforasymp} yields `limits' even though they do not necessarily arise as the image of $\rmXi$.
%
In general $\sfS\sARpair(\varphi)$ is a large compact ordered  space. Of importance in the compactification from a dynamics point of view are representations of  forward obits $\gamma_x^+$ for which $\omega(x)\neq \varnothing$ and complete orbits $\gamma_x$ for which at least one of the orbital limit sets $\omega(x)$ or $\alphaOg$ is nonempty.

The choice of attractor-repeller pairs $\sARpair(\varphi)$ is crucial in the above construction. The choice between  repelling neighborhoods or repelling regions has no effect on   $\sR(\varphi)$, or on its image in $\sfS\sARpair(\varphi)$. However, tt does affect the order structure on the ideals $I_x$.

The compactification procedure can also be performed using finite sublattices. Let $\sAR\subset\sARpair(\varphi)$ be a finite sublattice, and let $\sN=\sANbhd(\varphi;\sAR)\subset\sANbhd(\varphi)$ be the sublattice defined as the attracting neighborhoods $U$ for which the condition $\varpi(U)=\bigl(\omega(U),\alpha(U^c)\bigr)\in \sAR$ holds.  For $\RC(\varphi;\sAR)$ and $\SC(\varphi;\sN)$
we obtain the following commutative diagram:
\begin{equation}
    \label{compdiag89}
    \begin{diagram}
    \node{\beta X}\arrow{e,l}{\sfS\iota}\node{\sfS\sN}\node{\sfS\sAR}\arrow{w,l}{\sfS\omega}\\
    \node{X}\arrow{n,l,A}{i}\arrow{e,l}{\pi}\node{\SC(\varphi;\sN)}\arrow{n,r,A}{\rmXi}\node{\RC(\varphi;\sAR)}\arrow{w,l}{\supseteq}\arrow{n,r}{\rmPhi}
    \end{diagram}
\end{equation}
where $\SC(\varphi;\sN)$ is defined via isolating neighborhoods in $\sN$.
The embedding map $\sfS\varpi\colon\sfS\sAR\hookrightarrow \sfS\sN$  provides a $\sAR$-compactification for $\varphi$. This construction will 
 be subject of future study. If $\sN$ is chosen as a finite sublattice with $\varpi\colon \N\twoheadrightarrow \sAR$, then the construction yields a compactification of a Morse tessellation, cf.\ \cite{LSoA3}.
\begin{figure}
  \includegraphics[width=\linewidth]{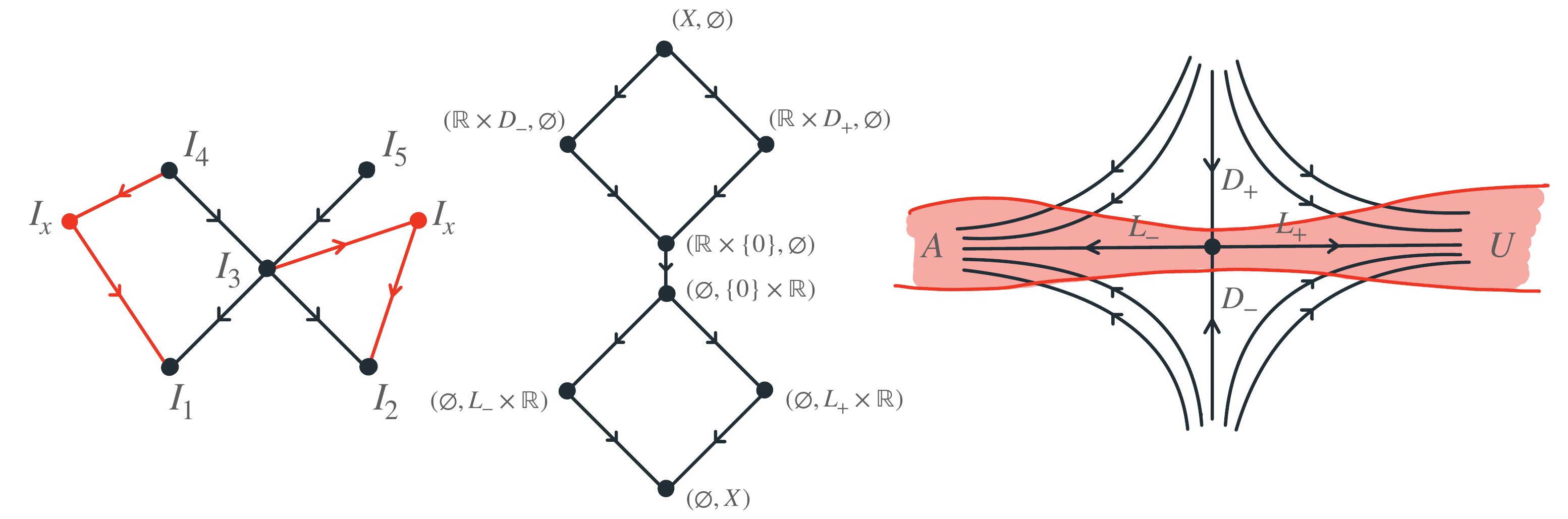}
  \caption{[right] The flow $\varphi$ in Example~\ref{exofsaddle} with an attracting neighborhood $U$ and attractor $A$ shown. [middle] The sub-poset $\sfS\sAR$ of prime ideals in $\sfS\sARpair(\varphi)$. [left] Some miscellaneous orders between prime ideals in $\sfS\sANbhd(\varphi)$. }
  \label{fig:comp23}
\end{figure}
The following example describes this procedure. 
\begin{example}
    \label{exofsaddle}
    Let $X=\R^2$ and $\varphi^t(x) = (x_1e^t,x_2e^{-t})$, where $x=(x_1,x_2)$, see Figure~\ref{fig:comp23}[right]. The recurrent set of $\varphi$ is given by $\sR(\varphi)=\{0\}$ and $\RC(\varphi)=\{0\}$.
    The lattice of attractor-repellers pairs is infinite. Indeed, every orbit disjoint from the coordinate axis may serve as attractor as well as repeller.
    In this example, it makes sense to choose a sublattice $\sAR\subset \sARpair(\varphi)$ such that $\sR(\varphi;\sAR)=\{0\}$.
    We only consider attractor-repeller pairs that are built from the coordinate axes, and the lattice $\sAR$ is shown in 
 Figure \ref{fig:comp23}[middle]. The prime ideals of $\sAR$ are given by
\begin{enumerate}
    \item[$\bullet$] $J_1\!=\!\{(\varnothing,X),\!(\varnothing,L_-\!\!\times\! \R)\}$;
    \item[$\bullet$] $J_2\!=\!\{(\varnothing,X),\!(\varnothing,L_+\!\!\times\! \R)\}$;
    \item[$\bullet$] $J_3\!=\!\{(\varnothing,X),\!(\varnothing,L_-\!\!\times\!\R),\!(\varnothing,L_+\!\!\times\!\R),\!(\varnothing,\{0\}\!\!\times\!\R)\}$;
    \item[$\bullet$] $J_4\!=\!\{(\varnothing,X),\!(\varnothing,L_-\!\!\times\!\R),\!(\varnothing,L_+\!\!\times\!\R),\!(\varnothing,\{0\}\!\!\times\!\R),\!(\R\!\times\!\{0\},\varnothing),\!(\R\!\times \!D_-,\varnothing)\}$;
    \item[$\bullet$] $J_5\!=\!\{(\varnothing,X),\!(\varnothing,L_-\!\!\times\!\R),\!(\varnothing,L_+\!\!\times\!\R),\!(\varnothing,\{0\}\!\!\times\!\R),\!(\R\!\times\!\{0\},\varnothing),\!(\R\!\times \!D_+,\varnothing)\}$,
\end{enumerate}
and the poset structure of $\sfS\sAR$ is given in Figure \ref{fig:comp23}[left].
For $x\in \R^2\smin\{0\}$ the prime ideals $I_x \in \sfS\sANbhd(\varphi;\sAR)$ satisfy $I_i\subsetneq I_x\subsetneq I_j$, via 
\eqref{fundinclforasymp}. Figure \ref{fig:comp23}[left] provides a compactification of $\sR(\varphi)$ with $\RC(\varphi)\hookrightarrow \sfS\sAR$.
In Figure \ref{fig:comp23}[right] we displayed a typical attractor that contributes to $\sfS\sAR$. 
\end{example}



\bibliographystyle{spmpsci}
\bibliography{priestley.bib}


\appendix 
\section{Omega-limit sets}
\label{eldyn}

Important for asymptotic behavior of a dynamical system  is the notion of omega-limit set.
For a dynamical system $\varphi$ (continuity not assumed) consider the net of sets
$\T^+ \to \sSet(X)$ given by $t\mapsto \varphi^t(U)\subset X$, $U\subset X$.
For the definition of limit set we follow \cite[Defn.\ 2.13]{Nishiguchi}.

 \begin{definition}
 \label{defn:limitset-a}
Let $U\subset X$ and let $\bigl(\varphi^t(U)\bigr)_{t\in \T^+}$ be a net of sets in $X$.
Then, the limit set
\begin{equation}
\label{eqn:om-char-10}
\omega(U)=\bigcap_{t\ge 0}\,\cl \bigcup_{s\ge t}\varphi^s(U).
\end{equation}
is called the 
 \emph{omega-limit set} $\omega(U)$ of $U$ with respect to $\varphi$.
 The latter is closed by definition, and contained in
 $\cl\rmGamma^+(U)$ with $\rmGamma^+(U):=\bigcup_{t\ge 0} \varphi^t(U)$.
\end{definition}

The definition of omega-limit set can be captured in the following equivalent characterization, cf.\ \cite{Nishiguchi}. It is important to point out that the notion of subnet in \cite{Nishiguchi} follows the definition of Kelley, cf.\ \cite[pp.\ 70]{Kelley}.\footnote{Let $(x_\alpha)_{\alpha\in A}$ be a net. 
A net $(y_\beta)_{\beta\in B}$ 
is a \emph{subnet} of $(x_\alpha)_{\alpha\in A}$ if
(i) there exists a map $h\colon B\to A$ such that for every $\alpha\in A$ there exists a $\beta_0\in B$ such that $h(\beta)\ge \alpha$ for all $\beta\ge \beta_0$, and 
(ii) $y_\beta = x_{h(\beta)}$ for all $\beta\in B$.
}

 \begin{proposition}
\label{lem:om-2}
The following statements are equivalent:
\begin{description}
\item [(i)] $y\in \omega(U)$;
\item [(ii)] there exists a subnet $\bigl(\varphi^{t_\alpha}(U)\bigr)_{\alpha\in \cA}$ of $\bigl(\varphi^t(U)\bigr)_{t\in \T^+}$ and $y_\alpha\in \varphi^{t_\alpha}(U)$
such that the net $(y_\alpha)_{\alpha\in \cA}$ converges to $y$.
\end{description}
\end{proposition}

\begin{proof} cf.\ \cite[Thm.\ 2.15]{Nishiguchi}.
\qed
\end{proof}

\begin{remark}
\label{class}
By the definition of subnet
for every $t>0$ there exists a $\alpha_t\in \cA$ such that $t_\alpha\ge t$ for all $\alpha\ge \alpha_t$. In this case we write $t_\alpha \to \infty$.
Omega-limit points are thus characterized as points $y$ for which there exist nets $(t_\alpha,x_\alpha)_{\alpha\in \cA}$ in $\T^+\times U$
with $t_\alpha\to \infty$  such that $y_\alpha= \varphi^{t_\alpha}(x_\alpha)\to y$. 
This is in accordance with the standard definition of omega-limit set in compact metric spaces, cf.\ \cite{Conley}.
\end{remark}

We now consider some properties of omega-limit sets. An immediate property following from the definition is:
\[
U\subset V \quad\implies \quad \omega(U) \subset \omega(V).
\]
In particular, $\omega(U\cap V)\subset \omega(U)\cap\omega(V)$.
Another  crucial property is the behavior  of omega-limit sets with respect to set-union.

\begin{lemma} 
\label{additive}
Let $U,V\subset X$. Then,
\begin{equation}
    \label{addom}
    \omega(U\cup V) = \omega(U) \cup \omega(V).
\end{equation}
\end{lemma}

\begin{proof}
By monotonicity $\omega(U) \cup \omega(V)\subset \omega(U\cup V)$. 
On other hand,
let $y\in \omega(U\cup V)$. By Proposition
\ref{lem:om-2}(ii) there exist nets $(t_\alpha,x_\alpha)$, with $t_\alpha\to \infty$ and $\varphi^{t_\alpha}(x_\alpha) \to y$.
We can choose subnets that are either contained in $U$ or in $V$, and therefore $y \in \omega(U)\cup \omega(V)$. The latter implies $\omega(U\cup V) \subset \omega(U) \cup \omega(V)$. 
\qed
\end{proof}

If we require $\varphi$ to be a continuous dynamical system, then
the continuity of the maps $\varphi^t$ guarantees that the omega-limit set is a forward invariant set. Indeed, for $r\ge0$,
\begin{equation}
    \label{theforwinv}
\begin{aligned}
\varphi^r(\omega(U)) &= \varphi^r\Bigl( \bigcap_{t\ge 0} \cl \bigcup_{s\ge t} \varphi^s(U)\Bigr) \subset 
\bigcap_{t\ge 0}  \varphi^r\Bigl( \cl \bigcup_{s\ge t} \varphi^s(U)\Bigr) \\
&\subset \bigcap_{t\ge 0} \cl \varphi^r\Bigl( \bigcup_{s\ge t} \varphi^s(U)\Bigr) = \bigcap_{t\ge 0} \cl  \bigcup_{s\ge t} \varphi^{s+r}(U) = \bigcap_{t'\ge r} \cl  \bigcup_{s'\ge t'} \varphi^{s'}(U)\\
&=\omega(U),
\end{aligned}
\end{equation}
which proves the forward invariance of $\omega(U)$ with respect to $\varphi$.

\begin{example}\label{noncompexom}
    In general an omega-limit set is forward invariant, but not necessarily invariant. Consider $X= \bigl(\R\times [-1,1)\bigr)\cup L$ with $L=\{(x,y)\mid x\ge 0,~y=1\}$ and the semiflow $\varphi$ given by the differential equation
    $\dot x=y$, $\dot y=(-x+y/10)(1-y^2)$. 
    The space $X\subset\R^2$ with the subspace topology is non-compact.
    For any point $(x,y)$, $x\neq 0$, $-1<y<1$ the omega limt set is given by $(\R\times\{-1\}) \cup L$, which forward invariant but not invariant. Note that $\varphi$ is not closed. Indeed,  $U=\{(-x,x)\mid x\ge 0\}\cap X$ is closed and $\phi^t(U)$ is not closed for $t$ sufficiently large.
\end{example}

There are additional properties of omega-limit sets in the case that subsets $U\subset X$ possess certain compactness properties.
Define $\rmGamma^+_\tau(U) := \bigcup_{t\ge \tau} \varphi^t(U)$.

\begin{proposition}[Topological properties]
\label{lem:props-al-om4}
Let $\varphi$ be continuous dynamical system on $X$.
Let $U\subset X$ with $\rmGamma^+_\tau(U)$ precompact for some $\tau\ge 0$.  Then,
\begin{description}
\item [(i)] $\omega(U)$ is  compact and closed;
\item [(ii)] $U\not = \varnothing$ implies $\omega(U)\neq \varnothing$;
\item [(iii)] $\omega\bigl(\omega(U)\bigr) \subset \omega(U)$;
\item [(iv)] $\omega(\cl U) = \omega(U)$.
\end{description}
\end{proposition}

\begin{proof}
\noindent{\bf(i)} 
For $t\ge \tau$ we have $  \cl \bigcup_{s\ge t} \varphi^s(U)   \subset 
\cl(\rmGamma^+_\tau(U))$ so that
\[
\omega(U) \subset \bigcap_{t\ge \tau} \cl \bigcup_{s\ge t} \varphi^s(U)
\subset \cl(\rmGamma^+_\tau(U))
\]
is compact as a closed subset of a compact space.

\noindent{\bf(ii)} 
If $U\not = \varnothing$, then $\omega(U)$ is the intersection of nested nonempty, compact, closed subsets of the compact space $\cl\rm\Gamma^+_\tau(U)$, which is nonempty by the Cantor intersection theorem.

\noindent{\bf(iii)} 
If $U$ is a forward invariant set, then $\omega(U) = \bigcap_{t\ge 0} \cl \varphi^t(U)\subset \cl U$. Since $\omega(U)$ closed and forward invariant, we obtain
$\omega(\omega(U)) \subset \cl \bigl(\omega(U)\bigr) = \omega(U)$.

\noindent{\bf(iv)} 
By continuity of $\varphi$ we have that 
\[
\bigcup_{s\ge t}\varphi^s(\cl U) \subset \bigcup_{s\ge t}\cl \varphi^s(U)\subset \cl\bigcup_{s\ge t}\varphi^s(U),
\]
which implies  $\omega(\cl U) \subset \omega(U)$. By monotonicity we have $\omega(U) \subset \omega(\cl U)$.
\qed
\end{proof}

\begin{remark}
If $X$ is compact, then every subset is precompact. 
In particular, by Proposition~\ref{lem:props-al-om4}(i), omega-limit sets are compact.
\end{remark}

In order to obtain invariance properties for omega-limit sets, we have the following result concerning images of infinite intersections.

\begin{theorem}
    \label{inverssysofinvsets}
    Let 
 $f\colon X\to Y$ be a map with compact fibers,\footnote{A map has compact fibers if the preimages $f^{-1}(y)$ are compact sets for all $y\in Y$.}  and let $\{U_\alpha\}_{\alpha\in \cA}$ be an inverse system of closed  sets, cf.\ Eqn.\ \eqref{ISclosed}. Then,
\begin{equation}
    \label{intcomminfHaus12}
f\Bigl( \bigcap_{\alpha\in \cA} U_\alpha \Bigr)
=\bigcap_{\alpha\in \cA} f(U_\alpha).
\end{equation}
\end{theorem}

\begin{proof}
Note that $f\bigl( \bigcap_{\alpha\in \cA} U_\alpha \bigr)
\subset \bigcap_{\alpha\in \cA} f(U_\alpha)$. Thus the condition
 $\bigcap_{\alpha\in \cA} f(U_\alpha)=\varnothing$, implies $f\bigl( \bigcap_{\alpha\in \cA} U_\alpha \bigr)=\varnothing$ so that \eqref{intcomminfHaus12} is satisfied.

Now suppose $\bigcap_{\alpha\in \cA} f(U_\alpha)\neq\varnothing$, and
let $y\in \bigcap_{\alpha\in \cA} f(U_\alpha)$.
Then, for all $\alpha\in \cA$  there exist $x_\alpha \in U_\alpha$  such that
$f(x_\alpha) = y$. 
Consequently, if we define
 $K = f^{-1}(y)$, then 
 $K_\alpha := K\cap U_\alpha \not = \varnothing$ for all $\alpha$. 
Since $K$ is compact by assumption, and $U_\alpha$ is closed, $K_\alpha$ is closed in $K$ and hence compact.
 Therefore, $\{K_\alpha\}_{\alpha\in \cA}$ is an inverse system of nonempty, compact, closed subsets of $K$.
Lemma \ref{CIT} implies that $\bigcap_{\alpha\in \cA} K_\alpha$ is also a nonempty, closed, compact subset of $K$.
Furthermore, $\bigcap_{\alpha\in \cA} K_\alpha \subset \bigcap_{\alpha\in \cA} U_\alpha$ and $\bigcap_{\alpha\in \cA} K_\alpha \subset K$.
Now choose $x\in \bigcap_{\alpha\in \cA} K_\alpha$, then $f(x) = y$, which implies that $ \bigcap_{\alpha\in \cA} f(U_\alpha)\subset f\bigl( \bigcap_{\alpha\in \cA} U_\alpha \bigr)$, which establishes \eqref{intcomminfHaus12}.
    \qed
\end{proof}

\begin{remark}
    \label{counterofintinv}
    Theorem \ref{inverssysofinvsets} is not true in general for inverse systems of closed sets.
\end{remark}

    \begin{proposition}
    \label{invofomega123}
Let $\varphi$ be continuous and proper dynamical system.
Then, 
\begin{description}
    \item[(i)] $\omega(U)$ is a closed, invariant set;
    \item[(ii)]  if $\varphi^t(U)\subset U$ for all $t\ge \tau\ge 0$, then $\omega(U) =\Inv\bigl(\cl U\bigr)$.
\end{description}
\end{proposition}

\begin{proof}
\noindent{\bf(i)} 
 By definition $\omega(U) = \bigcap_{t\ge 0} U_t$ where
$U_t = \cl \rmGamma^+_t(U)$.
The latter defines a nested family of closed sets, i.e. $U_{t'}\subset U_{t}$ for $t\le t'$.
Since $\varphi^r\colon X\to X$ has compact fibers for all $r\ge 0$,
 Theorem \ref{inverssysofinvsets} yields 
\begin{equation}
    \label{intcomminf}
\varphi^r\Bigl( \bigcap_{t\ge 0} U_t \Bigr)
=\bigcap_{t\ge 0} \varphi^r(U_t),\quad \forall r\ge 0.
\end{equation}
By assumption $\varphi^r$ are closed maps for all $r\ge 0$. Therefore,
analogously to \eqref{theforwinv}, and using \eqref{intcomminf} and the closedness of $\varphi^r$ we have
\[
\begin{aligned}
\varphi^r(\omega(U)) &= \varphi^r\Bigl( \bigcap_{t\ge 0} U_t\Bigr) = 
\bigcap_{t\ge 0}  \varphi^r(U_t) =\bigcap_{t\ge 0}  \varphi^r\Bigl( \cl \bigcup_{s\ge t} \varphi^s(U)\Bigr)\\
&= \bigcap_{t\ge 0} \cl \varphi^r\Bigl( \bigcup_{s\ge t} \varphi^s(U)\Bigr) = \bigcap_{t\ge 0} \cl  \bigcup_{s\ge t} \varphi^{s+r}(U) = \bigcap_{t'\ge r} \cl  \bigcup_{s'\ge t'} \varphi^{s'}(U)\\
&=\omega(U),\quad\forall r\ge 0,
\end{aligned}
\]
which proves the invariance of $\omega(U)$.

\noindent{\bf(ii)}  Note that $\omega(U) = \bigcap_{t\ge 0} \cl \bigcup_{s\ge t} \varphi^s(U) =  \bigcap_{t\ge \tau} \cl \bigcup_{s\ge t} \varphi^s(U)\subset \cl U$.
Let $S=\Inv(\cl U)$. Then,
$\omega(U) \subset S$ and 
\[ 
S =    \bigcap_{t\ge 0} \varphi^t(S)
\subset  \bigcap_{t\ge 0} \cl \varphi^t(S) = \omega(S)\subset \omega(\cl U) =\omega(U),
\]
which proves that $S=\omega(U)$ is the maximal  invariant set in $\cl U$.
\qed
\end{proof}

\begin{example}
    \label{ordertopnoninv}
    Consider the set  $X=\{-2,-1,0,+1,+2\}$ with order topology $\scrT_\le$, where $\le$ is the standard total order on the integers. This is a compact, $T_0$-topology  and the open sets are the up-sets $\{x\mid x\ge x_0\text{~for some~} x_0\in X\}$. 
    Consider the dynamical system given by the continuous map $f(x)= 0$, which has compact fibers but is not closed (not proper). The open set $U=\{0,+1,+2\}$ is an attracting neighborhood and  
$\omega(U)= \{-2,-1,0\}$ and $\Inv(U)=\{0\}$. The omega-limit set $\omega(U)$ is  forward invariant, but \emph{not} invariant, and not contained in $\Int U$. The omega-limit set does clearly not satisfy Proposition \ref{invofomega123}(ii) since it is not proper.
\end{example}

\begin{remark}
\label{speccaseforom}
Some immediate special cases for which $\varphi$ is a continuous and proper dynamical system are for instance: 

\noindent{\bf (i)} $\varphi$ is a continuous and invertible dynamical system. 
Clearly, $\varphi$ is closed since the maps $\varphi^t$ are homeomorphisms and $\varphi^{-t}(y)$ are singleton sets and thus compact;

\noindent{\bf (ii)} $\varphi$ is a continuous  dynamical system on a compact, Hausdorff  space, which implies that $\varphi$ is closed.
    Moreover, since  singleton sets $\{y\}$ are closed, $\varphi^{-t}(y)$ is closed and thus compact;

\noindent{\bf (iii)} $\varphi$ is a continuous and closed dynamical system on a compact $T_1$ space $X$. On a $T_1$ space points are closed. Therefore, by continuity, $\varphi^{-t}(y)$ is closed and thus compact for all $y$ and for all $t\ge 0$. 
The latter, 
together with the assumption that $\varphi$ is closed, implies that $\varphi$ is proper;

\noindent{\bf (iv)} $\varphi$ is a continuous and proper dynamical system on a compact topological space $(X,\scrT)$. In this case omega-limit sets are compact, closed and invaraint and have the additional property that $U\neq \varnothing$ implies $\omega(U)\neq \varnothing$.
\end{remark}

\section{Alpha-limit sets}
\label{alphaApp}
In backward time we have the asymptotic notion of alpha-limit set. For invertible dynamical systems, flows and homeomorphisms, there is symmetry between past and future  given by time-reversal $t\to -t$, which allows one to define an alpha-limit as the omega-limit set of the time-reversed system.
In noninvertible settings, this symmetry is lost, which has implications not only for the definition of alpha-limit set, but also for the basic properties that alpha-limit sets possess.

 \begin{definition}
 \label{defn:limitset-alpha}
Let $U\subset X$ for which $\bigl(\varphi^{-t}(U)\bigr)_{t\in \T^+}$ is a net of sets in $X$.
Then, the limit set
\begin{equation}
\label{eqn:al-char-101}
\alpha(U)=\bigcap_{t\ge 0}\,\cl \bigcup_{s\ge t}\varphi^{-s}(U),
\end{equation}
is called the 
 \emph{alpha-limit set} $\alpha(U)$ of $U$ with respect to $\varphi$.
An alpha-limit set is closed by definition and contained in
$\cl \rmGamma^-(U)$ with $\rmGamma^-(U):=\bigcup_{t\ge 0} \varphi^{-t}(U)$.
 \end{definition}

As for omega-limit sets, we have that $U\subset V$ implies that
$\alpha(U)\subset \alpha(V)$, and thus $\alpha(U\cap V) \subset \alpha(U)\cap \alpha(V)$.
\begin{lemma} 
\label{additivealpha}
Let $U,V\subset X$. Then,
\begin{equation}
    \label{addom11}
    \alpha(U\cup V) = \alpha(U) \cup \alpha(V).
\end{equation}
\end{lemma}
\begin{proof}
By the same arguments as Lemma \ref{additive}.
    \qed
\end{proof}

If we assume continuity of $\varphi$, then for all $r\ge 0$
\[
\begin{aligned}
\varphi^r(\alpha(U)) &=\varphi^r\Bigl( \bigcap_{t\ge 0}\,\cl \bigcup_{s\ge t}\varphi^{-s}(U)\Bigr) \subset \bigcap_{t\ge 0}
\varphi^r\Bigl( \,\cl \bigcup_{s\ge t}\varphi^{-s}(U)\Bigr)\\
&\subset\bigcap_{t\ge 0}
\cl \varphi^r\Bigl( \bigcup_{s\ge t}\varphi^{-s}(U)\Bigr)
= \bigcap_{t\ge 0}
\cl \bigcup_{s\ge t}\varphi^{r}\bigl(\varphi^{-s}(U)\bigr)\subset\bigcap_{t\ge 0}
\cl \bigcup_{s\ge t}\varphi^{r-s}(U)\\
&\subset\bigcap_{t\ge r}
\cl \bigcup_{s\ge t}\varphi^{r-s}(U)
=\alpha(U),
\end{aligned}
\]
 which proves that $\alpha(U)$ is forward invariant.
Define $\rmGamma^-_\tau(U) := \bigcup_{t\ge \tau} \varphi^{-t}(U)$.

\begin{proposition}[Topological properties]
\label{lem:props-al-om14}
Let $\varphi$ be continuous dynamical system on  $X$, and let
 $U\subset X$. Then,
\begin{description}
\item [(i)]   $\rmGamma^-_\tau(U)$  precompact implies $\alpha(U)$ is compact and closed;
\item [(ii)]  $U\neq \varnothing$, $\rmGamma^-_\tau(U)$  precompact, and $\varphi$ surjective implies that $\alpha(U) \neq \varnothing$;
\item[(iii)] $U$ backward invariant implies $\alpha(U) \subset\cl U$, and $U$  closed and backward invariant implies  $\alpha(U)$ is closed and forward-backward invariant with  $\alpha(U) = \Inv^+\!(U)$;
\item [(iv)] $U$  closed and backward invariant, and
$\varphi$ surjective, implies that $\alpha(U)$ is  closed, strongly invariant, and $\alpha(U) = \Inv(U)$;
\item [(v)] $\rmGamma^-(U)$ closed and $\varphi$ surjective implies $\alpha(U)$ closed, strongly invariant,
and $\alpha(U) = \Inv(\rmGamma^-(U))$;
\item [(vi)] $U$ forward invariant implies $\cl U \subset \alpha(U)$. In particular, when $U$
is forward-backward invariant, then $\cl U = \alpha(U)$;
\item [(vii)] $\alpha\bigl(\alpha(U)\bigr) \supseteq \alpha(U)$;
\item[(viii)] if $\gamma_x^+\subset U$, then $x \in \alpha(U)$. In particular $\Inv^+\!(U) \subset \alpha(U)$, and $\Inv^+\!(U) = \alpha(U)$ whenever $\alpha(U) \subset  U$.
\end{description}
\end{proposition}

\begin{proof}
\noindent{\bf(i)} 
By definition $\alpha(U)$ is closed and contained in
$\cl \rmGamma^-_\tau(U)$. Since the latter is compact also $\alpha(U)$ is compact.

\noindent{\bf(ii)}
If, in addition $U\neq \varnothing$ and $\varphi$ is surjective (i.e.\ $\varphi^t$ is surjective for all $t\ge 0$), then $\rmGamma^-_t(U) \neq \varnothing$. The alpha-limit set is the intersection of nested nonempty, closed subsets $\cl\rmGamma^-_t(U)$, contained in the compact set $\cl\rmGamma^-_\tau(U)$. By  Cantor's intersection theorem the intersection is nonempty.

\noindent{\bf(iii)}
Backward invariance of $U$ implies that 
$\alpha(U) = \bigcap_{t\ge 0} \cl\varphi^{-t}(U)\subset \cl U$.
If in addition $U$ is closed, then, by continuity, $\varphi^{-t}(U)$ is closed for all $t\ge 0$ and thus $\alpha(U)  = \bigcap_{t\ge 0} \varphi^{-t}(U)$.
For $r\ge 0$ we have
\[
\begin{aligned}
\varphi^{-r}(\alpha(U)) &= \varphi^{-r}\Bigl( \bigcap_{t\ge 0} \varphi^{-t}(U)\Bigr) 
= \bigcap_{t\ge 0} \varphi^{-r}\bigl(\varphi^{-t}(U)\bigr)\\
&= \bigcap_{t\ge 0} \varphi^{-t}\bigl(\varphi^{-r}(U)\bigr)
 \subset \bigcap_{t\ge 0} \varphi^{-t}(U) =\alpha(U),
\end{aligned}
\]
since $\varphi^{-r}(U) \subset U$, $r\ge 0$. 
This establishes the forward-backward invariance of $\alpha(U)$.
Since $\alpha(U)$ is forward invariant, we have that $\alpha(U)\subset \Inv^+(U)=:S$. This implies
$S\subset \varphi^{-t}(S)$ for all $t\ge 0,$ and thus
$S\subset \cl\rmGamma_t^-(S)\subset \cl \rmGamma_t^-(U)$.
Consequently, $S\subset \alpha(S)=\bigcap_{t\ge 0}\cl\rmGamma_t^-(S) \subset \bigcap_{t\ge 0}\cl \rmGamma_t^-(U)=\alpha(U)$, which proves that
 $S=\alpha(U)$.

\noindent{\bf(iv)}
By (iii), $\alpha(U)$ is closed and forward-backward invariant and contained in $U$. 
Since $\varphi$ is  surjective, forward-backward invariance implies strong invariance, cf.\ \cite[Lem.\ 2.9]{LSoA1}, and thus $\alpha(U)\subset \Inv(U)$.
Let $S = \Inv(U)$, then $S\subset \varphi^{-t}(S)$ for all $t\ge 0$ and, as before,
\[
S \subset  \bigcap_{t\ge 0}\cl\rmGamma_t^-(S) =\alpha(S)
\subset \alpha(U),
\]
which shows that $\alpha(U)$ is the maximal invariant set in $U$.

\noindent{\bf(v)}
Note that $\alpha(U) = \alpha(\rmGamma^-(U))$. Since $\rmGamma^-(U)$ is backward invariant and  closed, we can apply
(iv) to $\alpha(\rmGamma^-(U))$, which establishes (v).

\noindent{\bf(vi)}
 From forward invariance
it follows that $U\subset \varphi^{-t}(U)$ for all $t\ge 0$.
Therefore, $\cl U\subset \cl \rmGamma_t^-(U)$ and thus $\cl U \subset \bigcap_{t\ge 0}\cl \rmGamma_t^-(U)=\alpha(U)$.
If $U$ is forward-backward invariant, then  (iii) implies that 
 $\cl U \subset \alpha(U) \subset \cl U$, which proves  (vi).

 \noindent{\bf(vii)}
 By definition $\alpha(U)$ is closed and forward invariant. By (vi) this yields 
 $\alpha(U)\subset \alpha(\alpha(U))$.

 \noindent{\bf(viii)}
 Since $x\in \cl \gamma_x^+$ and $\gamma_x^+\subset U$ is forward invariant (vi) implies 
 $x\in \cl\gamma_x^+ \subset \alpha(\gamma_x^+)\subset\alpha(U)$. In particular, $\Inv^+(U)\subset \alpha(U)$. If $\alpha(U)\subset U$, then the forward invariance of $\alpha(U)$ implies that $\alpha(U)\subset \Inv^+(U)$ and thus $\Inv^+(U) = \alpha(U)$.
 \qed
\end{proof}

For backward orbits it is also important to consider limit sets restricted to a single complete orbit. This leads to the notion of {orbital alpha-limit set}.
Given a complete orbit $\gamma_x$ the \emph{orbital alpha-limit set} is defined as 
\begin{equation}
    \label{orbitalalpha1}
    \alphaOg := \bigcap_{t\ge 0} \cl \bigcup_{s\ge t} \gamma_x(-s).
\end{equation}
The latter coincides with $\alpha(x)$ if $\varphi$ is injective.

\begin{proposition}
    \label{orbitalalpha2}
Let $\varphi$ be continuous and proper dynamical system on a  topological space $(X,\scrT)$ and $\gamma_x$ be a complete orbit.
Then, 
\begin{description}
    \item[(i)] $\alphaOg$ is a closed, invariant set;
    \item[(ii)]  $\gamma_x^-$  precompact implies that $\alphaOg\neq \varnothing$.
\end{description}
\end{proposition}
\begin{proof}
\noindent{\bf(i)} By definition $\alphaOg=\bigcap_{t\ge 0} \cl\gamma_x^{-t}$, where $\gamma_x^{-t} = \bigcup_{s\ge t} \gamma_x(-s)$.
The sets $\cl \gamma_x^{-t}$  define a nested family of closed sets, i.e. $\cl\gamma_x^{-t'}\subset \cl\gamma_x^{-t}$ for $t\le t'$.
Since $\varphi^r\colon X\to X$ has compact fibers for all $r\ge 0$,
 Theorem \ref{inverssysofinvsets} yields 
\begin{equation}
    \label{intcomminf12}
\varphi^r\Bigl( \bigcap_{t\ge 0} \cl\gamma_x^{-t} \Bigr)
=\bigcap_{t\ge 0} \varphi^r(\cl \gamma_x^{-t}),\quad \forall r\ge 0.
\end{equation}
By assumption $\varphi^r$ are closed maps for all $r\ge 0$. Therefore, using that the identity $\varphi^r(\gamma_x(-s)) = \gamma_x(r-s)$ for all $s\ge 0$,
\[
\begin{aligned}
\varphi^r(\alphaOg) &= \varphi^r\Bigl( \bigcap_{t\ge 0} \cl\gamma_x^{-t} \Bigr)
=\bigcap_{t\ge 0} \varphi^r(\cl \gamma_x^{-t}) =\bigcap_{t\ge 0}  \varphi^r\Bigl( \cl \bigcup_{s\ge t} \gamma_x(-s)\Bigr)\\
&= \bigcap_{t\ge 0} \cl \varphi^r\Bigl( \bigcup_{s\ge t} \gamma_x(-s)\Bigr) = \bigcap_{t\ge r} \cl  \bigcup_{s\ge t} \gamma_x(r-s) = \bigcap_{t'\ge 0} \cl  \bigcup_{s'\ge t'} \gamma_x(-s')\\
&=\alphaOg,\quad \forall r\ge 0,
\end{aligned}
\]
which proves the invariance of $\alphaOg$.

\noindent{\bf(ii)} As before by the Cantor's intersection theorem 
$\alphaOg$ is nonempty, since $\cl\gamma_x^{-t}\subset \cl\gamma_x^-$ yields a nested family of nonempty, compact, closed subsets.
    \qed
\end{proof}
\end{document}